\documentclass[12pt]{amsart}
\usepackage{amsfonts, amssymb, latexsym, hyperref, xcolor, dynkin-diagrams, tikz,tikz-cd}
\usepackage{ytableau}
\usepackage{enumitem}
\usepackage{subcaption}
\usepackage[utf8]{inputenc}
\usepackage[export]{adjustbox}
\usetikzlibrary{calc, shapes, backgrounds,arrows,positioning,decorations.pathmorphing}

\newtheorem{theorem}{Theorem}[section]
\newtheorem{proposition}[theorem]{Proposition}
\newtheorem{lemma}[theorem]{Lemma}

\newtheorem{claim}[theorem]{Claim}
\newtheorem*{claim*}{Claim}
\newtheorem{corollary}[theorem]{Corollary}
\newtheorem{Main Conjecture}[theorem]{Main Conjecture}

\theoremstyle{definition}
\newtheorem{definition}[theorem]{Definition}
\theoremstyle{remark}

\newtheorem{example}[theorem]{Example}

\theoremstyle{plain}

\newcommand\integers{{\mathbb Z}}

\newcommand{\cellsize}{12}
\newlength{\cellsz} 
\newsavebox{\cell}
\sbox{\cell}{\begin{picture}(\cellsize,\cellsize)
\put(0,0){\line(1,0){\cellsize}}
\put(0,0){\line(0,1){\cellsize}}
\put(\cellsize,0){\line(0,1){\cellsize}}
\put(0,\cellsize){\line(1,0){\cellsize}}
\end{picture}}
\newcommand\cellify[1]{\def\thearg{#1}\def\nothing{}
\ifx\thearg\nothing
\vrule width0pt height\cellsz depth0pt\else
\hbox to 0pt{\usebox{\cell} \hss}\fi
\vbox to \cellsz{
\vss
\hbox to \cellsz{\hss$#1$\hss}
\vss}}
\newcommand\tableau[1]{\vtop{\let\\\cr
\baselineskip -16000pt \lineskiplimit 16000pt \lineskip 0pt
\ialign{&\cellify{##}\cr#1\crcr}}}

\newcommand{\excise}[1]{}

\newcommand{\RomanNumeralCaps}[1] 
    {\MakeUppercase{\romannumeral #1}}

\begin{document}
\pagestyle{plain}
\title{Multiplicity-free skew Schur polynomials}
\author{Shiliang Gao}
\author{Reuven Hodges}
\author{Gidon Orelowitz}
\address{Dept.~of Mathematics, U.~Illinois at Urbana-Champaign, Urbana, IL 61801, USA} 
\email{sgao23@illinois.edu, rhodges@illinois.edu, gidono2@illinois.edu}
\date{\today}

\begin{abstract}
We provide a non-recursive, combinatorial classification of multiplicity-free skew Schur polynomials. These polynomials are $GL_n$, and $SL_n$, characters of the skew Schur modules. Our result extends work of H.~Thomas--A.~Yong, and C.~Gutschwager, in which they classify the
multiplicity-free skew Schur functions.   \end{abstract}

\maketitle

\section{Introduction}\label{sec:intro}

The \emph{skew Schur polynomials} are a fundamental family of symmetric polynomials whose connection to representation theory was first studied by I.~Schur. This family is indexed by \emph{skew partitions} $\lambda / \mu$, where $\lambda = (\lambda_1 \geq \lambda_2 \geq \hdots \geq \lambda_n > 0)$ and $\mu = (\mu_1 \geq \mu_2 \geq \hdots \geq \mu_m > 0)$ are \emph{partitions} with $\mu \subseteq \lambda$ (that is, $m \leq n$ and $\mu_i \leq \lambda_i$ for $1\leq i \leq m$). The skew Schur polynomial of \emph{shape} $\lambda / \mu$ is the generating function
\begin{equation}
\label{eqn:skewpolydef}
s_{\lambda / \mu}(x_1,\ldots,x_n) = \sum_T x^{\eta(T)}, \qquad x^{\eta(T)}:=x_1^{\eta_1(T)}x_2^{\eta_2(T)}\cdots x_n^{\eta_n(T)},
\end{equation}
where the sum is over all semistandard tableaux $T$ of shape $\lambda / \mu$ with entries in $[n]$, and $\eta_i(T)$ is the number of entries in $T$ equal to $i$.

Though not immediately apparent from \eqref{eqn:skewpolydef}, $s_{\lambda / \mu}(x_1,\ldots,x_n)$ is an element of $\Lambda_n$, the ring of symmetric polynomials in variables $x_1,\ldots,x_n$. For $\lambda = (\lambda_1, \lambda_2, \hdots ,\lambda_n )$, the \emph{length} of the partition is $\ell(\lambda) = n$. Defining $s_{\lambda}(x_1,\ldots,x_n) := s_{\lambda / \emptyset}(x_1,\ldots,x_n)$ we recover the \emph{Schur polynomials}; the $s_{\lambda}(x_1,\ldots,x_n)$ with $\ell(\lambda) \leq n$ are a $\mathbb{Z}$-linear basis of $\Lambda_n$. This implies
\begin{equation}
\label{eqn:skewdecomp}
s_{\lambda / \mu}(x_1,\ldots,x_n) = \sum_\nu c_{\mu,\nu}^{\lambda} s_{\nu}(x_1,\ldots,x_n)
\end{equation}
where the sum is over partitions $\nu$ with $\ell(\nu) \leq n$ and the coefficients $c_{\mu,\nu}^{\lambda}$ are the celebrated Littlewood-Richardson coefficients. The expansion \eqref{eqn:skewdecomp} is \emph{multiplicity-free} if $c_{\mu,\nu}^{\lambda} \in \{0,1\}$ for all $\nu$ with $\ell(\nu) \leq n$.  A natural question is then: \[\text{When is $s_{\lambda / \mu}(x_1,\ldots,x_n)$ multiplicity-free?}\] 
Our Theorem \ref{theorem:mainResult} gives a complete, non-recursive answer to this question.

\subsection{Multiplicity-free representation theory} Multiplicity questions, along the same line as the one posed above, have both representation theoretic and geometric import. The Schur polynomials are characters of the irreducible polynomial representations of $GL_n$ (and of $SL_n$, but we focus here on $GL_n$). Expressing a $GL_n$ representation character in the basis of Schur polynomials is equivalent to decomposing the representation into a sum of irreducible representations. 

Knowing that a representation is multiplicity-free has a number of applications (see, \emph{e.g.}, the survey \cite{H95}). Pivoting to a geometric perspective, a group action on a projective algebraic variety induces an action on its homogeneous coordinate ring. The representations that arise from this induced action can often illuminate the orbit structure of the group in the variety \cite{P14,MWZ01,S03,AP14}. In \cite{HL19}, the second author and V.~Lakshmibai classified \emph{spherical Schubert varieties} by showing that certain infinite sets of skew Schur polynomials always contain elements that are not multiplicity-free.  

Developing multiplicity-free criterion for important families of representations has seen considerable interest. In \cite{S01}, the prototypical example, J.~Stembridge classified products of Schur functions that are multiplicity-free. In \emph{ibid.}, he extended this to a classification of the multiplicity-free products of Schur polynomials (equivalently multiplicity-free tensor products of irreducible $GL_n$ representations). The skew Schur functions whose coefficients are all equal to $1$ over an entire dominance order interval, and equal to $0$ otherwise, are characterized in ~\cite{ACM19}. Other examples of multiplicity-free classifications include \cite{TY10,G10,FMS19}. As the skew Schur polynomials are $GL_n$ characters of the skew Schur modules~\cite{FH91}, this paper represents our contribution to this body of work.  

\subsection{Main theorem} Altering definition \eqref{eqn:skewpolydef} so that it becomes a sum over all semistandard tableaux of shape $\lambda / \mu$ yields the \emph{skew Schur function} $s_{\lambda / \mu}$. In \cite{TY10}, H.~Thomas and A.~Yong classified multiplicity-free products of Schubert classes in the cohomology ring of the Grassmannian. As the authors note, this is equivalent to classifying multiplicity-free skew Schur functions (see also \cite{G10}). 

Implicitly, $s_{\lambda / \mu}(x_1,\ldots,x_n)$ is a specialization of $s_{\lambda / \mu}$ that arises by setting $x_m=0$ for $m > n$. In particular, if $s_{\lambda / \mu}$ is multiplicity-free, then so is $s_{\lambda / \mu}(x_1,\ldots,x_n)$, but the converse does not hold. In this sense, our classification is a generalization of \cite{TY10}, and we follow much of their terminology and notation.

 Our result relies on four reductions. The first is central to \cite{TY10}. A partition $\lambda$ is visualized by its \emph{Young diagram}, also denoted $\lambda$; it is a collection of left justified boxes, with $\lambda_i$ boxes in row $i$. Analogously, given a skew partition $\lambda / \mu$, the \emph{skew (Young) diagram} $\lambda / \mu$ equals the Young diagram $\lambda$ with the leftmost $\mu_i$ boxes deleted in each row $i$.

\begin{example} Let $\lambda = (5,4,1,1)$ and $\mu = (2,1,1)$. The Young diagrams $\lambda$, $\mu$ and the skew diagram $\lambda / \mu$ are listed below, left to right.
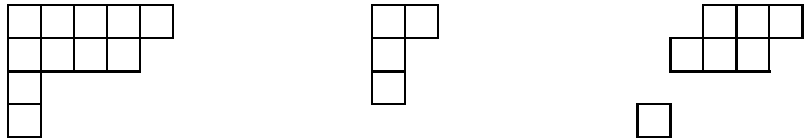
\begin{center}
\ytableausetup{boxsize=1em}
\begin{ytableau}
\; & \; & \; & \; & \; \\
\; & \; & \; & \; \\
\; \\
\; \\
\end{ytableau}
\qquad \qquad \qquad
\begin{ytableau}
\; & \;  \\
\; \\
\; \\
\end{ytableau}
\qquad \qquad \qquad
\begin{ytableau}
\none & \none & \; & \; & \; \\
\none & \; & \; & \; \\
\none \\
\; \\
\end{ytableau}
\end{center}
\end{example}

If there is a box $B$ in row $r$ and column $c$ of the skew diagram $\lambda / \mu$ we will write $B = (r,c) \in \lambda / \mu$. Then ${\sf col}(B)=c$ and ${\sf row}(B)=r$. For any set $S \subset [\lambda_1]$, we refer to column $k$ as an $S$-column if $k\in S$.

\begin{definition}
A skew partition is \emph{basic} if it's skew diagram does not contain any empty rows or columns. The skew partition that arises from the deletion of all empty rows and columns of $\lambda / \mu$ is called the \emph{basic demolition} of the skew partition, and is denoted $(\lambda / \mu)^{{\sf ba}}$.
\end{definition}

Let ${\sf CS}_k(\lambda / \mu)=|\{r : (r,k) \in \lambda / \mu \}|$. Define
\begin{equation}
\rho(\lambda / \mu) := \max\{{\sf CS}_k(\lambda / \mu) : 1 \leq k \leq \lambda_1\}.
\end{equation} 

\begin{definition} \label{def:nsharpandnsharpdemo} A skew partition is $n$-\emph{sharp} if  ${\sf CS}_k(\lambda / \mu) < n$ for all $k$. The $n$-\emph{sharp demolition} of a skew partition is equal to $\emptyset / \emptyset$ if there exists a $k$ such that ${\sf CS}_k(\lambda / \mu) > n$. Otherwise
it equals the skew partition that remains after deleting each column $k$ such that ${\sf CS}_k(\lambda / \mu) = n$. We denote the $n$-sharp demolition by $(\lambda / \mu)^{n\sharp}$.
\end{definition}

When convenient, a partition will be presented as $(\lambda_1^{l_1},\lambda_2^{l_2},\ldots,\lambda_p^{l_p})$, corresponding to the partition with the first $l_1$ entries equal to $\lambda_1$, the next $l_2$ entries equal to $\lambda_2$, and so on. The {\em number of parts} of $\lambda$ is ${\sf np}(\lambda) = p$. If ${\sf np}(\lambda)=1$, the partition is called a \emph{rectangle}. A partition $\lambda$ with ${\sf np}(\lambda)=2$ is called a \emph{fat hook}. 

For a skew partition $\lambda/\mu$ where $\lambda = (\lambda_1^{l_1},\lambda_2^{l_2},\ldots,\lambda_p^{l_p})$ and $\mu = (\mu_1^{k_1},\mu_2^{k_2},\ldots,\mu_q^{k_q})$, define $\tau(\lambda / \mu) = \sigma(\lambda/\mu) = 0$ if $\mu = \emptyset$ or $\lambda$ is a rectangle.
Otherwise, define
\begin{equation}
\tau(\lambda / \mu) = l_1 - \min(l_1,\ell(\mu))\text{ and }
\sigma(\lambda/\mu) = \lambda_p - \min(\mu_1, \lambda_p).
\end{equation}

\begin{definition} \label{def:tightandtightdemo} A basic skew partition $\lambda/\mu$ is \emph{tight} if $\tau(\lambda / \mu) = \sigma(\lambda/\mu) = 0$. The \emph{tight demolition} of  $\lambda / \mu$ is \[(\lambda / \mu)^{{\sf ti}} = (((\lambda_1-\sigma(\lambda/\mu))^{l_1 - \tau(\lambda / \mu)},(\lambda_2-\sigma(\lambda/\mu))^{l_2},\ldots,(\lambda_p-\sigma(\lambda/\mu))^{l_p}) / \mu)^{\sf ba}. \] 
Note that if $\tau(\lambda / \mu) = \sigma(\lambda/\mu) = 0$, then $(\lambda / \mu)^{{\sf ti}} = (\lambda / \mu)^{\sf{ba}}$.
Visually, if $\mu \neq \emptyset$ and $\lambda$ is not a rectangle, then $(\lambda / \mu)^{{\sf ti}}$ is the deletion of all rows with $\lambda_1$ boxes and all columns with $\ell(\lambda)$ boxes from $\lambda / \mu$ followed by the deletion of all empty rows and columns.
\end{definition}

If $\lambda \subseteq (a^b)$, the $(a^b)$-\emph{complement} of $\lambda$ is
\begin{equation}
\lambda^{\vee(a^b)} = (a - \lambda_b, a - \lambda_{b-1},\ldots,a - \lambda_1),
\end{equation}
where trailing zeros are removed, and for simplicity of notation, we set $\lambda_i = 0$ if $i > \ell(\lambda)$. Visually, $\lambda^{\vee(a^b)}$ is the complement of $\lambda$ inside $(a^b)$, rotated by 180 degrees. There is a unique shortest lattice path, from the southwest corner to the northeast corner of $(a^b)$, separating $\lambda$ and its complement. A \emph{segment} of this lattice path is a maximal consecutive sequence of north, or east, steps. The $(a^b)$-\emph{shortness} of a partition $\lambda$, denoted $\sf{short}_{(a^b)}(\lambda)$, is the length of the shortest segment in the associated lattice path in $(a^b)$. Once a rectangle is fixed, we omit the respective $(a^b)$ from the notation.

\begin{example} Fix the rectangle $(5^4)$. If $\lambda = (5,4,1,1) \subseteq (5^4)$, then $\lambda^{\vee} = (4,4,1)$. The shortness of $\lambda$ equals $1$ as seen by inspecting the lattice path below.
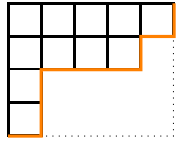
\begin{center}
\begin{tikzpicture}[inner sep=0in,outer sep=0in]
\node (n) {
\begin{ytableau}
\; & \; & \; & \; & \; \\
\; & \; & \; & \;  \\
\;  \\
\;  \\
\end{ytableau}};
\draw[black, dotted] (n.south west)--++(2.2,0)--++(0,1.75);
\draw[very thick,orange] (n.south west)--++(.365*1.21,0)--++(0,.73*1.21)--++(1.095*1.21,0)--++(0,.365*1.21)--++(.365*1.21,0)--++(0, .365*1.21);
\end{tikzpicture}
\end{center}
\end{example}

\begin{definition} A basic skew partition $\lambda / \mu$ is \emph{ordinary} if ${\sf np}(\lambda)-1 \leq {\sf np}(\mu)$ or $\mu = \emptyset$. The \emph{ordinary reduction} of a skew partition is equal to $\mu^{\vee(\lambda_1^{\ell(\lambda)})} / \lambda^{\vee(\lambda_1^{\ell(\lambda)})}$ if ${\sf np}(\lambda)-1 > {\sf np}(\mu)$ and $\mu\neq \emptyset$, otherwise it equals $\lambda / \mu$. We denote the ordinary reduction by $(\lambda / \mu)^{\sf or}$. 
\end{definition}

\begin{definition}
A skew Schur polynomial $s_{\lambda / \mu}(x_1,\ldots,x_n)$ is basic (resp. $n$-sharp, tight, ordinary) if $\lambda / \mu$ is basic (resp. $n$-sharp, tight, ordinary). 
\end{definition}

\begin{theorem} 
\label{theorem:mainReduction}
Let $s_{((((\lambda / \mu)^{n\sharp})^{{\sf ba}})^{{\sf ti}})^{{\sf or}}}(x_1,\ldots,x_{n'})$ with $n' = n - \tau(((\lambda / \mu)^{n\sharp})^{{\sf ba}})$ be the $n$-sharp, basic, tight, ordinary demolition of $s_{\lambda / \mu}(x_1,\ldots,x_n)$. Then $s_{((((\lambda / \mu)^{n\sharp})^{{\sf ba}})^{{\sf ti}})^{{\sf or}}}(x_1,\ldots,x_{n'})$ is basic, $n$-sharp, tight, and ordinary. Further, $s_{\lambda / \mu}(x_1,\ldots,x_n)$ is multiplicity-free if and only if $s_{((((\lambda / \mu)^{n\sharp})^{{\sf ba}})^{{\sf ti}})^{{\sf or}}}(x_1,\ldots,x_{n'})$ is multiplicity-free.
\end{theorem}

\begin{definition}
\label{def:r1andr2}
Let $\lambda/\mu$ be a basic, tight, ordinary skew partition, and say $\lambda = (\lambda_1^{l_1},\lambda_2^{l_2},\ldots,\lambda_p^{l_p})$ and $\mu = (\mu_1^{k_1},\mu_2^{k_2},\ldots,\mu_q^{k_q})$.  Fix the rectangle $(\lambda_1^{\ell(\lambda)})$.  Denote
\begin{equation}
r_1(\lambda/\mu) :=  \begin{cases}
0 & \sf{np}(\lambda)>2, \sf{np}(\mu)> 1\\
1 & \sf{np}(\lambda)=2, \sf{np}(\mu)>2, \sf{short}(\lambda)\geq 2\\
2 & \sf{np}(\lambda)=2, \sf{np}(\mu)=2, \sf{short}(\lambda)\geq 3, \sf{short}(\mu)\geq 2\\
\infty & \text{Otherwise}
\end{cases}
\end{equation}
and
\begin{equation}
r_2(\lambda/\mu) := \begin{cases}
1 & \lambda_2 = \mu_q, l_2 \geq \ell(\mu)\\
1 & \lambda_2 = \mu_1, k_1 \geq l_2\\
0 & \text{Otherwise}
\end{cases}.
\end{equation}
\end{definition}

Now we may state our main result.

\begin{theorem}
\label{theorem:mainResult}
Let $s_{\lambda / \mu}(x_1,\ldots,x_n)$ be a basic, $n$-sharp, tight, ordinary skew Schur polynomial.  Then $s_{\lambda / \mu}(x_1,\ldots,x_n)$ is multiplicity-free if and only if
\begin{equation}
\rho(\lambda/\mu) < n < \rho(\lambda/\mu) + r_1(\lambda/\mu)+ r_2(\lambda/\mu)
\end{equation}
\end{theorem}


\section{The reductions}\label{sec:reductions}

\subsection{The Littlewood-Richardson coefficients, rules and identities}\label{subsec:LRrule}

There are myriad rules for computing the Littlewood-Richardson coefficients, one of them being the Littlewood-Richardson rule.

A \emph{filling} of shape $\lambda / \mu$ is an assignment of values from $[n]$ to each box in the Young diagram $\lambda / \mu$. A filling is called a \emph{tableau} if the entries in each column strictly increase top to bottom. A tableau where the entries in each row weakly increase left to right is \emph{semistandard tableau}. The \emph{content} of a filling $T$ is $\eta(T) = (\eta_1(T), \ldots, \eta_n(T))$, where $\eta_i(T)$ is equal to the number of entries in $T$ equal to $i$.  For a filling $T$ of $\lambda / \mu$, and $(r,c) \in \lambda / \mu$, $T(r,c)=a$ indicates that the box in row $r$ and column $c$ of $T$ contains $a$.

The \emph{reverse reading word} of a filling is the sequence obtained by concatenating the entries of each row from right to left, top to bottom. A word $a_1 a_2 \ldots a_t$ is \emph{ballot} if in every initial factor $a_1 a_2 \ldots a_s$ the number of $i$'s is greater than or equal to the number of $i+1$'s, for all $i$. A \emph{ballot tableau} is a semistandard tableau whose reverse reading word is ballot.

\begin{theorem}[the Littlewood-Richardson rule~\cite{S99}] 
\label{theorem:lrrule}
The Littlewood-Richardson coefficient $c_{\mu, \nu}^{\lambda}$ is equal to the number of ballot tableaux of shape $\lambda / \mu$ and content $\nu$.
\end{theorem}

The column reading word of a filling is the sequence obtained by concatenating the entries of each column of the filling, top to bottom, right to left. The following lemma is a well known result whose proof we leave to the reader.

\begin{lemma}
\label{lemma:readingwordequiv}
Let $T$ be a semistandard tableau. The reverse reading word of $T$ is ballot if and only if the column reading word of $T$ is ballot.
\end{lemma}

The Littlewood-Richardson coefficients satisfy the following two symmetries:
\begin{equation}
\label{eq:LRsymmetrybasic}
c_{\mu, \nu}^{\lambda} = c_{\nu, \mu}^{\lambda}\,\, ,
\end{equation}
\begin{equation}
\label{eq:LRsymmetry-transpose}
c_{\mu, \nu}^{\lambda} = c_{\mu, \nu^{\vee(a^b)}}^{\lambda^{\vee(a^b)}}\,\, ,
\end{equation}
for $\lambda \subseteq (a^b)$. Applying the two above identities yields
\begin{equation}
\label{eq:LRsymmetry-rotate}
c^\lambda_{\mu,\nu} = c^\lambda_{\nu,\mu} = c^{\mu^{\vee(a^b)}}_{\nu, \lambda^{\vee(a^b)}} = c^{\mu^{\vee(a^b)}}_{\lambda^{\vee(a^b)}, \nu}\, \,.
\end{equation}

Given a sequence $\nu = (\nu_1,\ldots,\nu_z)\in \mathbb{N}_{\geq 0}^z$, define ${\sf sort}(\nu)$ to be the partition that arises from sorting the entries of $\nu$ so that they are weakly decreasing. If $\lambda = (\lambda_1,\ldots,\lambda_a)$ and $\mu = (\mu_1,\ldots,\mu_b)$ are two partitions, then $\lambda \cup \mu = {\sf sort}(\lambda_1,\ldots,\lambda_a,\mu_1,\ldots,\mu_b)$. Set $\lambda + (c^d) = (\lambda_1 + c, \lambda_2 + c,\ldots,\lambda_d + c, \lambda_{d+1}, \ldots)$ where $\lambda_i = 0$ for $i > a$.

\begin{theorem}[{~\cite[Theorem 3.1]{G10}}] \label{theorem:gutReduction} Let $\lambda, \mu, \nu$ be partitions with $a,b \in \mathbb{N}_{\geq 0}$ and $a \geq b$. Then 
\begin{equation}
\label{eq:LRsymmetryswap}
c_{\mu,\nu}^{\lambda} \leq c_{\mu + (1^b),\nu + (1^{a-b})}^{\lambda + (1^a)}\, \, ,
\end{equation}
\begin{equation}
c_{\mu,\nu}^{\lambda} \leq c_{\mu \cup (b),\nu \cup (a-b)}^{\lambda \cup (a)}\, \,.
\end{equation}

\end{theorem}

\begin{corollary}\label{Cor:GUT}
Let $\lambda/\mu$ and $\lambda^{*}/\mu^{*}$ be skew partitions such that $\lambda^{*}/\mu^{*}$ is obtained from $\lambda/\mu$ by removing a subset of columns. If $\rho(\lambda/\mu) = \rho(\lambda^{*}/\mu^{*})$ then for all $n\in \integers$, $s_{\lambda^{*}/\mu^{*}}(x_1,\ldots,x_n)$ is not multiplicity-free implies $s_{\lambda/\mu}(x_1,\ldots,x_n)$ is not multiplicity-free.
\end{corollary}
\begin{proof} There is a sequence of skew diagrams $\lambda^{*}/\mu^{*}=\lambda^1/\mu^1,\ldots,\lambda^z/\mu^z=\lambda/\mu$ with $\lambda_i / \mu_i = \lambda_{i-1} + (1^{a_i}) / \mu_{i-1} + (1^{b_i})$ and $b_i < a_i$, for $1<i\leq z$. By construction
\begin{equation}
\label{eq:addlength}
a_i - b_i \leq \rho(\lambda/\mu).
\end{equation}
If $s_{\lambda^{*}/\mu^{*}}(x_1,\ldots,x_n)$ is not multiplicity-free, then $n \geq \rho(\lambda/\mu)$ and there is a partition $\nu^{*}$ with $\ell(\nu^{*}) \leq n$ such that $c^{\lambda^{*}}_{\mu^{*},\nu^{*}} > 1$. By (\ref{eq:LRsymmetryswap}), $c^{\lambda^2}_{\mu^2,\nu^{*} + (1^{a_2 - b_2}) } \geq c^{\lambda^{*}}_{\mu^{*},\nu^{*}} > 1$. Then \eqref{eq:addlength} implies $\ell(\nu^{*} + (1^{a_i - b_i})) \leq n$. Thus $s_{\lambda^2/\mu^2}(x_1,\ldots,x_n)$ is not multiplicity-free. Our desired result follows by inductively repeating this argument.
\end{proof}
.
\begin{corollary}\label{Cor:GUT2}
Let $\lambda/\mu$ and $\lambda^{*}/\mu^{*}$ be skew partitions such that $\lambda^{*}/\mu^{*}$ is obtained from $\lambda/\mu$ by removing a subset of columns. If $s_{\lambda^{*}/\mu^{*}}(x_1,\ldots,x_{\rho(\lambda^{*}/\mu^{*}) + 1})$ is not multiplicity-free, then $s_{\lambda/\mu}(x_1,\ldots,x_{\rho(\lambda/\mu)+1})$ is not multiplicity-free.
\end{corollary}
\begin{proof}
This follows by a nearly identical argument to the proof of Corollary~\ref{Cor:GUT}, taking care at each step to show that the final content $\nu$ will have $\ell(\nu) \leq \rho(\lambda/\mu) + 1$.
\end{proof}

\begin{lemma}\label{Lemma:switchingcol}
    Let $\lambda/\mu$ be a basic skew partition such that ${\sf{np}}(\lambda),{\sf{np}}(\mu)\geq 2$. Suppose that $\lambda_2 = \mu_q$ and $\ell(\mu) = l_1$. Set $\tilde\lambda = (\tilde\lambda_1^{\tilde l_1}, \ldots, \tilde\lambda_{p+q-1}^{\tilde l_{p+q-1}})$ and $\tilde\mu = (\mu_q^{l_1})$ where 
    \[\tilde\lambda_i^{\tilde l_i} = 
    \begin{cases}
        (\lambda_1+\mu_q-\mu_{q-i+1})^{k_{q-i+1}} & \text{if }i\in [q]\\
        \lambda_{i-q+1}^{l_{i-q+1}} & \text{else}
    \end{cases},
    \] then
    \[c^{\lambda}_{\mu,\nu} = c^{\tilde\lambda}_{\tilde\mu,\nu} \text{ for all partition $\nu$}.\]
    
\end{lemma}

\begin{example} Let $\lambda = (6^3,2^2,1)$ and $\mu = (5,3,2)$. By Lemma~\ref{Lemma:switchingcol}, we obtain $\tilde\lambda = (6,5,3,2^2,1)$, $\tilde\mu = (2^3)$ and $c^{\lambda}_{\mu,\nu} = c^{\tilde\lambda}_{\tilde\mu,\nu} $ for all partition $\nu$. The skew diagram $\lambda / \mu$ (on the left) and $\tilde\lambda/\tilde\mu$ (on the right) are listed below.
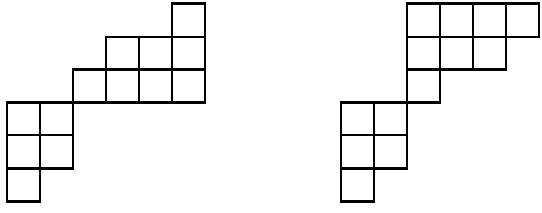
\begin{center}
\ytableausetup{boxsize=1em}
\begin{ytableau}
\none & \none & \none & \none & \none & \; \\
\none & \none & \none & \; & \; & \; \\
\none & \none & \; & \; & \; & \; \\
\; & \; \\
\; & \; \\
\; \\
\end{ytableau}
\qquad \qquad
\begin{ytableau}
\none & \none & \; & \; & \; & \; \\
\none & \none & \; & \; & \; \\
\none & \none & \; \\
\; & \; \\
\; & \; \\
\; \\
\end{ytableau}
\end{center}
Visually, we obtain $\tilde\lambda/\tilde\mu$ from $\lambda/\mu$ by reversing the left to right order of column $3$ through $6$ while maintaining a valid skew partition.
\end{example}

\noindent \emph{Proof of Lemma~\ref{Lemma:switchingcol}:} Since $\lambda_2 = \mu_q$ and $\ell(\mu) = l_1$, all ballot tableaux of shape $\lambda/\mu$, of any content, are identical in columns $\mu_q+1$ through $\lambda_1$. Notice that $\tilde\lambda_{q+1} = \lambda_2$ and $\ell(\mu) = \sum_{i=1}^{q}k_i = \sum_{j=i}^{q}\tilde l_j$. Therefore all ballot tableaux of shape $\tilde\lambda/\tilde\mu$, of any content, are also identical in columns $\mu_q+1$ through $\lambda_1$. Let $T$ be any ballot tableau of shape $\lambda/\mu$ and of any content, then column $c$ of $T$ is filled by $1$ through $\sf{CS}_c(\lambda/\mu)$ for $c\in [\mu_q+1,\lambda_1]$. Similarly, let $\tilde T$ be any ballot tableau of shape $\tilde\lambda/\tilde\mu$ and of any content, then column $c$ of $\tilde T$ is filled by $1$ through $\sf{CS}_c(\tilde\lambda/\tilde\mu)$ for $c\in[\mu_q+1,\lambda_1]$. Since ${\sf{CS}}_c(\lambda/\mu) = {\sf{CS}}_{\lambda_1+\mu_q+1-c}(\tilde\lambda/\tilde\mu)$ for all $c\in[\mu_q+1,\lambda_1]$, the content of $T$ and $\tilde T$ are identical in columns $\mu_q+1$ through $\lambda_1$. Therefore we can find a content preserving involution from the set of ballot tableaux of shape $\lambda/\mu$ to the set of ballot tableaux of shape $\tilde\lambda/\tilde\mu$. As a result,
    \[c^{\lambda}_{\mu,\nu} = c^{\tilde\lambda}_{\tilde\mu,\nu} \text{ for all partition $\nu$}.\] \qed

\subsection{A quartet of reductions}\label{subsec:treductions}

\begin{proposition}[{\cite[Lemma 2]{TY10}}]
\label{prop:reductionbasic}
If $\lambda / \mu$ is a skew partition, then \[s_{\lambda / \mu}(x_1,\ldots,x_n) = s_{(\lambda / \mu)^{\sf ba}}(x_1,\ldots,x_n).\]
\end{proposition}

\begin{proposition}
\label{prop:reductionnsharp}
If $\lambda / \mu$ is a skew partition, then 
\begin{equation}
\label{eq:nsharp}
s_{\lambda / \mu}(x_1,\ldots,x_n) = (x_1 \cdots x_n)^k s_{(\lambda / \mu)^{\sf n\sharp}}(x_1,\ldots,x_n),
\end{equation}
where $k = \# \{ d : {\sf CS}_d(\lambda / \mu) = n \}$. In particular, $s_{\lambda / \mu}(x_1,\ldots,x_n)$ is multiplicity-free if and only if $s_{(\lambda / \mu)^{\sf n\sharp}}(x_1,\ldots,x_n)$ is multiplicity-free.
\end{proposition}
\begin{proof}
If $\rho(\lambda / \mu) > n$, then $(\lambda / \mu)^{\sf n\sharp} = \emptyset / \emptyset$ and $s_{(\lambda / \mu)^{\sf n\sharp}}(x_1,\ldots,x_n) = 0$. Let $\nu$ be a partition with $\ell(\nu) \leq n$. Since there is a column in $\lambda / \mu$ with more than $n$ boxes, there are no tableau of shape $\lambda / \mu$ and content $\nu$. Hence, Theorem~\ref{theorem:lrrule} implies $c_{\mu, \nu}^{\lambda} = 0$. By \eqref{eqn:skewpolydef}, $s_{\lambda / \mu}(x_1,\ldots,x_n) =  0$. Thus \eqref{eq:nsharp} holds when $\rho(\lambda / \mu) > n$.

If $\rho(\lambda / \mu) \leq n$, then $(\lambda / \mu)^{\sf n\sharp}$ is the skew partition that arises from deleting all columns with exactly $n$ boxes. If $k = 0$, then $(\lambda / \mu)^{\sf n\sharp}=\lambda / \mu$ and \eqref{eq:nsharp} is trivial. Hence, we assume $k > 0$. Let $\lambda^{n\sharp}$ and $\mu^{n\sharp}$ be the partitions such that $(\lambda / \mu)^{\sf n\sharp} = \lambda^{n\sharp} / \mu^{n\sharp}$. If $\nu$ is a partition such that $\ell(\nu) \leq n$, then we claim
\begin{equation}
\label{eq:nsharpLR}
c_{\mu^{n\sharp}, \nu}^{\lambda^{n\sharp}} =  c_{\mu, \nu + (k^n)}^{\lambda}.
\end{equation}

Denote by $S_1$ the set of ballot tableaux of shape $\lambda/\mu$ with content $\nu + (k^n)$, and let $S_2$ denote the set of ballot tableaux of shape $(\lambda / \mu)^{\sf n\sharp}$ with content $\nu$. Consider the map $f:S_1\xrightarrow{}S_2$ that sends a $T \in S_1$ to a $T^* \in S_2$ by removing all columns in $T$ of length $n$. We show that $f$ is a bijection.

\emph{$f$ is well-defined:} Since each column of $T$ is strictly increasing downwards, a column with $n$ boxes has to be filled with $1$ through $n$. Removing the subword corresponding to the boxes in such a column will result in a column reading word that is ballot. Thus by Lemma~\ref{lemma:readingwordequiv} the reverse reading word is ballot. Each column remains strictly increasing and each row remains weakly increasing after removing columns of length $n$. Hence $f(T)$ is a ballot tableau of shape $(\lambda / \mu)^{\sf n\sharp}$ and content $\nu$.

\emph{$f$ is injective:} Since all the columns of length $n$ must be filled with $1$ through $n$, all ballot tableaux in $S_1$ are identical in these columns. Therefore if $T_1,T_2 \in S_1$ with $T_1 \neq T_2$, then they must differ in a column of length less than $n$. As a result, their image, $f(T_1)$ and $f(T_2)$, differ in a column of length less than $n$.

\emph{$f$ is surjective:} Let $\{c^*_1,\ldots,c^*_k\}$ be the columns in $\lambda/\mu$ of length less than $n$ and let $\{c_1,\ldots,c_l\}$ be the columns of length equal to $n$. For $T^* \in S_2$, consider a tableau $T_0$ of shape $\lambda/\mu$ where we fill column $c^*_i$ with the same entries as the $i$\textsuperscript{th} column in $T^*$, for all $i\in \{1,\ldots,k\}$. By construction the two columns have the same size. Fill in the remaining columns of $T_0$ with $1$ through $n$ in each column. It is clear that if $T_0\in S_1$, then $f(T_0) = T^*$. 

\emph{$T_0$ is semistandard:} It is clear from the construction that each column of $T_0$ is strictly increasing. Since $T^*$ is semistandard, it suffices to verify that the entries in each row are weakly increasing at the boxes in $c_i$ for all $i\in \{1,\ldots,l\}$. Since $\lambda/\mu$ is a skew-shape, the column directly to the left of $c_i$, denoted as $c_i^-$, ends in the same row or lower than $c_i$ ends. Therefore if an entry in $c_i^-$ is greater than its right neighbor in $c_i$, then every entries in $c_i^-$ below it have to be greater than their right neighbors (if exist). As a result, the entry in $c_i^-$ that has the bottom entry of $c_i$ as its right neighbor has to filled by at least $n+1$ which is impossible by our construction. Similarly, the column directly to the right of $c_i$, denoted as $c_i^+$ starts in the same row or higher than $c_i$. Therefore if an entry in $c_i^+$ is smaller than its left neighbor in $c_i$, then every entry above it in $c_i^+$ is smaller than than their left neighbors, if they exist. However, there will not be a valid entry to the right of the top entry in $c_i$. Therefore $T_0$ is semistandard.

\emph{$T_0$ is a ballot tableau:} By adding columns $\{c_1,\ldots,c_l\}$ filled by $1$ through $n$, one insert $l$ subwords into the column reading word of $T^*$ where each subword is $1,2,\ldots,n$. Therefore the column reading word of $T_0$ is also ballot, and by Lemma~\ref{lemma:readingwordequiv} the reverse reading word is ballot.

Thus \eqref{eq:nsharpLR} holds. The final equality we need to complete the proof is
\begin{equation}
\label{eq:schurpolyfulllenrows}
s_{\nu + (k^n)}(x_1,\ldots,x_n) = (x_1 \cdots x_n)^k s_{\nu}(x_1,\ldots,x_n),
\end{equation}
which follows from \eqref{eqn:skewpolydef}. Thus
\begin{equation}
\begin{array}{rlr}
\label{eq:nsharpdemomain}
s_{\lambda / \mu}(x_1,\ldots,x_n) & = \displaystyle \sum_{\alpha \text{ s.t. }\ell(\alpha) \leq n} c_{\mu, \alpha}^{\lambda} s_{\alpha}(x_1,\ldots,x_n) & \eqref{eqn:skewpolydef}\\
& = \displaystyle \sum_{\nu\text{ s.t. }\ell(\nu) \leq n} c_{\mu, \nu + (k^n)}^{\lambda} s_{\nu + (k^n)}(x_1,\ldots,x_n) & \\
& = \displaystyle \sum_{\nu\text{ s.t. }\ell(\nu) \leq n} c_{\mu^{n\sharp}, \nu}^{\lambda^{n\sharp}} (x_1,\ldots,x_n)^k s_{\nu}(x_1,\ldots,x_n) & \eqref{eq:nsharpLR}\text{, }\eqref{eq:schurpolyfulllenrows}\\
& = \displaystyle (x_1 \cdots x_n)^k s_{(\lambda / \mu)^{\sf n\sharp}}(x_1,\ldots,x_n) & \eqref{eqn:skewpolydef}\\
\end{array}
\end{equation}
The second equality in \eqref{eq:nsharpdemomain} follows from the fact that any nonzero $c_{\mu, \alpha}^{\lambda}$ must have $(k^n) \subseteq \alpha$. If $s_{\lambda / \mu}(x_1,\ldots,x_n)$ is multiplicity-free then all the Littlewood-Richardson coefficients that appear in \eqref{eq:nsharpdemomain} are equal to $0$ or $1$, and hence $s_{(\lambda / \mu)^{\sf n\sharp}}(x_1,\ldots,x_n)$ is multiplicity-free. The converse holds by the same argument.
\end{proof}
\begin{lemma}
\label{lemma:rectanglecontent}
If $\lambda / \mu$ is a skew partition that contains a $b \times k$ rectangle and $c_{\mu, \nu}^{\lambda} \neq 0$, then $(k^b) \subseteq \nu$.
\end{lemma}
\begin{proof}
If $c_{\mu, \nu}^{\lambda} \neq 0$, then there must exist a ballot tableau $T$ of shape $\lambda / \mu$ and content $\nu$. Let $(r,c) \in \lambda / \mu$ be the bottom, left corner of the $b \times k$ rectangle. Then $T(r,c) \geq b$ since $T$ is a tableau, and $T(r,c+i) \geq b$ for $0 \leq i \leq k$ since $T$ is a semistandard tableau. Thus the reverse reading word $a_1,\ldots,a_m$ of $T$ contains a consecutive subsequence $a_{z+1},\ldots,a_{z+k}$ with 
\begin{equation}
\label{eq:bottomofrectanglevals}
a_{z+j} \geq b\text{ for }1 \leq j \leq k,
\end{equation}
\begin{equation}
\label{eq:bottomofrectanglesemistandard}
a_{z+j} \geq a_{z+j+1}\text{ for }1 \leq j < k.
\end{equation}

Since $a_1,\ldots,a_m$ is a ballot sequence, we must have that there is a subsequence of $a_1,\ldots,a_z$ that equals $1,2,\ldots,a_{z+1}-1$. By \eqref{eq:bottomofrectanglesemistandard}, $a_{z+1} \geq a_{z+2}$, so there must be a second subsequence of $a_1,\ldots,a_z$ that equals $1,2,\ldots,a_{z+2}-1$, and it must be disjoint from the first sequence. Continuing inductively we arrive at the conclusion that there are $k$ disjoint subsequences of $a_{1},\ldots,a_{z+k}$ of the form $1,2,\ldots,a_{z+j}$ for $1 \leq j \leq k$. Each $a_{z+j} \geq b$ by \eqref{eq:bottomofrectanglevals}, and thus we arrive at our desired result that $(k^b) \subseteq \nu$.
 \end{proof}

\begin{proposition}
\label{prop:reductiontight}
If $\lambda / \mu$ is a basic skew partition and $n' = n - \tau(\lambda / \mu)$, then $s_{\lambda / \mu}(x_1,\ldots,x_n)$ is multiplicity-free if and only if $s_{(\lambda / \mu)^{\sf ti}}(x_1,\ldots,x_n')$ is multiplicity-free.
\end{proposition}
\begin{proof}
Let $\tau = \tau(\lambda / \mu)$ and $\sigma = \sigma(\lambda/\mu)$. If $\tau = \sigma = 0$, the result is immediate, so we assume $\tau > 0$ or $\sigma > 0$. Let $\lambda^{\sf ti}$ and $\mu^{\sf ti}$ be the partitions such that $(\lambda / \mu)^{\sf ti} = \lambda^{\sf ti} / \mu^{\sf ti}$. By Definition~\ref{def:tightandtightdemo}, 
\begin{equation}
\label{eq:tightform}
\lambda = \lambda^{\sf ti} \cup ((\lambda_1-\sigma)^\tau) + (\sigma^{\ell(\lambda)}) \text{ and }\mu = \mu^{\sf ti}. 
\end{equation}
Suppose that $c_{\mu,\nu}^{\lambda} \neq 0$ for some $\nu$ such that $\ell(\nu) \leq n$. Since $\lambda / \mu$ contains a $\tau \times \lambda_1$ rectangle and a $\ell(\lambda)\times \sigma$ rectangle, Lemma~\ref{lemma:rectanglecontent} implies $(\lambda_1^\tau) \subseteq \nu$ and $(\sigma^{\ell(\lambda)})\subseteq \nu$. Therefore both $\lambda$ and $\nu$ contains the fat hook $(\lambda_1^{\tau}, \sigma^{\ell(\lambda)-\tau})$. Further, $\nu \subseteq \lambda$, and so $\nu_i = \lambda_1\text{ for }i \in \{1,\ldots,\tau \}$ and $\nu_{\ell(\lambda)} = \lambda_p$. This implies
\begin{equation}
\label{eq:tightrect}
\nu = \alpha \cup ((\lambda_1-\sigma)^\tau) + (\sigma^{\ell(\lambda)}),
\end{equation}
for some partition $\alpha$ with $\ell(\alpha) \leq n'$ and $\alpha_1 \leq \lambda_1$.
\begin{equation}
\begin{array}{rlr}
\label{eq:tightmain}
c_{\mu,\nu}^{\lambda} & = c_{\mu^{\sf ti},\alpha \cup ((\lambda_1-\sigma)^\tau) + (\sigma^{\ell(\lambda)})}^{\lambda^{\sf ti} \cup ((\lambda_1-\sigma)^\tau) + (\sigma^{\ell(\lambda)})} \qquad\qquad\qquad &\eqref{eq:tightform}\text{, } \eqref{eq:tightrect}\\
& = c_{\alpha \cup ((\lambda_1-\sigma)^\tau) + (\sigma^{\ell(\lambda)}),\mu^{\sf ti}}^{\lambda^{\sf ti} \cup ((\lambda_1-\sigma)^\tau) + (\sigma^{\ell(\lambda)})} & \eqref{eq:LRsymmetrybasic}\\
& = c_{\alpha,\mu^{\sf ti}}^{\lambda^{\sf ti}} & \text{Proposition }~\ref{prop:reductionbasic}\text{, }\alpha_1 \leq \lambda_1-\sigma\text{ and }\ell(\alpha)\leq \ell(\lambda)-\tau \\
& = c_{\mu^{\sf ti}, \alpha}^{\lambda^{\sf ti}} & \eqref{eq:LRsymmetrybasic}\\
\end{array}
\end{equation}
Notice that the above equality holds for all $\nu$ such that $\ell(\nu) \leq n$ and $c_{\mu,\nu}^{\lambda} \neq 0$, and all $\alpha$ such that $\ell(\alpha) \leq n'$ and $c_{\mu^{\sf ti}, \alpha}^{\lambda^{\sf ti}} \neq 0$. Therefore if $s_{(\lambda/\mu)^{\sf ti}}(x_1,\ldots,x_{n'})$ is not multiplicity-free, by (\ref{eq:tightmain}), we can find $\nu$ such that $c^{\lambda}_{\mu,\nu}> 1$ as well. By \eqref{eqn:skewdecomp} we are done.
\end{proof}
\begin{proposition}
\label{prop:reductionordinary}
If $\lambda / \mu$ is a skew partition, then \[s_{\lambda / \mu}(x_1,\ldots,x_n) = s_{(\lambda / \mu)^{\sf or}}(x_1,\ldots,x_n).\]
\end{proposition}
\begin{proof}
If $\lambda / \mu$ is ordinary, then $(\lambda / \mu)^{\sf or}=\lambda / \mu$ and the proof is trivial. If $\lambda / \mu$ is not ordinary, then $(\lambda / \mu)^{\sf or}= \mu^{\vee((\lambda_1)^{\ell(\lambda)})} / \lambda^{\vee((\lambda_1)^{\ell(\lambda)})}$. Our result now follows from \eqref{eq:LRsymmetry-rotate}.
\end{proof}

\noindent \textit{Proof of Theorem~\ref{theorem:mainReduction}:} 
We first show that $\lambda^{\sf red} / \mu^{\sf red} := ((((\lambda / \mu)^{n\sharp})^{{\sf ba}})^{{\sf ti}})^{{\sf or}}$ is $n$-sharp, basic, tight, and ordinary. It follows from Definition~\ref{def:nsharpandnsharpdemo} that $(\lambda / \mu)^{n\sharp}$ is $n$-sharp. The basic demolition of $(\lambda / \mu)^{n\sharp}$ does not increase the number of boxes in any column, and hence is $n$-sharp and, by construction, basic.

Let $\alpha / \beta = (\alpha_1^{c_1},\ldots,\alpha_r^{c_r}) / (\beta_1^{d_1},\ldots,\beta_s^{d_s})$ be a $n$-sharp and basic skew partition. The tight demolition of $\alpha / \beta$ does not increase the number of boxes in any row or column, and hence is $n$-sharp. By definition the tight demolition of $\alpha / \beta$ deletes all empty rows and columns and hence is basic. It remains to show that the tight demolition of $\alpha / \beta$ is tight. Let us consider the following four cases:

\noindent \emph{Case 1 ($\tau(\alpha/\beta) = \sigma(\alpha/\beta) = 0$):} By definition $(\alpha/\beta)^{\sf{ti}} = (\alpha/\beta)^{\sf{ba}} = \alpha/\beta$. Since $\alpha/\beta$ is tight, the tight demolition is tight. 

\noindent \emph{Case 2 ($\tau(\alpha/\beta) \neq 0, \sigma(\lambda/\beta) = 0$):} By definition, $(\alpha_1^{c_1 - \tau(\alpha / \beta)},\ldots,\alpha_r^{c_r}) / (\beta_1^{d_1},\ldots,\beta_s^{d_s})$ is tight and $\ell(\beta) = c_1-\tau(\alpha/\beta)$. Therefore $(\alpha_1^{c_1 - \tau(\alpha / \beta)},\ldots,\alpha_r^{c_r}) / (\beta_1^{d_1},\ldots,\beta_s^{d_s})$ does not have empty rows. If there are also no empty columns, then $(\alpha/\beta)^{\sf{ti}} = (\alpha_1^{c_1 - \tau(\alpha / \beta)},\ldots,\alpha_r^{c_r}) / (\beta_1^{d_1},\ldots,\beta_s^{d_s})$ is tight. Suppose column $z$ is empty, then we have $\beta_r \geq z$ and $\alpha_2<z$. As a result, $(r,c)\in (\alpha_1^{c_1 - \tau(\alpha / \beta)},\ldots,\alpha_r^{c_r}) / (\beta_1^{d_1},\ldots,\beta_s^{d_s})$ implies that $(r,c)\in [1,c_1-\tau(\alpha/\beta)]\times[z+1,\alpha_1]\cup [c_1-\tau(\alpha/\beta)+1,\ell(\alpha)-\tau(\alpha/\beta)]\times[1,z-1]$. Since the row sets and column sets of the two rectangular boxes are disjoint, there are no rows with $\alpha_1$ boxes or columns with $\ell(\alpha)-\tau(\alpha/\beta)$ boxes and thus $(\alpha/\beta)^{\sf{ti}}$ is tight.

\noindent \emph{Case 3 ($\tau(\alpha/\beta) = 0, \sigma(\lambda/\beta) \neq 0$):} Since $(\alpha/\beta)^{T}$ satisfies the condition in Case 2, we know $((\alpha/\beta)^{T})^{\sf{ti}}$ is tight. Since taking the transpose does not affect tightness, $(\alpha/\beta)^{\sf{ti}} = (((\alpha/\beta)^{T})^{\sf{ti}})^{T}$ is also tight.

\noindent \emph{Case 4 ($\tau(\alpha/\beta) \neq 0, \sigma(\lambda/\beta) \neq 0$):} Since $\beta \neq \emptyset$, we know $\alpha_r-\sigma(\alpha/\beta)>0$ and $c_1-\tau(\alpha/\beta)>0$. Since $\alpha$ is not a rectangle, $\alpha_1>\alpha_r$ and $c_2>0$. Combining $c_1-\tau(\alpha/\beta)>0$ and $\alpha_1>\alpha_r$, we get $[1,c_1-\tau(\alpha/\beta)]\times[\alpha_r+1,\alpha_1]\subseteq ((\alpha_1-\sigma(\alpha/\beta))^{c_1 - \tau(\alpha / \beta)},\ldots,(\alpha_r-\sigma(\alpha/\beta))^{c_r}) / (\beta_1^{d_1},\ldots,\beta_s^{d_s})$. Combining $\alpha_r-\sigma(\alpha/\beta)>0$ and $c_2>0$, we get $[c_1-\tau(\alpha/\beta)+1,\ell(\alpha)-\tau(\alpha/\beta)]\times[1,\alpha_r-\sigma(\alpha/\beta)]\subseteq ((\alpha_1-\sigma(\alpha/\beta))^{c_1 - \tau(\alpha / \beta)},\ldots,(\alpha_r-\sigma(\alpha/\beta))^{c_r}) / (\beta_1^{d_1},\ldots,\beta_s^{d_s})$. Notice that there are no empty column or row in $((\alpha_1-\sigma(\alpha/\beta))^{c_1 - \tau(\alpha / \beta)},\ldots,(\alpha_r-\sigma(\alpha/\beta))^{c_r}) / (\beta_1^{d_1},\ldots,\beta_s^{d_s})$. Therefore $(\alpha/\beta)^{\sf{ti}}$ is tight by definition.

Since we have exhaust all possible cases of $\alpha/\beta$, we conclude that the tight demolition of $\alpha/\beta$ is tight.

Let $\alpha / \beta = (\alpha_1^{c_1},\ldots,\alpha_r^{c_r}) / (\beta_1^{d_1},\ldots,\beta_s^{d_s})$ be a $n$-sharp, basic, and tight skew partition that is not ordinary. Then $(\alpha / \beta)^{\sf or}$ is the skew partition with the same number of rows and columns as $\alpha / \beta$ that results from rotating $\alpha / \beta$ by $180$ degrees. Thus $(\alpha / \beta)^{\sf or}$ is $n$-sharp and basic. Fixing the rectangle $((\alpha_1)^{\ell(\alpha)})$, we have $(\alpha / \beta)^{\sf or} = \beta^{\vee} / \alpha^{\vee}$ with $\beta^{\vee} = ((\alpha_1)^{\ell(\alpha) - \ell(\beta)},(\alpha_1-\beta_s)^{d_s},\ldots, (\alpha_1-\beta_1)^{d_1})$. Since $\alpha / \beta$ is basic, $\ell(\alpha) > \ell(\beta)$ and hence the first row of $\beta^{\vee}$ has $\alpha_1$ boxes. Suppose that $(\alpha / \beta)^{\sf or}$ is not tight. This means $(\alpha / \beta)^{\sf or}$ contains a row with $\alpha_1$ boxes. Since $(\alpha / \beta)^{\sf or}$ is the 180 degree rotation of $\alpha / \beta$ this implies that $\alpha / \beta$ contains a row with $\alpha_1$ boxes. That is, $\alpha / \beta$ is not tight, a contradiction.

Since $\alpha / \beta$ is not ordinary, this implies 
\begin{equation}
\label{eq:orreducisor1}
{\sf np}(\alpha) - 1 > {\sf np}(\beta).
\end{equation}
If we fix the rectangle $((\alpha_1)^{\ell(\alpha)})$, then $(\alpha / \beta)^{\sf or} = \beta^{\vee} / \alpha^{\vee}$. Then, since $\alpha / \beta$ is basic, ${\sf np}(\beta^{\vee}) = {\sf np}(\beta) + 1$ and ${\sf np}(\alpha^{\vee}) = {\sf np}(\alpha) - 1$. Hence, \eqref{eq:orreducisor1} implies
\begin{equation}
{\sf np}(\beta^{\vee}) - 1 = {\sf np}(\beta) < {\sf np}(\alpha) - 1 = {\sf np}(\alpha^{\vee}),
\end{equation}
which means $(\alpha / \beta)^{\sf or}$ is ordinary.

The multiplicity-freeness claim is a corollary of Propositions ~\ref{prop:reductionbasic},~\ref{prop:reductionnsharp},~\ref{prop:reductiontight}, and~\ref{prop:reductionordinary}.
\qed
\section{Multiplicity-free upper bounds}

We reformulate Theorem \ref{theorem:mainResult} in the following equivalent way:

\begin{theorem}
\label{theorem:refomulatedResult}
Let $s_{\lambda / \mu}(x_1,\ldots,x_n)$ be a basic, $n$-sharp, tight, ordinary skew Schur polynomial where $\lambda = (\lambda_1^{l_1},\lambda_2^{l_2},\ldots,\lambda_p^{l_p})$ and $\mu = (\mu_1^{k_1},\mu_2^{k_2},\ldots,\mu_q^{k_q})$. Fix the rectangle $(\lambda_1^{\ell(\lambda)})$ and set $\rho = \rho(\lambda / \mu)$. Then $s_{\lambda / \mu}(x_1,\ldots,x_n)$ is multiplicity-free if and only if $\lambda$, $\mu$ satisfy at least one of:
\begin{enumerate}[label=(\Roman{*})]
	\item $\lambda^\vee$ is a rectangle of shortness $1$ or an empty partition.
	\item $\lambda^\vee$ is a rectangle of shortness $2$ and $\mu$ is a fat hook.
	\item $\lambda^\vee$ is a rectangle and $\mu$ is a fat hook of shortness $1$.
	\item $\lambda^\vee$ and $\mu$ are both rectangles.
	\item $\lambda^\vee$ is a rectangle of shortness at least $3$, $\mu$ is a fat hook of shortness at least $2$, $\lambda_2 = \mu_1$, $l_2>k_1$ and $n = \rho+1$.
	\item $\lambda^\vee$ is a rectangle of shortness at least $3$, $\mu$ is a fat hook of shortness at least $2$, $\lambda_2 = \mu_1$, $l_2\leq k_1$ and $n = \rho+1$ or $\rho+2$.
	\item $\lambda^\vee$ is a rectangle of shortness at least $3$, $\mu$ is a fat hook of shortness at least $2$, $\mu_1>\lambda_2>\mu_2$, $k_1\geq l_2$ and $n = \rho+1$.
	\item $\lambda^\vee$ is a rectangle of shortness at least $3$, $\mu$ is a fat hook of shortness at least $2$, $l_2\geq l_1$, $\mu_2=\lambda_2$ and $n = \rho+1$ or $\rho+2$.
	\item $\lambda^\vee$ is a rectangle of shortness at least $3$, $\mu$ is a fat hook of shortness at least $2$, $l_1> l_2$, $\mu_2=\lambda_2$ and $n = \rho+1$.
	\item $\lambda^\vee$ is a rectangle of shortness at least $2$, $\sf{np}(\mu)>2$, $\lambda_2 = \mu_q$, $l_2\geq l_1$ and $n = \rho+1$.
	\item $\lambda^\vee$ is a rectangle of shortness at least $2$, $\sf{np}(\mu) > 2$, $\mu_1 = \lambda_2$, $k_1 \geq l_2$ and $n = \rho+1$.
\end{enumerate}
\end{theorem}

Note that we don't need to consider the case where $n\leq \rho$ since the skew Schur polynomial would not be $n$-sharp.

\begin{theorem}[\cite{G10}, \cite{TY10}]
\label{thm:TYGUT}
    Let $\lambda/\mu$ be a basic, ordinary skew partition and fix the rectangle $(\lambda_1^{\ell(\lambda)})$. The skew Schur function $s_{\lambda/\mu}$ is multiplicity-free if and only if one of the following holds:
    \begin{enumerate}[label=(\Roman{*})]
        \item either $\lambda^\vee$ is a rectangle of shortness $1$ or an empty partition.
        \item $\lambda^\vee$ is a rectangle of shortness $2$ and $\mu$ is a fat hook.
        \item $\lambda^\vee$ is a rectangle and $\mu$ is a fat hook of shortness $1$.
	    \item $\lambda^\vee$ and $\mu$ are both rectangles.
    \end{enumerate}
\end{theorem}

In order to classify all multiplicity-free skew Schur polynomials, it is enough to find, for each basic, tight, ordinary skew partition $\lambda/\mu$ the minimal integer $\sf{m}(\lambda/\mu)$ such that $\lambda/\mu$ is $\sf{m}(\lambda/\mu)$-sharp and $s_{\lambda/\mu}(x_1,\ldots,x_{\sf{m}(\lambda/\mu)})$ is not multiplicity-free. As a result of Theorem~\ref{thm:TYGUT}, we only need to consider the basic, tight, ordinary skew partitions $\lambda/\mu$ that satisfies any of the three following conditions:
\begin{enumerate}[label=(\Roman{*})]
    \item $\lambda^\vee$ is a rectangle of shortness at least $3$ and $\mu$ is a fat hook of shortness at least $2$;
    \item $\lambda^\vee$ is a rectangle of shortness at least $2$ and $\sf{np} (\mu) \geq 3$;
    \item $\sf{np} (\lambda^\vee), \sf{np}(\mu) \geq 2$.
\end{enumerate}
 In this section, we will find upper bounds on $\sf{m}(\lambda/\mu)$ for $\lambda/\mu$ satisfying each of the three conditions. For the rest of this section, we fix $\lambda = (\lambda_1^{l_1},\lambda_2^{l_2},\ldots,\lambda_p^{l_p})$ and $\mu = (\mu_1^{k_1},\mu_2^{k_2},\ldots,\mu_q^{k_q})$ with $p = \sf{np}(\lambda)$ and $q = \sf{np}(\mu)$.

\subsection{Both \texorpdfstring{$\sf{np} (\lambda^\vee) \text{ and }\sf{np}(\mu)$}{} are at least \texorpdfstring{2}{}}

\begin{theorem}
\label{thm:lambdamunotrectangle}
Let $s_{\lambda/\mu}(x_1,\ldots,x_n)$ be a basic, n-sharp, tight skew Schur polynomial such that neither $\lambda^\vee$ nor $\mu$ is a rectangle, then $s_{\lambda/\mu}(x_1,\ldots,x_n)$ is not multiplicity-free (or equivalently, $\sf{m}(\lambda/\mu) = \rho+1$).
\end{theorem}

For $1 \leq c \leq \lambda_1$, let $U(c) := \min\{r:(r,c)\in \lambda / \mu\}$ and $L(C) := \max\{r:(r,c)\in \lambda / \mu\}$. For the rest of this section, we fix $\lambda = (\lambda_1^{l_1},\dots, \lambda_p^{l_p})$ and $\mu = \{\mu_1^{k_1},\dots, \mu_q^{k_q})$ with $p = {\sf np}(\lambda)$ and $q = {\sf np}(\mu)$.

\begin{lemma}
\label{lem:nonrect}
    If neither $\lambda^\vee$ nor $\mu$ is a rectangle, then at least one of the following is true:
    \begin{enumerate}[label=(\Roman{*})]
        \item There exist three columns indexed by $c_1,c_2,c_3$ in $\lambda/\mu$ such that $U(c_1),U(c_2),U(c_3)$ are all distinct, and $L(c_1),L(c_2),L(c_3)$ are all distinct.
        \item There exist four columns indexed by $c_1,c_2,c_3,c_4$ in $\lambda/\mu$ such that $U(c_1)=U(c_2)$,$L(c_3)=L(c_4)$, $U(c_1),U(c_3),U(c_4)$ are all distinct,  and $L(c_1),L(c_2),L(c_3)$ are all distinct.
    \end{enumerate}
\end{lemma}

\begin{figure}
    \begin{center}
    \begin{tikzpicture}[scale = 1.6]
        \draw[black, thick] (0,0) -- (4,0) -- (4,4) -- (0,4) -- (0,0);
        \draw[black, thick] (0,1) -- (1.5,1) -- (1.5,2) -- (2.5,2) -- (2.5,3) -- (3.5,3) -- (3.5,4);
        \draw (1,2.75) node{$\mu$};
        \draw[gray, dashed] (1.5,4) -- (1.5,0);
        \draw[gray, dashed] (2.5,4) -- (2.5,0);
        \draw[gray, dashed] (3.5,4) -- (3.5,0);
        \draw[black, thick] (1,0) -- (1,0.5) -- (2,0.5) -- (2,1) -- (3.6,1) -- (3.6,2.4) -- (4,2.4);
        \draw (1.15,0.75) node{$a_1$};
        \draw[very thin, gray] (1,0.5)--(1,1);
        \draw[very thin, gray] (1.25,0.5)--(1.25,1);
        \draw (3.75, 3.25) node{$a_2$};
        \draw[very thin, gray] (3.6,2.4)--(3.6,4);
        \draw[very thin, gray] (3.9,2.4)--(3.9,4);
        \draw (3,0.5) node{$\lambda^c$};
        \draw (0.75,4.2) node{$S_4$};
        \draw (2,4.2) node{$S_3$};
        \draw (3,4.2) node{$S_2$};
        \draw (3.75,4.2) node{$S_1$};

    \end{tikzpicture}
\end{center}
    \caption{}
    \label{Case1}
\end{figure}

\begin{proof}
    Since $\mu$ and $\lambda^\vee$ are not rectangles, $p\geq 3$ and $q\geq 2$. For $1\leq i\leq q+1$, set $S_i = \{\mu_i+1,\dots, \mu_{i-1}\}$, where $\mu_0 = \lambda_1$ and $\mu_{q+1} = 0$.

    By construction, 
    \begin{equation}
    \label{eq:samestart}
    U(c) = U(d)\text{ for }c, d \in S_i
    \end{equation}
    and
    \begin{equation}
    \label{eq:greaterstart}
    U(c) < U(d)\text{ for }c \in S_i, d \in S_j\text{ with }i > j
    \end{equation}
    
    Denote
    \begin{equation}
    a_1 = \min\{k: L(k) < \ell(\lambda)\} \text{ and } a_2 = \min\{k:L(k) = l_1\}.
    \end{equation}
    Note that by construction $L(1)>L(a_1)>L(a_2)=L(\lambda_1)$, so in particular $1<a_1<a_2$. If $a_1\not\in S_1\cup S_{q+1}$, then columns 1, $a_1$ and $\lambda_1$ satisfy Lemma~\ref{lem:nonrect} (\RomanNumeralCaps{1}). 
    
    If $a_1 \in S_1$, then $a_1 < a_2$ implies $a_2 \in S_1$. Set $c_1 = a_1,c_2 = a_2$, and $c_3 = \max(S_2), c_4 = \max(S_3)$. Notice that $U(c_1)=U(c_2)$ by \eqref{eq:samestart}, with $U(c_1),U(c_3),U(c_4)$ all distinct by \eqref{eq:greaterstart}. Since $c_1$ is the minimum index such that $L(c_1)<\ell(\lambda)$ and $c_1 \in S_1$, we have $L(c_3) = L(c_4) = \ell(\lambda)$. Finally, $L(c_1)=\ell(\lambda)-l_p,L(c_2)=l_1,L(c_3)=\ell(\lambda)$ are all distinct. Thus these four columns satisfy Lemma~\ref{lem:nonrect} (\RomanNumeralCaps{2}). 
    
    If $a_1 \in S_{q+1}$ and $a_2 \leq \min(S_{q})$, then let $c_1 = 1$, $c_2 =a_1$, $c_3 = \min(S_{q})$, and $c_4 = \lambda_1$. Then $U(c_1) = U(c_2)$ by \eqref{eq:samestart}, and $U(c_1),U(c_3),U(c_4)$ are all distinct by \eqref{eq:greaterstart}. Since $a_2 = \min\{k:L(k) = l_1\}$ and $a_2 \leq c_3$, then $L(c_3) = L(c_4) = l_1$. Furthermore, $L(c_1) = \lambda_1$, $L(c_2)=\ell(\lambda)-l_p$, and $L(c_3) =l_1$ are all distinct. Hence $c_1,c_2,c_3,$ and $c_4$ satisfy Lemma~\ref{lem:nonrect}(\RomanNumeralCaps{2}).

    If $a_1 \in S_{q+1}$ and $a_2 > \min(S_{q})$, then columns 1, $\min(S_{q})$, and $\lambda_1$ satisfy Lemma~\ref{lem:nonrect}(\RomanNumeralCaps{1}). Since we have exhausted all possibilities for $a_1$ and $a_2$, we conclude that either Lemma~\ref{lem:nonrect} (\RomanNumeralCaps{1}) or (\RomanNumeralCaps{2}) always hold.
\end{proof}

\begin{proposition}
\label{prop:3col}
If there exist columns $c_1,c_2,c_3$ in $\lambda/\mu$ such that $U(c_1),U(c_2),U(c_3)$ are all distinct and $L(c_1),L(c_2),L(c_3)$ are all distinct, then $s_{\lambda/\mu}(x_1,\ldots,x_{\rho+1})$ is not multiplicity-free.
\end{proposition}
\begin{proof}

First consider the case where $\lambda/\mu$ consists of only those three columns, and without loss of generality say that $c_1=3$, $c_2=2$, and $c_3=1$.  Let $\rho_i := {\sf CS}_{c_i}(\lambda/\mu)$ for $i \in \{1,2,3\}$.

Our goal is to construct two distinct ballot fillings of $\lambda/\mu$ with the same content.  In both fillings, we fill column $c_1$ with the numbers from $1$ to $\rho_1$.  The entries of column $c_2$ and column $c_3$ will depend on the relative sizes of $\rho_1$, $\rho_2$, and $\rho_3$:

\noindent \emph{Case $1$ $(\rho_1\leq \rho_2 < \rho_3)$:} For Filling 1, fill column $c_2$ with $[\rho_2]$ and column $c_3$ with $[\rho_3+1]\setminus \{\rho_1\}$. For Filling 2, fill column $c_2$ with $[\rho_2+1]\setminus \{\rho_1\}$ and column $c_3$ with $[\rho_3+1]\setminus \{\rho_2+1\}$.

\noindent \emph{Case $2$ $(\rho_1< \rho_3 \leq \rho_2)$:} For Filling 1, fill column $c_2$ with $[\rho_2]$ and column $c_3$ with $([\rho_3]\setminus \{\rho_1\})\cup \{\rho_2+1\}$. For Filling 2, fill column $c_2$ with $[\rho_2+1]\setminus \{\rho_1\}$ and column $c_3$ with $[ \rho_3]$.

\noindent \emph{Case $3$ $(\rho_2\leq \rho_1 < \rho_3)$:} For Filling 1, fill column $c_2$ with $[\rho_2]$ and column $c_3$ with $[\rho_3+1]\setminus \{\rho_2\}$. For Filling 2, fill column $c_2$ with $[\rho_2-1]\cup\{\rho_1+1\}$ and column $c_3$ with $[\rho_3+1]\setminus\{\rho_1+1\}$.

\noindent \emph{Case $4$ $(\rho_2< \rho_3 \leq \rho_1)$:} For Filling 1, fill column $c_2$ with $[\rho_2]$ and column $c_3$ with $([\rho_3]\setminus \{\rho_2\})\cup \{\rho_1+1\}$. For Filling 2, fill column $c_2$ with $[\rho_2-1]\cup\{\rho_1+1\}$ and column $c_3$ with $[\rho_3]$.

\noindent \emph{Case $5$ $(\rho_3\leq \rho_1 \leq \rho_2)$:} For Filling 1, fill column $c_2$ with $[\rho_2]$ and column $c_3$ with $[\rho_3-1]\cup \{ \rho_2+1\}$. For Filling 2, fill column $c_2$ with $[\rho_2+1]\setminus \{\rho_1\}$ and column $c_3$ with $[\rho_3-1]\cup \{\rho_1\}$.

\noindent \emph{Case $6$ $(\rho_3\leq \rho_2 \leq \rho_1)$:} For Filling 1, fill column $c_2$ with $[\rho_2]$ and column $c_3$ with $[\rho_3-1]\cup \{\rho_1+1\}$. For Filling 2, fill column $c_2$ with $[\rho_2-1]\cup \{\rho_1+1\}$ and column $c_3$ with $[\rho_3-1]\cup\{ \rho_2\}$.

In all of the above cases, both fillings have the same content, and they are strictly increasing within columns.  Similarly, in all cases the column reading word is ballot. It remains to show that the fillings are weakly increasing within rows.

\begin{claim}
\label{claim:c2c3}
Let $T$ be any one of the fillings defined in the six cases above. 
\begin{enumerate}[label=(\Roman{*})]
\item If $(r,c_2),(r,c_3)\in \lambda/\mu$, then $T(r,c_3) \leq T(r,c_2)$.
\item If $(r,c_1),(r,c_2)\in \lambda/\mu$, then $T(r,c_2) \leq T(r,c_1)$.
\end{enumerate}
\end{claim}
\begin{proof}
We first prove (I). Since $U(c_3)>U(c_2)$, at most the top $\rho_2-1$ boxes of column $c_3$ are to the left of a box in  column $c_2$. In addition, there is at least one box in column $c_2$ in a higher row than the highest box in column $c_3$.  Similarly, because $L(c_3)>L(c_2)$, at most the top $\rho_3-1$ labels of $c_3$ are to the left of a label of column $c_2$.

Combining these two facts, we know that at most the top $\min(\rho_2,\rho_3)-1$ boxes in column $c_3$ have a box in column $c_2$ directly to the right.  Notice that in all of the cases of the fillings above, for $1\leq k\leq \min(\rho_2,\rho_3)-1$, the $k^{th}$ box from the top of column $c_3$ has a value of at most $k+1$, and its right neighbor is at least the $(k+1)^{th}$ box from the top of column $c_2$.  Thus even if column $c_2$ was minimally filled, its entries would still be larger than the corresponding entries in column $c_3$, which proves the claim.

The proof of (II) follows by the same logic as the proof of (I).
\end{proof}

This completes the proof in the case where $\lambda/\mu$ consists of exactly three columns. If $\lambda/\mu$ has more than three columns, then let $\lambda^*/\mu^*$ be the skew shape consisting of only the three columns indexed by $c_1$, $c_2$, and $c_3$ in the statement of the theorem. By Corollary~\ref{Cor:GUT2}, $s_{\lambda^*/\mu^*}(x_1,\dots, x_{\rho(\lambda^*/\mu^*) + 1})$ is not multiplicity-free implies $s_{\lambda/\mu}(x_1,\dots, x_{\rho(\lambda/\mu) + 1})$ is not multiplicity-free.
\end{proof}

\begin{proposition}
\label{prop:4col}
Suppose there exist columns $c_1,c_2,c_3,c_4$ in $\lambda/\mu$ such that either $U(c_1)=U(c_2)$ or $L(c_1)=L(c_2)$ but not both, either $U(c_1)=U(c_2)$ or $L(c_1)=L(c_2)$ but not both, and $U(c_i)\not=U(c_j)$ and $L(c_i)\not=L(c_j)$ for all $i\in \{1,2\}$ and $j\in \{3,4\}$.  If neither $\lambda^\vee$ nor $\mu$ are rectangles of shortness 1, then $s_{\lambda/\mu}(x_1,\ldots,x_{\rho+1})$ is not multiplicity-free.
\end{proposition}
\begin{proof}
Consider the case where $\lambda/\mu$ consists of only those four columns $c_1,c_2,c_3,c_4$. Without loss of generality say $\{c_1,c_2\} = \{3,4\}$ and ${\sf CS}_{c_1}(\lambda/\mu)<{\sf CS}_{c_2}(\lambda/\mu)$, and $\{c_3,c_4\} = \{1,2\}$ with ${\sf CS}_{c_3}(\lambda/\mu)<{\sf CS}_{c_4}(\lambda/\mu)$. There are four possible arrangements of $c_1$ through $c_4$ as shown in Figure~\ref{only4col}. Let $\rho_i := {\sf CS}_{c_i}(\lambda/\mu)$ for $i \in \{1,2,3, 4\}$.

\begin{figure}
    \begin{center}
    \begin{subfigure}{0.2\textwidth}
    \hspace{0.5cm}
    \begin{tikzpicture}[scale = 0.75]
        \draw[black, thick] (0,0) -- (4,0) -- (4,4) -- (0,4) -- (0,0);
        \draw[black, thick] (1,0) -- (1,1) -- (2,1) -- (2,2) -- (3,2) -- (3,3) -- (4,3);
        \draw[black, thick] (0,2.5) -- (2,2.5) -- (2,4);
        \draw[gray, dashed] (1,1) -- (1,2.5);
        \draw[gray, dashed] (3,3) -- (3,4);
		\draw[gray, dashed] (2,2) -- (2,2.5);
        \draw (3.5,3.5) node{$c_1$};
        \draw (2.5,3.05) node{$c_2$};
        \draw (1.5, 1.75) node{$c_3$};
        \draw (0.5, 1.25) node{$c_4$};
    \end{tikzpicture}
    \caption{}
	\label{Figure2a}
    \end{subfigure}
    \begin{subfigure}{0.2\textwidth}
    \hspace{0.6cm}
        \begin{tikzpicture}[scale=0.75]
            \draw[black, thick] (0,0) -- (4,0) -- (4,4) -- (0,4) -- (0,0);
            \draw[black, thick] (1,0) -- (1,1) -- (2,1) -- (2,2) -- (4,2);
            \draw[black, thick] (0,2.5) -- (2,2.5) -- (2,3.5) -- (3,3.5) -- (3,4);
            \draw[gray, dashed] (1,1) -- (1,2.5);
            \draw[gray, dashed] (3,2) -- (3,3.5);
			\draw[gray, dashed] (2,2) -- (2,2.5);
            \draw (3.5,3.05) node{$c_2$};
            \draw (2.5,2.8) node{$c_1$};
            \draw (1.5,1.75) node{$c_3$};
            \draw (0.5,1.25) node{$c_4$};
        \end{tikzpicture}
        \caption{}
    \end{subfigure}
    \begin{subfigure}{0.2\textwidth}
    \hspace{0.7cm}
        \begin{tikzpicture}[scale=0.75]
            \draw[black, thick] (0,0) -- (4,0) -- (4,4) -- (0,4) -- (0,0);
            \draw[black, thick] (2,0) -- (2,2) -- (3,2) -- (3,3) -- (4,3);
            \draw[black, thick] (0,1.5) -- (1,1.5) -- (1,2.5) -- (2,2.5) -- (2,4);
            \draw[gray, dashed] (1,0) -- (1,1.5);
            \draw[gray, dashed] (3,3) -- (3,4);
			\draw[gray, dashed] (2,2) -- (2,2.5);
            \draw (3.5,3.5) node{$c_1$};
            \draw (2.5,3.05) node{$c_2$};
            \draw (1.5,1.25) node{$c_4$};
            \draw (0.5,0.75) node{$c_3$};
        \end{tikzpicture}
        \caption{}
    \end{subfigure}
    \begin{subfigure}{0.2\textwidth}
    \hspace{0.75cm}
        \begin{tikzpicture}[scale=0.75]
            \draw[black, thick] (0,0) -- (4,0) -- (4,4) -- (0,4) -- (0,0);
            \draw[black, thick] (2,0) -- (2,2) -- (4,2);
            \draw[black, thick] (0,1.5) -- (1,1.5) -- (1,2.5) -- (2,2.5) -- (2,3.5) -- (3,3.5) -- (3,4);
            \draw[gray, dashed] (1,0) -- (1,1.5);
            \draw[gray, dashed] (3,2) -- (3,3.5);
			\draw[gray, dashed] (2,2) -- (2,2.5);
            \draw (3.5,3.05) node{$c_2$};
            \draw (2.5,2.8) node{$c_1$};
            \draw (1.5,1.25) node{$c_4$};
            \draw (0.5,0.75) node{$c_3$};
        \end{tikzpicture}
        \caption{}
		\label{Figure2d}
    \end{subfigure}
    \caption{}
    \label{only4col}    
    \end{center}
\end{figure}

We want to construct two distinct ballot fillings of $\lambda/\mu$ with the same content.  In both fillings, we fill column $c_1$ and $c_2$ with $[\rho_1]$ and $[\rho_2]$ respectively.  The entries of column $c_3$ and $c_4$ depend on the relative sizes of $\rho_1,\rho_2,\rho_3,\rho_4$ and the relative sizes of $c_3$ and $c_4$.

\noindent \emph{Case $1$ $(c_3 > c_4, \rho_2\geq \rho_4, \rho_4-1\geq \rho_1+1,\rho_3\geq \rho_1)$:} For Filling 1, fill column $c_3$ with $[\rho_3+1]\setminus \{\rho_1\}$ and column $c_4$ with $([\rho_4]\setminus\{\rho_3+1\})\cup\{\rho_2+1\}$.
For Filling 2, fill column $c_3$ with $([\rho_3]\setminus \{\rho_1\})\cup \{\rho_2+1\}$ and column $c_4$ with $[\rho_4]$.

\noindent \emph{Case $2$ $(c_3 > c_4, \rho_2\geq \rho_4, \rho_4-1\geq \rho_1+1,\rho_3< \rho_1)$:} For Filling 1, fill column $c_3$ with $[\rho_3-1]\cup \{\rho_1 + 1\}$ and column $c_4$ with $([\rho_4]\setminus\{\rho_1+1\})\cup\{\rho_2+1\}$.
For Filling 2, fill column $c_3$ with $[\rho_3-1]\cup \{\rho_2+1\}$ and column $c_4$ with $[\rho_4]$.

\noindent \emph{Case $3$ $(c_3 > c_4, \rho_2\geq \rho_4, \rho_4-1< \rho_1+1)$:} For Filling 1, fill column $c_3$ with $[\rho_3-1]\cup \{\rho_1+1\}$ and column $c_4$ with $[\rho_4-1]\cup \{\rho_2+1\}$. For Filling 2, fill column $c_3$ with $[\rho_3-1]\cup \{\rho_2+1\}$ and column $c_4$ with $[\rho_4-1]\cup \{\rho_1+1\}$.

\noindent \emph{Case $4$ $(c_3> c_4, \rho_2< \rho_4, \rho_3\geq \rho_2)$:} For Filling 1, fill column $c_3$ with $[\rho_3+1]\setminus\{\rho_1\}$ and column $c_4$ with $[\rho_4+1]\setminus \{\rho_2\}$. For Filling 2, fill column $c_3$ with $[\rho_3+2]\setminus \{\rho_1,\rho_2\}$ and column $c_4$ with $[\rho_4+1]\setminus \{\rho_3+2\}$.

\noindent \emph{Case $5$ $(c_3 > c_4, \rho_2< \rho_4, \rho_3< \rho_2,\rho_1\geq \rho_3)$:} For Filling 1, fill column $c_3$ with $[\rho_3 -1 ]\cup \{\rho_1+1\}$ and column $c_4$ with $[\rho_4+1]\setminus \{\rho_1+1\}$. For Filling 2, fill column $c_3$ with $[\rho_3 -1 ]\cup \{\rho_2+1\}$ and column $c_4$ with $[\rho_4+1]\setminus \{\rho_2+1\}$.

\noindent \emph{Case $6$ $(c_3 > c_4, \rho_2< \rho_4, \rho_3< \rho_2,\rho_1< \rho_3)$:} For Filling 1, fill column $c_3$ with $[\rho_3+1]\setminus \{\rho_1\}$ and column $c_4$ with $[\rho_4+1]\setminus \{\rho_3+1\}$. For Filling 2, fill column $c_3$ with $([\rho_3]\setminus \{\rho_1\})\cup\{\rho_2+1\}$ and column $c_4$ with $[\rho_4+1]\setminus \{\rho_2+1\}$.

\noindent \emph{Case $7$ $(c_3 < c_4, \rho_2\geq \rho_4, \rho_4-1 < \rho_1+1)$:} For Filling 1, fill column $c_3$ with $[\rho_3-1]\cup \{\rho_1+1\}$ and column $c_4$ with $[\rho_4-1]\cup \{\rho_2+1\}$. For Filling 2, fill column $c_3$ with $c_3$ with $[\rho_3-1]\cup \{\rho_4-1\}$ and column $c_4$ with $[\rho_4-2]\cup \{\rho_1+1,\rho_2+1\}$.

\noindent \emph{Case $8$ $(c_3 < c_4, \rho_2\geq \rho_4, \rho_4-1 \geq \rho_1+1, \rho_1\geq \rho_3)$:} For Filling 1, fill column $c_3$ with $[\rho_3-1]\cup \{\rho_1\}$ and column $c_4$ with $([\rho_4]\setminus\{\rho_1\}\cup \{\rho_2+1\}$. For Filling 2, fill column $c_3$ with $[\rho_3-1]\cup \{\rho_4\}$ and column $c_4$ with $[\rho_4-1]\cup \{\rho_2+1\}$.

\noindent \emph{Case $9$ $(c_3 < c_4, \rho_2\geq \rho_4, \rho_4-1 \geq \rho_1+1, \rho_1< \rho_3)$:} For Filling 1, fill column $c_3$ with $([\rho_3]\setminus \{\rho_1\})\cup \{\rho_4\}$ and column $c_4$ with $[\rho_4-1]\cup \{\rho_2+1\}$.
For Filling 2, fill column $c_3$ with $[\rho_3]$ and column $c_4$ with $([\rho_4]\setminus\{\rho_1\})\cup \{\rho_2+1\}$.

\noindent \emph{Case $10$ $(c_3 < c_4, \rho_2< \rho_4, \rho_2 \leq \rho_3)$:} For Filling 1, fill column $c_3$ with $[\rho_3+1]\setminus \{\rho_1\}$ and column $c_4$ with $[\rho_4+1]\setminus \{\rho_2\}$. For Filling 2, fill column $c_3$ with $[\rho_3+1]\setminus \{\rho_2\}$ and column $c_4$ with $[\rho_4+1]\setminus \{\rho_1\}$.

\noindent \emph{Case $11$ $(c_3 < c_4, \rho_2< \rho_4, \rho_2 > \rho_3, \rho_3>\rho_1)$:} For Filling 1, fill column $c_3$ with $([\rho_3]\setminus \{\rho_1\})\cup \{\rho_2\}$ and column $c_4$ with $[\rho_4+1]\setminus \{\rho_2\}$. For Filling 2, fill column $c_3$ with $[\rho_3]$ and column $c_4$ with $[\rho_4+1]\setminus \{\rho_1\}$.

\noindent \emph{Case $12$ $(c_3 < c_4, \rho_2< \rho_4, \rho_2 > \rho_3, \rho_3\leq\rho_1)$:} For Filling 1, fill column $c_3$ with $[\rho_3-1]\cup \{\rho_1\}$ and column $c_4$ with $[\rho_4+1]\setminus \{\rho_2\}$. For Filling 2, fill column $c_3$ with $[\rho_3-1]\cup \{\rho_2\}$ and column $c_4$ with $[\rho_4+1]\setminus \{\rho_1\}$.

In all of the above cases, both fillings have the same content, and they are strictly increasing within columns. Similarly, in all cases the column reading word is ballot. It remains to show that the fillings are weakly increasing within rows.

For each pair of indices $(c_1,c_2)$ or $(c_3,c_4)$, the fillings described above are weakly increasing within rows for those column pairs. Let $c'' = \max\{c_3,c_4\}$ and $c' = \min \{c_1,c_2\}$. It remains to show that the fillings described above are weakly increasing within rows for the column pair $c'$ and $c''$.

\begin{claim}
\label{claim:col2col3}
Let $T$ be any one of the fillings defined in the twelve cases above. If $(r,c''),(r,c')\in \lambda/\mu$, then $T(r,c'') \leq T(r,c')$.
\end{claim}

\begin{proof}
Let $\rho'' = {\sf CS}_{c''}(\lambda/\mu)$.  Because $U(c'')>U(c')$ and ${\sf CS}_{c'}(\lambda/\mu)\leq \rho_2$, at most the top $\rho_2-1$ boxes of column $c''$ are to the left of a box in column $c'$.  In addition, the boxes in column $c'$ have at least one more box above them than the corresponding box in column $c''$.  Similarly, because $L(c'')>L(c')$, at most the top $\rho''-1$ boxes of column $c''$ are to the left of a box in column $c'$.  Combining these two facts, we know that at most the top $\min(\rho_2,\rho'')-1$ boxes in column $c''$ have a box in $c'$ directly to the right of it.

Observe that in nearly all of the cases above, for $1\leq k\leq \min(\rho_2,\rho'')-1$, the $k^{th}$ box from the top of column $c''$ has an entry of at most $k+1$, and its right neighbor has at least $k$ boxes above it.  Even if column $c'$ was minimally filled, its entries would still be at least the value of the corresponding entries in column $c''$. This proves the claim in these cases.

There are only two exceptions to the above observation: Case 4 and Case 7.

In Case 4, the only box contradicting the observation is the box that is $\rho_2-1$ from the top of column $c''=c_3$, which has value $\rho_2+1$.  This would violate semistandardness if and only if
\begin{equation}
\label{eq:zz1010}
c'=c_2\text{ and }U(c_2) = U(c_3)-1.
\end{equation}
In this case $\mu$ is a rectangle of shortness 1, and thus \eqref{eq:zz1010} can not be satisfied.

In Case 7, the only box contradicting the observation is the box that is $\rho_4-1$ from the top of column $c''=c_4$, which has value $\rho_1+1$.  This would violate semistandardness if and only if
\begin{equation}
\label{eq:zz2020}
c'=c_1\text{ and }L(c_1) = L(c_4)-1.
\end{equation}
In this case $\lambda^\vee$ is a rectangle of shortness 1, and thus \eqref{eq:zz2020} can not be satisfied.

As a result, in all cases, both fillings of $\lambda/\mu$ are semistandard.
\end{proof}

This completes the proof in the case where $\lambda/\mu$ consists of exactly four columns. If $\lambda/\mu$ has more than four columns, then let $\lambda^*/\mu^*$ be the skew shape consisting of the four columns in the statement of the theorem.  If neither $(\lambda^*)^\vee$ nor $\mu^*$ is a rectangle of shortness $1$, then we have shown that $s_{\lambda^*/\mu^*}(x_1,\dots, x_{\rho(\lambda^*/\mu^*) + 1})$ is not multiplicity-free. By Corollary~\ref{Cor:GUT2}, $s_{\lambda^*/\mu^*}(x_1,\dots, x_{\rho(\lambda^*/\mu^*) + 1})$ is not multiplicity-free implies $s_{\lambda/\mu}(x_1,\dots, x_{\rho(\lambda/\mu) + 1})$ is not multiplicity-free.

On the other hand, suppose either $(\lambda^*)^\vee$ or $\mu^*$ is a rectangle of shortness $1$ as in Figure~\ref{Figure2a} or Figure~\ref{Figure2d}. Since neither $\lambda^\vee$ nor $\mu$ is a rectangle of shortness $1$, there must exist some other column $c$ such that $U(c),U(c_i),U(c_j)$ are all distinct, and $L(c),L(c_i),L(c_j)$ are all distinct for some $1\leq i<j\leq 4$.  By Proposition~\ref{prop:3col}, $s_{\lambda/\mu}(x_1,\dots, x_{\rho(\lambda/\mu)+1})$ is not multiplicity-free.
\end{proof}

\noindent \emph{Proof of Theorem ~\ref{thm:lambdamunotrectangle}:} The four columns defined in Lemma~\ref{lem:nonrect}(II) satisfy the hypotheses in Proposition~\ref{prop:4col}. Combining Lemma~\ref{lem:nonrect}, Proposition~\ref{prop:3col} and Proposition~\ref{prop:4col}, we conclude that $s_{\lambda/\mu}(x_1,\ldots,x_{\rho+1})$ is not multiplicity-free. Thus all basic, n-sharp, tight skew Schur polynomials are not multiplicity-free if neither $\lambda^\vee$ nor $\mu$ is not a rectangle.

\subsection{\texorpdfstring{$\lambda^\vee$}{} is a rectangle of shortness at least \texorpdfstring{$2$}{} and \texorpdfstring{$\sf{np}(\mu) \geq 3$}{}}

\begin{theorem}
\label{thm:suffnotfathook}
Let $s_{\lambda/\mu}(x_1,\ldots,x_n)$ be a basic, n-sharp, tight skew Schur polynomial such that $\lambda^\vee$ is a rectangle of shortness at least $2$ and ${\sf np}(\mu) > 2$, then $s_{\lambda/\mu}(x_1,\ldots,x_n)$ is not multiplicity-free if one of the following conditions hold:
\begin{enumerate}[label=(\Roman{*})]
    \item $\lambda_2 = \mu_q, l_2\geq l_1$ and $n\geq \rho+2$ 
    \item $\lambda_2 = \mu_q, l_2< l_1$ 
    \item $\lambda_2 = \mu_1, k_1\geq l_2$ and $n\geq \rho+2$ 
    \item $\lambda_2 = \mu_1, k_1< l_2$.
    \item $\mu_q<\lambda_2<\mu_1$ 
\end{enumerate}
\end{theorem}

Here we omit the cases where $\lambda_2<\mu_q$ or $\lambda_2>\mu_1$ since the skew partition $\lambda/\mu$ would be not basic if $\lambda_2<\mu_q$ and $l_1 = \ell(\mu)$, and not tight otherwise. 

\begin{definition}
For a basic skew-shape $\lambda/\mu$, let $(\lambda/\mu)^{(-k)}$ be the skew-shape obtained by removing the top $k$ boxes in each column of $\lambda/\mu$ and then applying a basic reduction.
\end{definition}
\begin{lemma}\label{Lemma:krowreduction}
Let $\lambda / \mu$ be a basic skew diagram. If $k\leq \min\{ {\sf CS}_d(\lambda/\mu):1\leq d\leq \lambda_1\}$, then $s_{(\lambda/\mu)^{(-k)}}(x_1,\ldots,x_{n-k})$ is not multiplicity-free implies $s_{\lambda/\mu}(x_1,\ldots,x_n)$ is not multiplicity-free.
\end{lemma}
\begin{proof}
If $s_{(\lambda/\mu)^{(-k)}}(x_1,\ldots,x_{n-k})$ is not multiplicity-free, then there exists at least two ballot tableaux of shape $(\lambda/\mu)^{(-k)}$ with the same content $\nu$ with $\ell(\nu)\leq n-k$. Consider the following filling of $\lambda/\mu$: fill in the top $k$ boxes in each column with $1$ through $k$, with values increasing downwards, and fill in the remaining shape by adding $k$ to the corresponding entries in $(\lambda/\mu)^{(-k)}$. It is trivial, using Lemma 2.2, to show that the resulting two tableaux are ballot. Therefore $s_{\lambda/\mu}(x_1,\ldots,x_n)$ is not multiplicity-free.
\end{proof}
\begin{lemma}\label{Lemma:ballot}
If $w = w_1w_2\ldots w_m$ is a ballot word, then the word $w^{(k)}$ obtained by concatenating $w_m+1,\ldots w_m+k$ to the end of $w$ is a ballot word.
\end{lemma}

\begin{proof}
For $k=1$, it suffices to show that the number of occurrences of $w_m$ in $w$ is strictly greater than the number of occurrences of $w_m+1$. Indeed, if not, then the sequence $w_1\ldots w_{m-1}$ is not ballot, contradicting our assumption that $w$ is ballot. The lemma follows by induction on $k$
\end{proof}

\noindent \emph{Proof of Theorem~\ref{thm:suffnotfathook}:} Since $\lambda^\vee$ is a rectangle and ${\sf np}(\mu) \geq 2$, we have $q\geq 3$, $p=2$.

\begin{figure}
    \begin{center}
    
        \begin{tikzpicture}[scale=1.3]
        \draw[black, thick] (0,0) -- (4,0) -- (4,4) -- (0,4) -- (0,0);
        \draw[black, thick] (0,1) -- (1,1) -- (1,1.8) -- (2,1.8) -- (2,2.6) -- (2.7,2.6) -- (2.7,3.4) -- (3.4,3.4) -- (3.4,4);
        \draw[black, thick] (1,0) -- (1,1) -- (4,1);
        \draw (0.75,3) node{$\mu$};
        \draw (2.4, 0.5) node{$\lambda^c$};
        \draw (0.5, 0.5) node{$d$};
        \draw (1.5, 1.4) node{$c$};
        \draw (3.7, 2.5) node{$a$};
        \draw (2.35, 1.8) node{$b$};
        
        \draw[gray, dashed] (2.7,1) -- (2.7,2.6);
        \draw[gray, dashed] (2,1) -- (2,1.8);
        \draw[gray, dashed] (3.4,1) -- (3.4,3.4);
        \end{tikzpicture}
        
    \caption{Case 1 in Theorem~\ref{thm:suffnotfathook}, with column lengths $a,b,c,$ and $d$.}
    \label{case1necessary}    
    \end{center}
\end{figure}
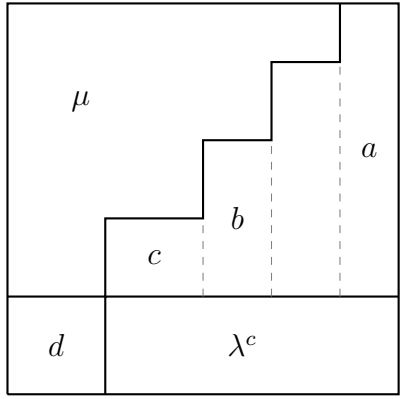

\noindent \emph{Case 1 ($\lambda_2 = \mu_q$):} As illustrated by Figure~\ref{case1necessary}, set
\begin{align*}
A & = \{ \mu_1 + 1,\ldots,\lambda_1 \} & a & = l_1 = {\sf CS}_k(\lambda / \mu)\text{ for }k \in A \\
B & = \{ \mu_{q-1} + 1,\ldots,\mu_{q-2} \} & b & = k_{q-1} + k_q = {\sf CS}_k(\lambda / \mu)\text{ for }k \in B \\
C & = \{ \mu_{q} + 1,\ldots,\mu_{q-1} \} & c & = k_q = {\sf CS}_k(\lambda / \mu)\text{ for }k \in C \\
D & = \{ 1,\ldots,\mu_{q} \} & d & = l_2 = {\sf CS}_k(\lambda / \mu)\text{ for }k \in D
\end{align*}
Either $a$ or $d$ equals $\rho(\lambda/\mu)$, and by Corollary~\ref{Cor:GUT}, it suffices for us to establish multiplicity in the case where
\[\lambda = (5^{l_1},2^{l_2})\text{ and }\mu = (4^{k_1},3^{k_2},2^{k_3}).\]

In $\lambda / \mu$ there are exactly two $D$-columns and one each of the $A,B$ and $C$-columns. In order to construct a ballot tableau, there is a unique way to fill the $A$, $B$ and $C$-columns (fill each column with $1$ through the length of that column). 

We consider four subcases: Case 1.1 corresponds to Theorem~\ref{thm:suffnotfathook} (\RomanNumeralCaps{1}) while cases 1.2, 1.3 and 1.4 correspond to Theorem~\ref{thm:suffnotfathook} (\RomanNumeralCaps{2}). 

\noindent \emph{Case 1.1 ($d\geq a$):} In this case $\rho = d$. By Lemma~\ref{Lemma:krowreduction}, we reduce to the case where $c=1$ by setting $k = c-1$ in the lemma. For filling 1, fill in the right $D$-column with $[2,d+2]\setminus\{a\}$ and the left $D$-column with $[d+1]\setminus \{b\}$ as in Figure~\ref{case1a(i)}. For filling 2, fill in the right $D$-column with $[2,d+2]\setminus \{b\}$ and the left $D$-column with $[d+1]\setminus \{a\}$ as in Figure~\ref{case1a(ii)}. The second filling is obtained from the first by swapping the red entries in the figure with the blue entries, maintaining their vertical order.

It remains to show these are ballot tableaux. The two tableaux are semistandard by construction. Let $u$ and $v$ be the column reading words of filling 1 and 2, respectively. To show that $u$ and $v$ are ballot, it suffices, by Lemma~\ref{Lemma:ballot}, to show that the initial factors of $u$ and $v$ that terminate at the underlined entries in Figure~\ref{case1a} are ballot.

In the first filling, Figure~\ref{case1a(i)}, since $c=1$, there are three entries equal to $1$ and two entries equal to $2$ before the underlined $2$ in $u$. There is one entry equal to $a$, and none equal to $a+1$, before the underlined $a+1$ in $u$. Finally, there are three entries equal to $b$, and two equal to $b+1$, before the underlined $b+1$ in u. Therefore this filling is ballot. By similar reasoning, we conclude that the second filling, Figure~\ref{case1a(ii)}, is ballot. As a result, $s_{\lambda/\mu}(x_1,\ldots,x_{\rho+2})$ is not multiplicity-free in this case.

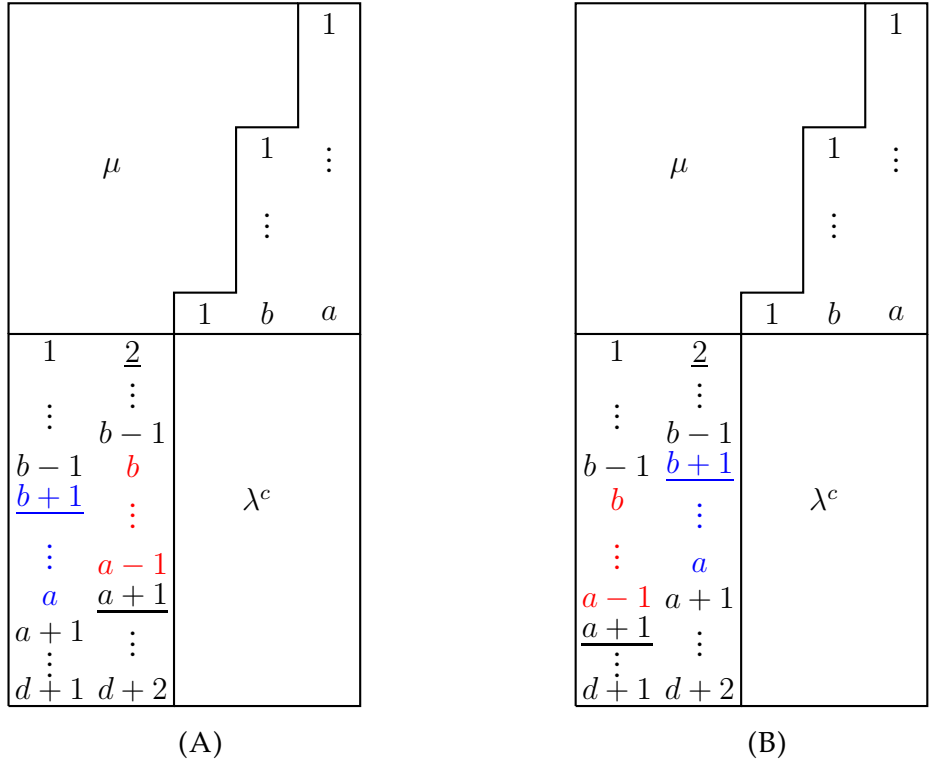
\begin{figure}
    \begin{center}
    \begin{subfigure}{0.45\textwidth}
    \hspace{1cm}
    \begin{tikzpicture}[scale = 0.55]
        \draw[black, thick] (2,-2) -- (6,-2) -- (6,7) -- (2,7) -- (2,-2);
        \draw[black, thick] (6,7) -- (10.5,7) -- (10.5,15) -- (9,15) -- (9,12) -- (7.5,12) -- (7.5,8) -- (6,8) -- (6,7);
        \draw[black, thick] (2,7) -- (2,15) -- (9,15);
        \draw[black, thick] (6,-2) -- (10.5,-2) -- (10.5,7);
        
        \draw (4.5,11) node{$\mu$};
        \draw (8,3) node{$\lambda^c$};
        
        \draw (5,6.5) node{\underline{$2$}};
        \draw (5,5.7) node{$\vdots$};
        \draw (5,4.6) node{$b-1$};
        \draw (5,3.8) node[text=red]{$b$};
        \draw (5,2.8) node[text=red]{$\vdots$};
        \draw (5,1.4) node[text=red]{$a-1$};
        \draw (5,0.6) node{\underline{$a+1$}};
        \draw (5,-0.3) node{$\vdots$};
        \draw (5,-1.6) node{$d+2$};
        
        \draw (3,6.6) node{$1$};
        \draw (3,5.2) node{$\vdots$};
        \draw (3,3.8) node{$b-1$};
        \draw (3,3) node[text=blue]{\underline{$b+1$}};
        \draw (3,1.8) node[text=blue]{$\vdots$};
        \draw (3,0.6) node[text=blue]{$a$};
        \draw (3,-0.2) node{$a+1$};
        \draw (3,-0.8) node{$\vdots$};
        \draw (3,-1.6) node{$d+1$};
        
        \draw (9.75, 14.5) node{$1$};
        \draw (9.75, 11.4) node{$\vdots$};
        \draw (9.75, 7.5) node{$a$};
        
        \draw (8.25, 11.5) node{$1$};
        \draw (8.25, 9.8) node{$\vdots$};
        \draw (8.25, 7.5) node{$b$};
        
        \draw (6.75, 7.5) node{$1$};
        
    \end{tikzpicture}
    \caption{}
    \label{case1a(i)}
    \end{subfigure}
    \begin{subfigure}{0.45\textwidth}
    \hspace{1cm}
        \begin{tikzpicture}[scale=0.55]
        \draw[black, thick] (2,-2) -- (6,-2) -- (6,7) -- (2,7) -- (2,-2);
        \draw[black, thick] (6,7) -- (10.5,7) -- (10.5,15) -- (9,15) -- (9,12) -- (7.5,12) -- (7.5,8) -- (6,8) -- (6,7);
        \draw[black, thick] (2,7) -- (2,15) -- (9,15);
        \draw[black, thick] (6,-2) -- (10.5,-2) -- (10.5,7);
        
        \draw (4.5,11) node{$\mu$};
        \draw (8,3) node{$\lambda^c$};
        
        \draw (5,6.5) node{\underline{$2$}};
        \draw (5,5.7) node{$\vdots$};
        \draw (5,4.6) node{$b-1$};
        \draw (5,3.8) node[text=blue]{\underline{$b+1$}};
        \draw (5,2.8) node[text=blue]{$\vdots$};
        \draw (5,1.4) node[text=blue]{$a$};
        \draw (5,0.6) node{$a+1$};
        \draw (5,-0.3) node{$\vdots$};
        \draw (5,-1.6) node{$d+2$};
        
        \draw (3,6.6) node{$1$};
        \draw (3,5.2) node{$\vdots$};
        \draw (3,3.8) node{$b-1$};
        \draw (3,3) node[text=red]{$b$};
        \draw (3,1.8) node[text=red]{$\vdots$};
        \draw (3,0.6) node[text=red]{$a-1$};
        \draw (3,-0.2) node{\underline{$a+1$}};
        \draw (3,-0.8) node{$\vdots$};
        \draw (3,-1.6) node{$d+1$};
        
        \draw (9.75, 14.5) node{$1$};
        \draw (9.75, 11.4) node{$\vdots$};
        \draw (9.75, 7.5) node{$a$};
        
        \draw (8.25, 11.5) node{$1$};
        \draw (8.25, 9.8) node{$\vdots$};
        \draw (8.25, 7.5) node{$b$};
        
        \draw (6.75, 7.5) node{$1$};
        \end{tikzpicture}
        \caption{}
        \label{case1a(ii)}
    \end{subfigure}
    \caption{Two fillings in Case 1.1}
    \label{case1a}    
    \end{center}
\end{figure}

\noindent \emph{Case 1.2 ($d\leq c+1$):} In this case $\rho = a$. By Lemma~\ref{Lemma:krowreduction}, we reduce to $d=2$. Consider the filling of the skew partition in Figure~\ref{case1b} and a second filling obtained by swapping the red $b+1$-entry with the blue $c+1$-entry in the figure. Both tableaux are semistandard. Let $u$ and $v$ be the column reading words of filling 1 and 2, respectively.  The initial factors of $u$ and $v$, ending at the $c$ in the $C$ column, are both ballot. There is one entry equal to $a$, zero entries equal to $a+1$, two entries equal to $b$, one entry equal to $b+1$, three entries equal to $c$, and two entries equal to $c+1$ in each of the two initial factors. As a result both $u$ and $v$ are ballot and thus $s_{\lambda/\mu}(x_1,\ldots,x_{\rho+1})$ is not multiplicity-free.

\begin{figure}
    \begin{center}
    \begin{tikzpicture}[scale = 0.6]
        \draw[black, thick] (2,0) -- (6,0) -- (6,2) -- (2,2) -- (2,0);
        \draw[black, thick] (6,2) -- (10.5,2) -- (10.5,10) -- (9,10) -- (9,7) -- (7.5,7) -- (7.5,5) -- (6,5) -- (6,2);
        
        \draw[black, thick] (2,2) -- (2,10) -- (9,10);
        \draw[black, thick] (6,0) -- (10.5,0) -- (10.5,2);
        
        \draw (4,7) node{$\mu$};
        \draw (8,1) node{$\lambda^c$};
        
        \draw (5,1.5) node[text=red]{\underline{$b+1$}};
        \draw (5,0.5) node{\underline{$a+1$}};
        \draw (3,1.5) node{$1$};
        \draw (3,0.5) node[text=blue]{\underline{$c+1$}};
        
        \draw (9.75, 9.5) node{$1$};
        \draw (9.75, 6.4) node{$\vdots$};
        \draw (9.75, 2.5) node{$a$};
        
        \draw (8.25, 6.5) node{$1$};
        \draw (8.25, 4.8) node{$\vdots$};
        \draw (8.25, 2.5) node{$b$};
        
        \draw (6.75, 4.5) node{$1$};
        \draw (6.75, 3.7) node{$\vdots$};
        \draw (6.75, 2.5) node{$c$};

    \end{tikzpicture}
\end{center}
    \caption{Filling in Case 1.2}
    \label{case1b}
\end{figure}
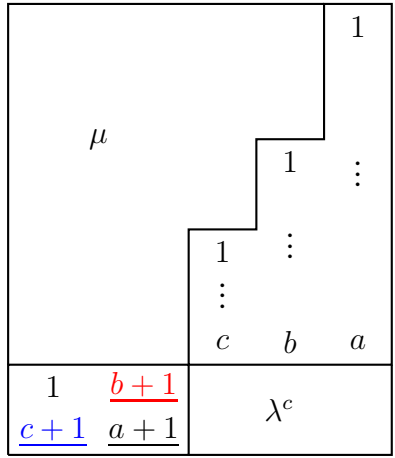

\noindent \emph{Case 1.3 ($c+1<d\leq b$):} In this case $\rho=a$. By Lemma~\ref{Lemma:krowreduction}, we reduce to $c=1$. For filling 1, fill in the right $D$-column with $[2,d]\cup \{a+1\}$ and the left $D$-column with $[d-1]\cup \{b+1\}$ as in Figure~\ref{case1c}. For filling 2, fill in the right $D$-column with $[2,d-1]\cup\{b+1,a+1\}$ and the left $D$-column with $[d]$.
Filling 2 arises from filling 1 by swapping the blue $b+1$ with the red $d$ in the figure. It is again clear that the two tableaux are semistandard. Let $u$ and $v$ be the column reading words of filling 1 and 2, respectively. The initial factors of $u$ and $v$, ending at the $1$ in the $C$-column, are both ballot. There is one entry equal to $a$, zero entries equal to $a+1$, two entries equal to $b$, and one entry equal to $b+1$ in each of the two initial factors. This, combined with Lemma~\ref{Lemma:ballot}, implies that $u$ and $v$ are ballot. Thus $s_{\lambda/\mu}(x_1,\ldots,x_{\rho+1})$ is not multiplicity-free. 

\begin{figure}
    \begin{center}
    \begin{tikzpicture}[scale = 0.6]
        \draw[black, thick] (2,0) -- (6,0) -- (6,4) -- (2,4) -- (2,0);
        \draw[black, thick] (6,4) -- (10.5,4) -- (10.5,12) -- (9,12) -- (9,9) -- (7.5,9) -- (7.5,5) -- (6,5) -- (6,4);
        \draw[black, thick] (2,4) -- (2,12) -- (9,12);
        \draw[black, thick] (6,0) -- (10.5,0) -- (10.5,4);
        
        \draw (4.2,8.8) node{$\mu$};
        \draw (8.4,2) node{$\lambda^c$};
        
        \draw (5,3.5) node{\underline{$2$}};
        \draw (5,2.7) node{$\vdots$};
        \draw (5,1.5) node[text=red]{$d$};
        \draw (5,0.5) node{\underline{$a+1$}};
        \draw (3,3.5) node{$1$};
        \draw (3,2.7) node{$\vdots$};
        \draw (3,1.5) node{$d-1$};
        \draw (3,0.5) node[text=blue]{\underline{$b+1$}};
        
        \draw (9.75, 11.5) node{$1$};
        \draw (9.75, 8.4) node{$\vdots$};
        \draw (9.75, 4.5) node{$a$};
        
        \draw (8.25, 8.5) node{$1$};
        \draw (8.25, 6.8) node{$\vdots$};
        \draw (8.25, 4.5) node{$b$};
        
        \draw (6.75, 4.5) node{$1$};
        
    \end{tikzpicture}
\end{center}
    \caption{Filling in Case 1.3}
    \label{case1c}
\end{figure}
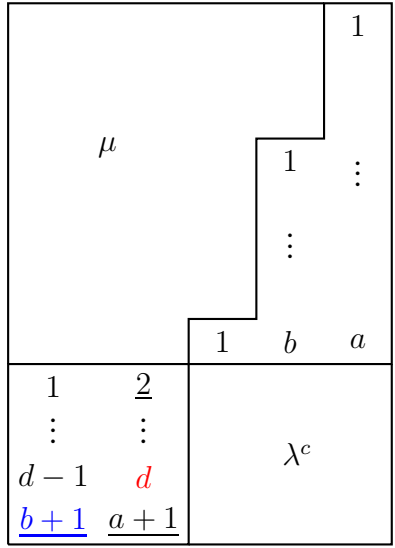

\noindent \emph{Case 1.4 ($b<d<a$):} In this case $\rho = a$. By Lemma\ref{Lemma:krowreduction}, we reduce to $c=1$. For filling 1, fill in the right $D$-column with $[2,d]\cup\{a+1\}$ and the left $D$-column with $[d+1]\setminus\{b\}$ as in Figure~\ref{case1d(i)}. For filling 2, fill in the right $D$-column with $[2,d+1]\cup\{a+1\}\setminus\{b\}$ and the left $D$-column with $[d]$ as in Figure~\ref{case1d(ii)}. The second fillings can be obtained from first by swapping the red entries with the blue entries, maintaining their vertical order. Both tableaux are semistandard. The ballot condition follows by a similar argument to Case 1.1. Therefore $s_{\lambda/\mu}(x_1,\ldots,x_{\rho+1})$ is not multiplicity-free.

\begin{figure}
    \begin{center}
    \begin{subfigure}{0.35\textwidth}
    \hspace{0.4cm}
    \begin{tikzpicture}[scale = 0.55]
    
        \draw[black, thick] (2,0) -- (6,0) -- (6,7) -- (2,7) -- (2,0);
        \draw[black, thick] (6,7) -- (10.5,7) -- (10.5,15) -- (9,15) -- (9,11) -- (7.5,11) -- (7.5,8) -- (6,8) -- (6,7);
        
        \draw[black, thick] (2,7) -- (2,15) -- (9,15);
        \draw[black, thick] (6,0) -- (10.5,0) -- (10.5,7);
        
        \draw (4.5,12) node{$\mu$};
        \draw (8.5, 3.5) node{$\lambda^c$};
        
        \draw (5,6.6) node{\underline{$2$}};
        \draw (5,5.7) node{$\vdots$};
        \draw (5,4.6) node{$b-1$};
        \draw (5,3.8) node[text=red]{$b$};
        \draw (5,2.8) node[text=red]{$\vdots$};
        \draw (5,1.3) node[text=red]{$d$};
        \draw (5,0.5) node{\underline{$a+1$}};
        
        \draw (3,6.6) node{$1$};
        \draw (3,5.2) node{$\vdots$};
        \draw (3,3.8) node{$b-1$};
        \draw (3,3) node[text=blue]{\underline{$b+1$}};
        \draw (3,1.8) node[text=blue]{$\vdots$};
        \draw (3,0.5) node[text=blue]{$d+1$};
        
        \draw (9.75, 14.5) node{$1$};
        \draw (9.75, 11.4) node{$\vdots$};
        \draw (9.75, 7.5) node{$a$};
        
        \draw (8.25, 10.5) node{$1$};
        \draw (8.25, 9.3) node{$\vdots$};
        \draw (8.25, 7.5) node{$b$};
        
        \draw (6.75, 7.5) node{$1$};

    \end{tikzpicture}
    \caption{}
    \label{case1d(i)}
    \end{subfigure}
    \hspace{0.5cm}
    \begin{subfigure}{0.35\textwidth}
    \hspace{0.5cm}
        \begin{tikzpicture}[scale=0.55]
        \draw[black, thick] (2,0) -- (6,0) -- (6,7) -- (2,7) -- (2,0);
        \draw[black, thick] (6,7) -- (10.5,7) -- (10.5,15) -- (9,15) -- (9,11) -- (7.5,11) -- (7.5,8) -- (6,8) -- (6,7);
        
        \draw[black, thick] (2,7) -- (2,15) -- (9,15);
        \draw[black, thick] (6,0) -- (10.5,0) -- (10.5,7);
        
        \draw (4.5,12) node{$\mu$};
        \draw (8.5, 3.5) node{$\lambda^c$};
        
        \draw (5,6.6) node{\underline{$2$}};
        \draw (5,5.7) node{$\vdots$};
        \draw (5,4.6) node{$b-1$};
        \draw (5,3.8) node[text=blue]{\underline{$b+1$}};
        \draw (5,2.8) node[text=blue]{$\vdots$};
        \draw (5,1.3) node[text=blue]{$d+1$};
        \draw (5,0.5) node{\underline{$a+1$}};
        
        \draw (3,6.6) node{$1$};
        \draw (3,5.2) node{$\vdots$};
        \draw (3,3.8) node{$b-1$};
        \draw (3,3) node[text=red]{$b$};
        \draw (3,1.8) node[text=red]{$\vdots$};
        \draw (3,0.5) node[text=red]{$d$};
        \draw (9.75, 14.5) node{$1$};
        \draw (9.75, 11.4) node{$\vdots$};
        \draw (9.75, 7.5) node{$a$};
        
        \draw (8.25, 10.5) node{$1$};
        \draw (8.25, 9.3) node{$\vdots$};
        \draw (8.25, 7.5) node{$b$};
        
        \draw (6.75, 7.5) node{$1$};
        
        \end{tikzpicture}
        \caption{}
        \label{case1d(ii)}
    \end{subfigure}
    \caption{Two fillings in Case 1.4}
    \label{case1d}    
    \end{center}
\end{figure}
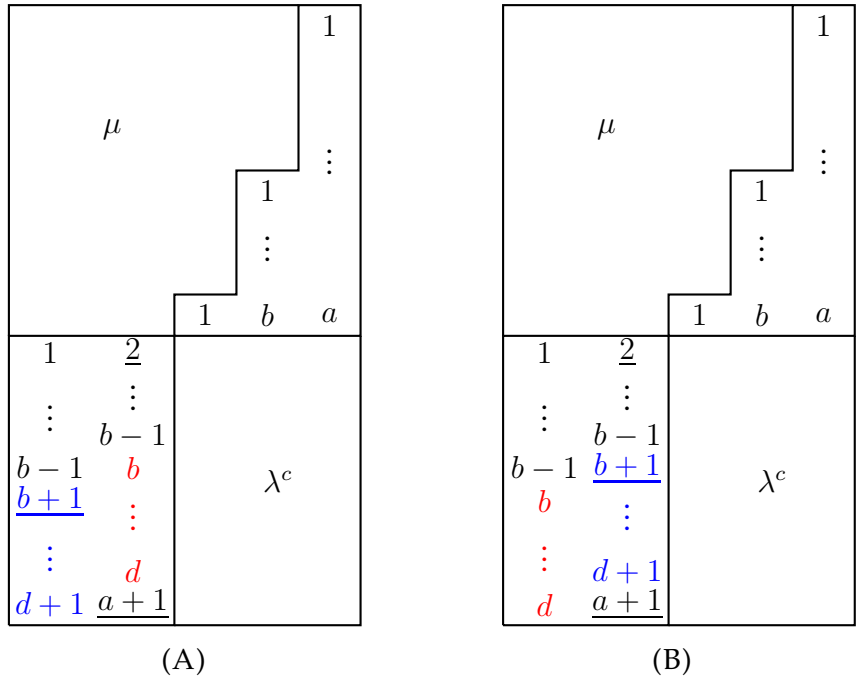

\noindent \emph{Case 2 ($\lambda_2 = \mu_1$):} As illustrated by Figure~\ref{case2}, set 
\begin{align*}
A & = \{ \mu_1 + 1,\ldots,\lambda_1 \} & a & = l_1 = {\sf CS}_k(\lambda / \mu)\text{ for }k \in A \\
B & = \{ \mu_{2} + 1,\ldots,\mu_{1} \} & b & = l_1+l_2-k_1 = {\sf CS}_k(\lambda / \mu)\text{ for }k \in B \\
C & = \{ \mu_{3} + 1,\ldots,\mu_{2} \} & c & = l_1+l_2-k_1-k_2 = {\sf CS}_k(\lambda / \mu)\text{ for }k \in C \\
D & = \{ \mu_{4} + 1,\ldots,\mu_{3} \} & d & = l_1+l_2-k_1-k_2-k_3 = {\sf CS}_k(\lambda / \mu)\text{ for }k \in D
\end{align*}
where $\mu_{4}=0$ if $q=3$.

By Corollary~\ref{Cor:GUT}, it suffices to consider the case where $\lambda / \mu$ consists of exactly two $A$-columns and one each of the $B$, $C$, and $D$-columns. Thus we consider
\[\lambda = (5^{l_1},2^{l_2})\ \text{and } \mu = (3^{k_1},2^{k_2},1^{k_3})\]

We divide into four subcases: Case 2.1 corresponds to Theorem~\ref{thm:suffnotfathook} (\RomanNumeralCaps{3}) while cases 2.2, 2.3 and 2.4 correspond to Theorem~\ref{thm:suffnotfathook} (\RomanNumeralCaps{4}). In all subcases we construct two distinct ballot fillings of $\lambda/\mu$ with the same content.

\begin{figure}
    \begin{center}
        \begin{tikzpicture}[scale = 1.3]
            \draw[black, thick] (0,0) -- (4,0) -- (4,4) -- (0,4) -- (0,0);
        \draw[black, thick] (0,1.4) -- (1,1.4) -- (1,2) -- (2,2) -- (2,2.7) -- (2.7,2.7) -- (2.7,3.4) -- (3.4,3.4) -- (3.4,4);
        \draw[black, thick] (3.4,0) -- (3.4,1.4) -- (4,1.4);
        \draw (0.75,3) node{$\mu$};
        \draw (3.7, 0.5) node{$\lambda^c$};
        \draw[gray, dashed] (3.4,1) -- (3.4,3.4);
        \draw[gray, dashed] (2.7,2.7) -- (2.7,0);
        \draw[gray, dashed] (2,2) -- (2,0);
        \draw[gray, dashed] (1,1.4) -- (1,0);
        \draw (3.7,2.5) node{$a$};
        \draw (3,1.9) node{$b$};
        \draw (2.35,1.5) node{$c$};
        \draw (1.5,1) node{$d$};
        
        \end{tikzpicture}
      \caption{Case 2 in Theorem~\ref{thm:suffnotfathook}, with column lengths $a,b,c,$ and $d$.}
    \label{case2}  
    \end{center}

\end{figure}
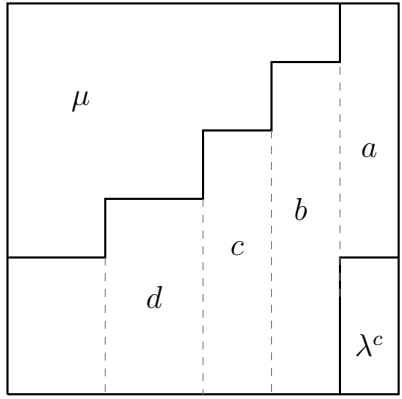

\noindent \emph{Case 2.1 ($a\geq b$):} In this case $\rho = a$. By Lemma~\ref{Lemma:krowreduction}, we reduce to $d=1$. For filling 1, fill the $B$-column with $[b-2]\cup\{a+1,a+2\}$, the $C$-column with $[c-1]\cup\{a+1\}$, and the $D$-column with $\{b-1\}$ as shown in Figure~\ref{case2a}. Filling 2 is defined to be the result of swapping the blue $b-1$ and red $c-1$ in Figure~\ref{case2a}. Both tableaux are semistandard. Let $u$ and $v$ be the column reading words of filling 1 and 2, respectively. By Lemma~\ref{Lemma:ballot}, in order to verify that $u$ and $v$ are ballot, it suffices to check that the initial factors of $u$ and $v$ that terminate at the underlined entries in Figure~\ref{case2a} are ballot. The following remarks hold for both $u$ and $v$. There are two entries equal to $a$ and two entries equal to $a+1$. Both $a$ entries appear before the $a+1$ entries. There are four entries equal to $c-2$, and three equal to $c-1$, before the underlined $c-1$. There are three or four entries equal to $b-2$, and two entries equal to $b-1$ before the underlined $b-1$. Thus $u$ and $v$ are ballot and $s_{\lambda/\mu}(x_1,\ldots,x_{\rho+2})$ is not multiplicity-free. 

\begin{figure}
    \begin{center}
    \begin{tikzpicture}[scale = 1.2]
        
        \draw[black, thick] (4,1.4) -- (4,0) -- (1,0) -- (1,0.5) -- (2,0.5) -- (2,2.25) -- (3,2.25) -- (3,3) -- (4,3);
        \draw[black, thick] (4,1.4) -- (5.5,1.4) -- (5.5,5.4) -- (4,5.4) -- (4,3);
        
        \draw[black, thick] (1,0.5) -- (1,5.4) -- (4,5.4);
        \draw[black, thick] (4,0) -- (5.5,0) -- (5.5,1.4);
        
        \draw (2,4) node{$\mu$};
        \draw (4.74, 0.65) node{$\lambda^c$};
        
        \draw (3.5,2.75) node{$1$};
        \draw (3.5,2) node{$\vdots$};
        \draw (3.5,1) node{$b-2$};
        \draw (3.5, 0.6) node{\underline{$a+1$}};
        \draw (3.5, 0.2) node{$a+2$};
        
        \draw (2.5, 2) node{$1$};
        \draw (2.5,1.4) node{$\vdots$};
        \draw (2.5, 0.6) node[text=red]{\underline{$c-1$}};
        \draw (2.5,0.2) node{\underline{$a+1$}};
        
        \draw (1.5,0.2) node[text=blue]{\underline{$b-1$}};
        
        \draw (5.1, 5.1) node{$1$};
        \draw (5.1, 3.6) node{$\vdots$};
        \draw (5.1, 1.7) node{$a$};
        
        \draw (4.4, 5.1) node{$1$};
        \draw (4.4, 3.6) node{$\vdots$};
        \draw (4.4, 1.7) node{$a$};
        
    \end{tikzpicture}
\end{center}
    \caption{Filling in Case 2.1}
    \label{case2a}
\end{figure}

\noindent \emph{Case 2.2 ($b>a\geq c$):} In this case $\rho=b$. By Lemma~\ref{Lemma:krowreduction}, we reduce to $d=1$. For filling 1, fill the $B$-column with $[b+1]\setminus\{a\}$, the $C$-column by $[c-2]\cup\{a,a+1\}$, and the $D$-column with $\{c-1\}$ as shown in Figure~\ref{case2b}. Filling 2 is defined to be the result of swapping the red $a$ and blue $c-1$ in Figure~\ref{case2b}. Since $b>a$, both tableaux are semistandard. Let $u$ and $v$ be the column reading words of filling 1 and 2, respectively. By Lemma~\ref{Lemma:ballot}, in order to verify that $u$ and $v$ are ballot, it suffices to check that the initial factors of $u$ and $v$ that terminate at the underlined entries in Figure~\ref{case2b} are ballot. The following remarks hold for both $u$ and $v$. There are three entries equal to $a$, and two equal to $a+1$, in $u$ and $v$, with at least two $a$ entries appearing prior to the $a+1$ entries. There are three entries equal to $a-1$, and two entries equal to $a$, before the underlined $a$. There are four entries equal to $c-2$, and three entries equal to $c-1$, before the underlined $c-1$. Hence $u$ and $v$ are ballot and  $s_{\lambda/\mu}(x_1,\ldots,x_{\rho+1})$ is not multiplicity-free.

\begin{figure}
    \begin{center}
    \begin{tikzpicture}[scale = 1.2]
        
        \draw[black, thick] (4,2.4) -- (4,0) -- (1,0) -- (1,0.45) -- (2,0.45) -- (2,2.5) -- (3,2.5) -- (3,4) -- (4,4) ;
        \draw[black, thick] (4,2.4) -- (5.5,2.4) -- (5.5,5.4) -- (4,5.4) -- (4,4);
        
        \draw[black, thick] (1,0.45) -- (1,5.4) -- (4,5.4);
        \draw[black, thick] (4,0) -- (5.5,0) -- (5.5,2.4);
        
        \draw (2,4) node{$\mu$};
        \draw (4.74,1.2) node{$\lambda^c$};
        
        \draw (3.5,3.75) node{$1$};
        \draw (3.5,2.8) node{$\vdots$};
        \draw (3.5,1.6) node{$a-1$};
        \draw (3.5,1.2) node{\underline{$a+1$}};
        \draw (3.5, 0.8) node{$\vdots$};
        \draw (3.5, 0.2) node{$b+1$};
        
        \draw (2.5, 2.25) node{$1$};
        \draw (2.5,1.7) node{$\vdots$};
        \draw (2.5,1) node{$c-2$};
        \draw (2.5, 0.6) node[text=red]{\underline{$a$}};
        \draw (2.5,0.2) node{\underline{$a+1$}};
        
        \draw (1.5,0.2) node[text=blue]{\underline{$c-1$}};
        
        \draw (5.1, 5.1) node{$1$};
        \draw (5.1, 4) node{$\vdots$};
        \draw (5.1, 2.6) node{$a$};
        
        \draw (4.4, 5.1) node{$1$};
        \draw (4.4, 4) node{$\vdots$};
        \draw (4.4, 2.6) node{$a$};

    \end{tikzpicture}
\end{center}
    \caption{Filling in Case 2.2}
    \label{case2b}
\end{figure}

\noindent \emph{Case 2.3 ($c>a>d$):} By Lemma~\ref{Lemma:krowreduction} we reduce to $d = 1$. For filling 1, fill the $B$-column with $[b+1]\setminus \{a\}$, the $C$-column with $[c+1]\setminus \{a-1\}$, and the $D$-column with $\{a-1\}$ as shown in Figure~\ref{case2c}. Filling 2 is defined to be the result of swapping the red $a$ in the $C$-column with the blue $a-1$ entry in the $D$-column in Figure~\ref{case2c}. Both tableaux are semistandard. By similar reasoning as in Case 2.2, both tableaux are ballot. As a result, $s_{\lambda/\mu}(x_1,\ldots,x_{\rho+1})$ is not multiplicity-free.

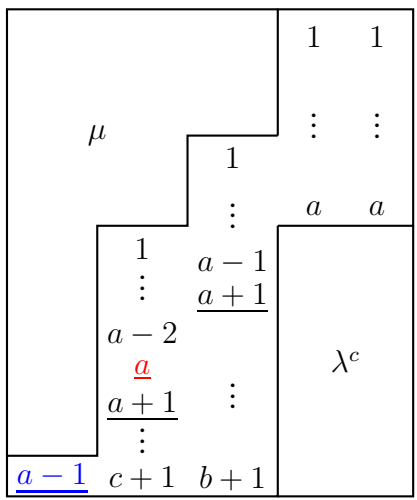
\begin{figure}
    \begin{center}
    \begin{tikzpicture}[scale = 1.2]
        
        \draw[black, thick] (4,3) -- (4,0) -- (1,0) -- (1,0.45) -- (2,0.45) -- (2,3) -- (3,3) -- (3,4) -- (4,4);
        \draw[black, thick] (4,3) -- (5.5,3) -- (5.5,5.4) -- (4,5.4) -- (4,4);
        
        \draw[black, thick] (1,0.45) -- (1,5.4) -- (4,5.4);
        \draw[black, thick] (4,0) -- (5.5,0) -- (5.5,3);
        
        \draw (2,4) node{$\mu$};
        \draw (4.75, 1.5) node{$\lambda^c$};
        
        \draw (3.5,3.75) node{$1$};
        \draw (3.5,3.2) node{$\vdots$};
        \draw (3.5,2.6) node{$a-1$};
        \draw (3.5,2.2) node{\underline{$a+1$}};
        \draw (3.5, 1.2) node{$\vdots$};
        \draw (3.5, 0.2) node{$b+1$};
        
        \draw (2.5, 2.75) node{$1$};
        \draw (2.5,2.4) node{$\vdots$};
        \draw (2.5,1.8) node{$a-2$};
        \draw (2.5, 1.4) node[text=red]{\underline{$a$}};
        \draw (2.5,1) node{\underline{$a+1$}};
        \draw (2.5, 0.7) node{$\vdots$};
        \draw (2.5,0.2) node{$c+1$};
        
        \draw (1.5,0.2) node[text=blue]{\underline{$a-1$}};
        
        \draw (5.1, 5.1) node{$1$};
        \draw (5.1, 4.2) node{$\vdots$};
        \draw (5.1, 3.2) node{$a$};
        
        \draw (4.4, 5.1) node{$1$};
        \draw (4.4, 4.2) node{$\vdots$};
        \draw (4.4, 3.2) node{$a$};

    \end{tikzpicture}
\end{center}
    \caption{Filling in Case 2.3}
    \label{case2c}
\end{figure}

\noindent \emph{Case 2.4 ($d\geq a$):} By Lemma~\ref{Lemma:krowreduction} we reduce to $a=2$. For filling 1, fill the $B$-column with $[b+1]\setminus\{2\}$, the $C$-column with $[c+1]\setminus\{2\}$, and the $D$-column with $[2,d+1]$ as shown in Figure~\ref{case2d}. Obtain filling 2 by swapping the blue $1$ in the $D$-column with the red $2$ in the $C$-column in Figure~\ref{case2d}. Both tableaux are semistandard. Their reduced words are easily verified to be ballot by considering the initial factors that terminate at the underlined entries in Figure~\ref{case2d}, and then applying Lemma~\ref{Lemma:ballot}. Therefore $s_{\lambda/\mu}(x_1,\ldots,x_{\rho+1})$ is not multiplicity-free.

\begin{figure}
    \begin{center}
    \begin{tikzpicture}[scale = 1.2]
        
        \draw[black, thick] (4,3.55) -- (4,0) -- (1,0) -- (1,2) -- (2,2) -- (2,3) -- (3,3) -- (3,4) -- (4,4);
        \draw[black, thick] (4,4) -- (4,4.5) -- (5.2,4.5) -- (5.2,3.55) -- (4,3.55);
        
        \draw[black, thick] (1,2) -- (1,4.5) -- (4,4.5);
        \draw[black, thick] (4,0) -- (5.2,0) -- (5.2,3.55);
        
        \draw (2,4.1) node{$\mu$};
        \draw (4.6,1.8) node{$\lambda^c$};
        
        \draw (3.5,3.75) node{$1$};
        \draw (3.5, 3.35) node{\underline{$3$}};
        \draw (3.5,1.85) node{$\vdots$};
        \draw (3.5, 0.2) node{$b+1$};
        
        \draw (2.5, 2.75) node[text=red]{\underline{$2$}};
        \draw (2.5,2.35) node{\underline{$3$}};
        \draw (2.5,1.35) node{$\vdots$};
        \draw (2.5,0.2) node{$c+1$};
        
        \draw (1.5, 1.75) node[text=blue]{$1$};
        \draw (1.5,1.35) node{\underline{$3$}};
        \draw (1.5, 0.9) node{$\vdots$};
        \draw (1.5,0.2) node{$d+1$};
        
        \draw (4.85,4.25) node{$1$};
        \draw (4.85,3.75) node{$2$};
        
        \draw (4.35,4.25) node{$1$};
        \draw (4.35,3.75) node{$2$};
        
    \end{tikzpicture}
\end{center}
    \caption{Filling in Case 2.4}
    \label{case2d}
\end{figure}
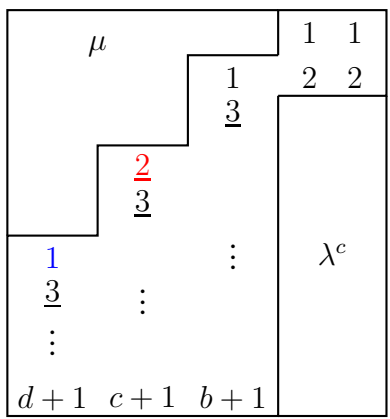

\noindent \emph{Case 3 ($\mu_q < \lambda_2 < \mu_1$):} This case corresponds to Theorem~\ref{thm:suffnotfathook} (\RomanNumeralCaps{5}). Because $\lambda_2 > \mu_q$, there exists at least two columns of different length that end in the last row of $\lambda$; namely those starting in row $(\sum_{i=1}^{q-1} k_i)  + 1$ and those starting in row $(\sum_{i=1}^{q} k_i)  + 1$. Because $\lambda_2 < \mu_1$, there exist at least two different length columns that end in row $l_1$; namely those starting in row $1$ and those starting in row $k_1+1$. All four of these columns start in different rows, since $q\geq 3$. We now apply Proposition~\ref{prop:4col} to these four columns to conclude $s_{\lambda/\mu}(x_1,\dots, x_{\rho+1})$ is not multiplicity-free. 
\qed

\subsection{\texorpdfstring{$\lambda^\vee$}{} is a rectangle of shortness at least \texorpdfstring{3}{} and \texorpdfstring{$\mu$}{} is a fat hook of shortness at least \texorpdfstring{2}{}}

\begin{theorem}
\label{thm:suffshort3}
If $s_{\lambda/\mu}(x_1,\ldots,x_n)$ is a basic, n-sharp, tight skew Schur polynomial such that $\lambda^\vee$ is a rectangle of shortness at least $3$ and $\mu$ is a fat hook of shortness at least $2$, then $s_{\lambda/\mu}(x_1,\ldots,x_n)$ is not multiplicity-free if one of the following conditions holds:
\begin{enumerate}[label=(\Roman{*})]
    \item $\lambda_2 = \mu_1, l_2 > k_1$ and $n\geq \rho+2$
    \item $\lambda_2 = \mu_1, l_2 \leq k_1$ and $n\geq \rho+3$
    \item $\mu_1 > \lambda_2 > \mu_2, k_1 \geq l_2$ and $n \geq \rho+2$
    \item $\mu_1 > \lambda_2 > \mu_2, l_2 > k_1$
    \item $\mu_2 = \lambda_2, l_2 \geq l_1$ and $n\geq \rho+3$
    \item $\mu_2 = \lambda_2, l_1 > l_2$ and $n\geq \rho+2$
\end{enumerate}
\end{theorem}

Here we omit the cases where $\lambda_2<\mu_2$ or $\lambda_2>\mu_1$ since the skew partition $\lambda/\mu$ would be not basic if $\lambda_2<\mu_2$ and $l_1 = \ell(\mu)$, and not tight otherwise. 

\noindent \emph{Proof.} Let $c_1 = \lambda_1$. Set $c_i = \max(\{k : U(k) \neq U(c_{i-1})\text{ or }L(k) \neq L(c_{i-1})\})$ for $i=2,3,4$. Since $\lambda$ and $\mu$ are both fat hooks, either $c_4$ does not exist and $L(c_3)=L(1)$ with $U(c_3)=U(1)$, or $L(c_4)=L(1),U(c_4)=U(1)$. Finally, set
\begin{align*}
A & = \{ k : L(k) = L(c_1)\textit{ and }U(k) = U(c_1) \}, & a & = {\sf CS}_k(\lambda / \mu)\text{ for }k \in A, \\
B & = \{ k : L(k) = L(c_2)\textit{ and }U(k) = U(c_2) \}, & b & = {\sf CS}_k(\lambda / \mu)\text{ for }k \in B, \\
C & = \{ k : L(k) = L(c_3)\textit{ and }U(k) = U(c_3) \}, & c & = {\sf CS}_k(\lambda / \mu)\text{ for }k \in C,
\end{align*}
and if $c_4$ exists set
\begin{align*}
D & = \{ k : L(k) = L(c_4)\textit{ and }U(k) = U(c_4) \} & d & = {\sf CS}_k(\lambda / \mu)\text{ for }k \in D.
\end{align*}
We divide into $4$ cases. In each case, we construct two ballot tableaux with the same content to conclude that the skew Schur polynomial is not multiplicity-free.

\noindent \emph{Case 1 ($\lambda_2 = \mu_1$)}: In this case $c_4$ is not defined, and $a = l_1$, $b = \ell(\lambda)-k_1$, and $c = \ell(\lambda)-\ell(\mu)$. By Corollary~\ref{Cor:GUT}, it is enough to show multiplicity when there are three $A$-columns, two $B$-columns, and two $C$-columns. That is, when \[ \lambda = (7^{l_1}, 4^{l_2})\text{ and }\mu=(4^{k_1},2^{k_2}).\] 
Now we consider two sub-cases: Case 1.1 corresponds to Theorem~\ref{thm:suffshort3} (\RomanNumeralCaps{1}) while Case 1.2 corresponds to Theorem~\ref{thm:suffshort3} (\RomanNumeralCaps{2}).

\begin{figure}
    \begin{center}
    \begin{tikzpicture}
        \draw[black, thick] (0,0) -- (4,0) -- (4,4) -- (0,4) -- (0,0);
        \draw[black, thick] (0,1) -- (1.5,1) -- (1.5,2.5) -- (3,2.5) -- (3,4);
        \draw[black, thick] (3,0) -- (3,2) -- (4,2);
        \draw[gray, dashed] (3,2) -- (3,2.5);
        \draw[gray, dashed] (1.5,1) -- (3,1);
        \draw[gray, dashed] (1.5,1) -- (1.5,0);
        \draw (0.75,3) node{$\mu$};
        \draw (3.5, 1) node{$\lambda^c$};
        \draw (3.5,3) node{$a$};
        \draw (2.25, 1.5) node{$b$};
        \draw (0.75, 0.5) node{$c$};

    \end{tikzpicture}
\end{center}
    \caption{Case 1}
    \label{Case1short3}
\end{figure}

\begin{figure}
    \begin{center}
    \begin{tikzpicture}[scale =2]
        \draw[black, thick] (0,0) -- (4,0) -- (4,4) -- (0,4) -- (0,0);
        \draw[black, thick] (0,1.5) -- (2,1.5) -- (2,2.5) -- (3,2.5) -- (3,4);
        \draw[black, thick] (3,0) -- (3,2) -- (4,2);
        \draw (0.75,3) node{\large$\mu$};
        \draw (3.5, 1.2) node{\large$\lambda^c$};
        \draw (3.2,3.8) node{$1$};
        \draw (3.2,2.2) node{$a$};
        \draw (3.2,3) node{$\vdots$};
        \draw (3.5,3.8) node{$1$};
        \draw (3.5,2.2) node{$a$};
        \draw (3.5,3) node{$\vdots$};
        \draw (3.8,3.8) node{$1$};
        \draw (3.8,2.2) node{$a$};
        \draw (3.8,3) node{$\vdots$};
        
        \draw (2.8,2.3) node{$1$};
        \draw (2.8,2) node{$\vdots$};
        \draw (2.8,1.6) node{\footnotesize$a$-$2$};
        \draw (2.8, 1.3) node{\underline{\footnotesize$a$+$1$}};
        \draw (2.8,0.8) node{$\vdots$};
        \draw (2.8, 0.15) node{\footnotesize$b$+$2$};
        
        \draw (2.3,2.3) node{$1$};
        \draw (2.3,2) node{$\vdots$};
        \draw (2.3,1.6) node{\footnotesize$a$-$2$};
        \draw (2.3, 1.3) node{\footnotesize$a$-$1$};
        \draw (2.3, 1) node{\underline{\footnotesize$a$+$1$}};
        \draw (2.3,0.7) node{$\vdots$};
        \draw (2.3, 0.15) node{\footnotesize$b$+$1$};

        \draw (1.5, 1.3) node{$\vdots$};
        \draw (1.5, 1) node{\tiny$\min\{c$-$2,a$-$3\}$};
        \draw (1.5, 0.8) node[text=red]{\underline{\footnotesize$a$-$1$}};
        \draw (1.5,0.6) node{\underline{\footnotesize$a$+$1$}};
        \draw (1.5,0.4) node{$\vdots$};
        \draw (1.5, 0.15) node{\tiny$\max\{c$+$2,a$+$1\}$};
        
        \draw (0.5,1.4) node{$1$};
        \draw (0.5, 1.2) node{$\vdots$};
        \draw (0.5, 0.8) node{\tiny$\min\{c$-$1,a$-$2\}$};
        \draw (0.5, 0.6) node[text=blue]{\underline{\footnotesize$a$}};
        \draw (0.5,0.4) node{$\vdots$};
        \draw (0.5, 0.15) node{\tiny$\max\{c$+$1,a\}$};

    \end{tikzpicture}
\end{center}
    \caption{Ballot filling in Case 1.1}
    \label{Case1:a<b}
\end{figure}

\noindent \emph{Case 1.1 ($l_2 > k_1$):} In this case, $a<b$ and $\rho = b$. Let $T_1$ be the filling of $\lambda/\mu$ shown in Figure~\ref{Case1:a<b}. It is constructed by filling each $A$-column with $[a]$, the right $B$-column with $[b+2]\setminus \{a-1,a\}$, and the left $B$-column with $[b+1]\setminus \{a\}$. The right $C$-column is filled with $[\min\{c-2,a-3\}] \cup \{a-1\} \cup \{a+1,\ldots \max\{c+2,a+1\}\}$. When $\min\{c-2,a-3\}= 0$, the top entry in this column is $a-1$. If $\min\{c-2,a-3\} = c-2$, then $\max\{c+2,a+1\} = a+1$ and thus there are $c$ entries in this column. If $\min\{c-2,a-3\} = a-3$, then $\max\{c+2,a+1\} = c+2$ and again there are $c$ entries in this column. For the left $C$-column, fill it with $[\min\{c-1,a-2\}]\cup \{a,\ldots, \max\{c+1,a\}\}$. By the same reasoning, there are $c$ entries in this column. Let $T_2$ be the filling obtained by swapping the blue $a-1$ entry and the red $a$ entry within the $C$-columns in Figure~\ref{Case1:a<b}. 

\noindent ($T_1, T_2$ are semistandard): By construction, in both $T_1$ and $T_2$, each column is strictly increasing downwards while the rows are weakly increasing within the $A$, $B$, and $C$-columns. It suffices to check that the entries in the rightmost $B$ and $C$ columns, in both $T_1$ and $T_2$, are less than or equal to their right neighbors. If $a=k_1$, then the entries in the rightmost $B$ column have no right neighbors. If $a > k_1$, then the left neighbor of the bottom box in the leftmost $A$-column contains $a - k_1 \leq a - 2$, where the inequality follows from $\mu$ having shortness at least $2$. Thus in both $T_1$ and $T_2$ every entry in the rightmost $B$-column is less than its right neighbor (if their right neighbor exists). 

By hypothesis $b>a$, and $b \geq c+2$ since $\mu$ has shortness at least $2$. Thus 
\begin{equation}
\label{eq:blessthanmax}
b \geq \max\{c+2,a+1\}.
\end{equation}
Let $X_k$ be the $k$\textsuperscript{th} box from the bottom in the leftmost $B$-column. Then $X_k$ has a left neighbor if and only if $k \leq c$. If $k \leq b-a + 1$, then $X_k$ contains $b-k+2$, and the left neighbor of $X_k$ contains at most $\max\{c+2,a+1\}-k+1 \leq b-k+2$ (the inequality follows by \eqref{eq:blessthanmax}). If $b-a+1 < k \leq c$, then $X_k$ contains $b-k+1$, and the left neighbor of $X_k$ contains at most $\max\{c+2,a+1\}-k+1 \leq b-k+1$ (the inequality follows by \eqref{eq:blessthanmax}). We conclude that $T_1$ and $T_2$ are semistandard.

\noindent ($T_1, T_2$ are ballot): Let $u$ and $v$ be the column reading words of $T_1$ and $T_2$, respectively. By Lemma~\ref{Lemma:ballot}, the ballot condition only needs to be checked at the initial factors of $u$ and $v$ terminating at the underlined entries in Figure~\ref{Case1:a<b}. The following remarks hold for both $u$ and $v$. There are three entries equal to $a$ that appear before the three underlined $a+1$. There are at least five entries equal to $a-2$, and at most four equal to $a-1$, before the underlined $a-1$. There are at least four entries equal to $a-1$, and three entries equal to $a$, before the underlined $a$. Thus both $T_1$ and $T_2$ are ballot.

Therefore $s_{\lambda/\mu}(x_1,\ldots,x_{n})$ is not multiplicity-free. 

\begin{figure}
    \begin{center}
    \begin{tikzpicture}[scale =2]
        \draw[black, thick] (0,0) -- (4,0) -- (4,4) -- (0,4) -- (0,0);
        \draw[black, thick] (0,1.3) -- (1.55,1.3) -- (1.55,2) -- (3,2) -- (3,4);
        \draw[black, thick] (3,0) -- (3,1.5) -- (4,1.5);
        \draw (0.75,3) node{\large$\mu$};
        \draw (3.5, 0.8) node{\large$\lambda^c$};
        \draw (3.2,3.8) node{$1$};
        \draw (3.2,1.7) node{$a$};
        \draw (3.2,2.8) node{$\vdots$};
        \draw (3.5,3.8) node{$1$};
        \draw (3.5,1.7) node{$a$};
        \draw (3.5,2.8) node{$\vdots$};
        \draw (3.8,3.8) node{$1$};
        \draw (3.8,1.7) node{$a$};
        \draw (3.8,2.8) node{$\vdots$};
        
        \draw (2.6,1.8) node{$1$};
        \draw (2.6,1.4) node{$\vdots$};
        \draw (2.6, 0.9) node{$b-3$};
        \draw (2.6, 0.65) node{\underline{$a+1$}};
        \draw (2.6, 0.4) node{ $a+2$};
        \draw (2.6, 0.15) node{$a+3$};
        
        \draw (1.9,1.8) node{$1$};
        \draw (1.9,1.3) node{$\vdots$};
        \draw (1.9, 0.65) node{$b-2$};
        \draw (1.9, 0.4) node{\underline{$a+1$}};
        \draw (1.9, 0.15) node{$a+2$};
        
        \draw (1.2, 0.15) node{\underline{$a+1$}};
        
        \draw (1.2, 1) node{$\vdots$};
        \draw (1.2, 0.65) node{ $c-2$};
        \draw (1.2, 0.4) node[text=red]{\underline{$b-2$}};
        
        \draw (0.5,1.1) node{$1$};
        \draw (0.5, 0.8) node{$\vdots$};
        \draw (0.5, 0.4) node{$c-1$};
        \draw (0.5, 0.15) node[text=blue]{\underline{$b-1$}};

    \end{tikzpicture}
\end{center}
    \caption{Ballot filling in Case 1.2}
    \label{Case1:ageqb}
\end{figure}

\noindent \emph{Case 1.2 ($l_2 \leq k_1$):} Since $l_2 \leq k_1$, we have $a\geq b$ and thus $\rho = a$. Let $T_1$ be the filling of $\lambda/\mu$ shown in Figure~\ref{Case1:a<b}. For each $A$-column, fill it with $[a]$. Fill the left $B$-column with $[b-2]\cup \{a+1,a+2\}$, and the right $B$-column with $[b-3]\cup \{a+1,a+2,a+3\}$. Fill the left $C$-column with $[c-1]\cup \{b-1\}$, and the right $C$-column with $[c-2]\cup \{b-2,a+1\}$. Let $T_2$ be the filling obtained from $T_1$ by swapping the blue $b-1$ and red $b-2$ within the $C$-columns in Figure~\ref{Case1:a<b}.

\noindent ($T_1,T_2$ are semistandard): Since $a \geq b \geq c+2$ and $\lambda$ has shortness at least $3$, it is clear from the construction that both tableaux are semistandard. 

\noindent ($T_1,T_2$ are ballot): Let $u$ and $v$ be the column reading words of $T_1$ and $T_2$, respectively. Applying Lemma~\ref{Lemma:ballot}, we only need to check that the initial factors of $u$ and $v$ terminating at the underlined entries in Figure~\ref{Case1:a<b} are ballot. The following remarks hold for both $u$ and $v$. There are three entries equal to $a$ preceding the three underlined $a+1$ entries. There are four $b-2$, and at least five $b-3$, before the underlined $b-2$. There are three $b-1$, and at least four $b-2$, before the underlined $b-1$. Thus $T_1$ and $T_2$ are ballot.

Therefore $s_{\lambda/\mu}(x_1,\ldots,x_{n})$ is not multiplicity-free.

\begin{figure}
    \begin{center}
    \begin{tikzpicture}
        \draw[black, thick] (0,0) -- (4,0) -- (4,4) -- (0,4) -- (0,0);
        \draw[black, thick] (0,1) -- (1.5,1) -- (1.5,2.5) -- (3,2.5) -- (3,4);
        \draw[black, thick] (2.5,0) -- (2.5,1) -- (4,1);
        \draw[gray, dashed] (3,2.5) -- (3,1);
        \draw[gray, dashed] (1.5,0) -- (1.5,1) -- (2.5,1) -- (2.5,2.5);
        \draw (0.75,3) node{$\mu$};
        \draw (3.3, 0.5) node{$\lambda^c$};
        \draw (3.5, 2.5) node{$a$};
        \draw (2.75, 1.75) node{$b$};
        \draw (2, 1.25) node{$c$};
        \draw (0.75,0.5) node{$d$};

    \end{tikzpicture}
\end{center}
    \caption{Case 2}
    \label{Case2}
\end{figure}

\noindent \emph{Case 2 ($\mu_2 < \lambda_2 < \mu_1$ with $k_1+k_2 = l_1$):} In this case $a = l_1$, $b = k_2$ , $c = l_1+l_2-k_1$ and $d= \ell(\lambda) - \ell(\mu)$. By Corollary~\ref{Cor:GUT}, it is enough to show multiplicity when there are two $A$-columns, one $B$-column, one $C$-column, and two $D$-columns; that is, \[ \lambda = (6^{l_1},3^{l_2})\text{ and } \mu=(4^{k_1},2^{k_2}). \]

We consider two subcases: Case 2.1 (together with Case 3.1.1-3.1.3) corresponds to Theorem~\ref{thm:suffshort3} (\RomanNumeralCaps{4}) while Case 2.2 (together with Case 3.2.1 and 3.2.2) corresponds to Theorem~\ref{thm:suffshort3} (\RomanNumeralCaps{3}).

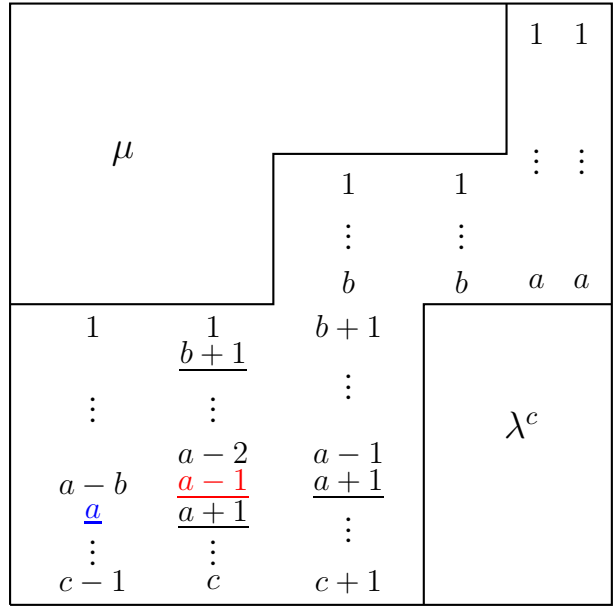
\begin{figure}
    \begin{center}
    \begin{tikzpicture}[scale =2]
        \draw[black, thick] (0,0) -- (4,0) -- (4,4) -- (0,4) -- (0,0);
        \draw[black, thick] (0,2) -- (1.75,2) -- (1.75,3) -- (3.3,3) -- (3.3,4);
        \draw[black, thick] (2.75,0) -- (2.75,2) -- (4,2);
        \draw (0.75,3) node{\large$\mu$};
        \draw (3.4, 1.2) node{\large$\lambda^c$};
        \draw (3.5,3.8) node{$1$};
        \draw (3.5,2.15) node{$a$};
        \draw (3.5,3) node{$\vdots$};
        \draw (3.8,3.8) node{$1$};
        \draw (3.8,2.15) node{$a$};
        \draw (3.8,3) node{$\vdots$};
        
        \draw (3,2.8) node{$1$};
        \draw (3, 2.5) node{$\vdots$};
        \draw (3,2.15) node{$b$};
        
        \draw (2.25,2.8) node{$1$};
        \draw (2.25, 2.5) node{$\vdots$};
        \draw (2.25,2.15) node{$b$};
        \draw (2.25,1.85) node{$b+1$};
        \draw (2.25,1.5) node{$\vdots$};
        \draw (2.25, 1) node{$a-1$};
        \draw (2.25, 0.8) node{\underline{$a+1$}};
        \draw (2.25,0.55) node{$\vdots$};
        \draw (2.25, 0.15) node{$c+1$};

        \draw (1.35, 1.85) node{$1$};
        \draw (1.35,1.65) node{\underline{$b+1$}};
        \draw (1.35, 1.35) node{$\vdots$};
        \draw (1.35, 1) node{$a-2$};
        \draw (1.35, 0.8) node[text=red]{\underline{$a-1$}};
        \draw (1.35,0.6) node{\underline{$a+1$}};
        \draw (1.35,0.4) node{$\vdots$};
        \draw (1.35, 0.15) node{$c$};
        
        \draw (0.55, 1.85) node{$1$};
        \draw (0.55, 1.35) node{$\vdots$};
        \draw (0.55, 0.8) node{$a-b$};
        \draw (0.55,0.6) node[text=blue]{\underline{$a$}};
        \draw (0.55,0.4) node{$\vdots$};
        \draw (0.55, 0.15) node{$c-1$};

    \end{tikzpicture}
\end{center}
    \caption{Ballot filling in Case 2.1}
    \label{Case2:a<b}
\end{figure}

\noindent \emph{Case 2.1 ($l_2 > k_1$):} In this case $c>a$ and thus $\rho = c$. Let $T_1$ be the filling shown in Figure~\ref{Case2:a<b}. Fill the $A$- and $B$-columns with $[a]$ and $[b]$, respectively. Fill in the $C$-column with $[c+1]\setminus \{a\}$. Fill in the left $D$-column with $[a-b]\cup \{a,\ldots,c-1\}$ and the right $D$-column with $[c]\setminus (\{2,\ldots,b\}\cup\{a\})$. Define $T_2$ to be the tableau obtained by swapping the blue $a$ and red $a-1$ within the $D$-columns in Figure~\ref{Case2:a<b}.

\noindent ($T_1,T_2$ are semistandard:) Since $\mu$ has shortness at least two, we know that $a-b \leq a-2$. It is trivial to check that both tableaux are semistandard.

\noindent ($T_1,T_2$ are ballot:) Let $u$ and $v$ be the column reading words of $T_1$ and $T_2$, respectively. Applying Lemma~\ref{Lemma:ballot}, we only need to check that the initial factors of $u$ and $v$ terminating at the underlined entries in Figure~\ref{Case2:a<b} are ballot. The following remarks hold for both $u$ and $v$. There are two entries equal to $a$, and none equal to $a+1$, preceding the two underlined $a+1$. Since $\mu$ has shortness at least two, we know that 
\begin{equation}
\label{eq:Case2:a<b}
b+1 \leq a-1.
\end{equation}
Thus there are four entries equal to $b$, and three entries equal to $b+1$, before the underlined $b+1$. By \eqref{eq:Case2:a<b}, there are at least four entries equal to $a-2$, and three entries equal to $a-1$, before the underlined $a-1$. Finally, we apply \eqref{eq:Case2:a<b} to conclude there are two entries equal to $a$, and at least three entries equal to $a-1$, before the underlined $a$. Thus both $T_1$ and $T_2$ are ballot tableaux.

As a result $s_{\lambda/\mu}(x_1,\ldots,x_{\rho+1})$ is not multiplicity-free.

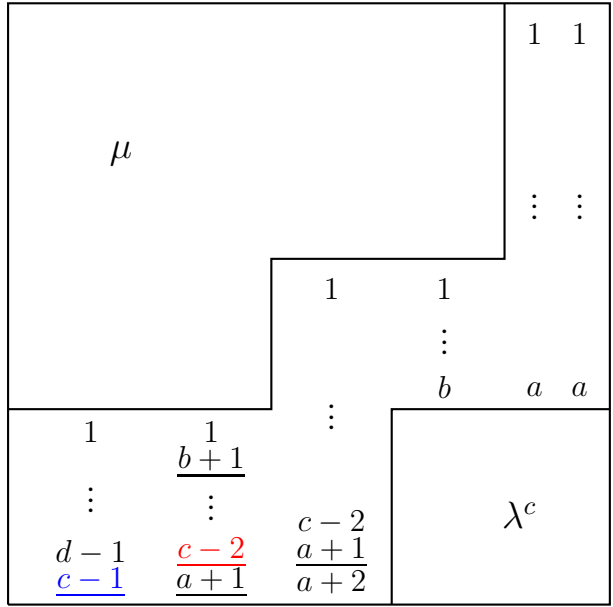
\begin{figure}
    \begin{center}
    \begin{tikzpicture}[scale =2]
        \draw[black, thick] (0,0) -- (4,0) -- (4,4) -- (0,4) -- (0,0);
        \draw[black, thick] (0,1.3) -- (1.75,1.3) -- (1.75,2.3) -- (3.3,2.3) -- (3.3,4);
        \draw[black, thick] (2.55,0) -- (2.55,1.3) -- (4,1.3);
        \draw (0.75,3) node{\large$\mu$};
        \draw (3.4, 0.6) node{\large$\lambda^c$};
        \draw (3.5,3.8) node{$1$};
        \draw (3.5,1.43) node{$a$};
        \draw (3.5,2.7) node{$\vdots$};
        \draw (3.8,3.8) node{$1$};
        \draw (3.8,1.43) node{$a$};
        \draw (3.8,2.7) node{$\vdots$};
        
        \draw (2.9,2.1) node{$1$};
        \draw (2.9,1.8) node{$\vdots$};
        \draw (2.9,1.45) node{$b$};
        
        \draw (2.15,2.1) node{$1$};
        \draw (2.15, 1.3) node{$\vdots$};
        \draw (2.15,0.55) node{$c-2$};
        \draw (2.15,0.35) node{\underline{$a+1$}};
        \draw (2.15, 0.15) node{$a+2$};

        \draw (1.35, 1.15) node{$1$};
        \draw (1.35,0.95) node{\underline{$b+1$}};
        \draw (1.35, 0.7) node{$\vdots$};
        \draw (1.35, 0.35) node[text=red]{\underline{$c-2$}};
        \draw (1.35, 0.15) node{\underline{$a+1$}};
        
        \draw (0.55, 1.15) node{$1$};
        \draw (0.55, 0.75) node{$\vdots$};
        \draw (0.55,0.35) node{$d-1$};
        \draw (0.55, 0.15) node[text=blue]{\underline{$c-1$}};
        
    \end{tikzpicture}
\end{center}
    \caption{Ballot filling in Case 2.2}
    \label{Case2:ageqb}
\end{figure}
\noindent \emph{Case 2.2 ($l_2\leq k_1$):}  In this case we have $a\geq c$ and thus $\rho = a$. Let $T_1$ be the filling shown in Figure~\ref{Case2:ageqb}. Fill in the $A$ and $B$-columns with $[a]$ and $[b]$, respectively. Fill in the $C$-column with $[c-2]\cup \{a+1,a+2\}$. The right $D$-column is filled with $\{1\}\cup ([c-2]\setminus [b])\cup \{a+1\}$. Since $\lambda^\vee$ has shortness at least $3$, $c-b\geq 3$ and thus $b < c-2$. Fill in the left $D$-column with $[d-1]\cup\{c-1\}$. Define $T_2$ to be the tableaux obtained by swapping the blue $c-1$ and red $c-2$ within the $D$-columns in Figure~\ref{Case2:ageqb}. 

\noindent ($T_1,T_2$ are semistandard): Since $a\geq c$, it is trivial to check that $T_1$ and $T_2$ are semistandard.

\noindent ($T_1,T_2$ are ballot): Let $u$ and $v$ be the column reading words of $T_1$ and $T_2$, respectively. Once again, it suffices, by Lemma~\ref{Lemma:ballot}, to check that the initial factors of $u$ and $v$ terminating at the underlined entries in Figure~\ref{Case2:ageqb} are ballot. The following remarks hold for both $u$ and $v$. There are exactly two entries equal to $a$, and none equal to $a+1$, before the two underlined $a+1$. There are four entries equal to $b$, and three entries equal to $b+1$, before the underlined $b+1$. If the $A$ columns contain $c-2$, then they must contain $c-3$. Since $\lambda^{\vee}$ has shortness $3$, we have $b < c-2$, and so $c-2$ can not appear in the $B$ column. Thus, there is at least one more entry equal $c-3$ than $c-2$, before the underlined $c-2$. As argued previously, $b < c-2$ implies that $c-2$ and $c-1$ do not appear in the $B$ column. If the $A$ columns contain $c-1$, then they must contain a $c-2$. Thus there is at least one more entry equal to $c-2$ than $c-1$, before the underlined $c-1$. Thus $T_1$ and $T_2$ are ballot. 

Therefore $s_{\lambda/\mu}(x_1,\ldots,x_{\rho+2})$ is not multiplicity-free.

\begin{figure}
    \begin{center}
    \begin{tikzpicture}
        \draw[black, thick] (0,0) -- (4,0) -- (4,4) -- (0,4) -- (0,0);
        \draw[black, thick] (0,1) -- (1.5,1) -- (1.5,2.5) -- (3,2.5) -- (3,4);
        \draw[black, thick] (2,0) -- (2,2) -- (4,2);
        \draw[gray, dashed] (3,2) -- (3,2.5);
        \draw[gray, dashed] (2,2) -- (2,2.5);
        \draw[gray, dashed] (1.5,1) -- (1.5,0);
        \draw (0.75,3) node{$\mu$};
        \draw (2.75, 1) node{$\lambda^c$};
        \draw (1.3,1.2) node{\footnotesize $A$};
        \draw (2.8,2.7) node{\footnotesize $B$};
        \draw (2.3,1.8) node{\footnotesize $C$};
        
        \draw (3.5,3) node{\footnotesize $a$};
        \draw (2.5,2.25) node{\footnotesize $b$};
        \draw (1.75,1.75) node{\footnotesize $c$};
        \draw (0.5,0.5) node{\footnotesize $d$};

    \end{tikzpicture}
\end{center}
    \caption{Case 3}
    \label{Case3}
\end{figure}

\noindent \emph{Case 3 ($\mu_2 < \lambda_2 < \mu_1$ with $k_1+k_2 > l_1$):} By Corollary~\ref{Cor:GUT}, it is enough to show multiplicity when there are two $A$-columns, one $B$-column, one $C$-column, and two $D$-columns; that is, \[ \lambda = (6^{l_1},4^{l_2})\text{ and } \mu=(4^{k_1},2^{k_2}). \]

Since $\mu$ has a shortness at least $2$, and $\lambda^c$ has shortness at least $3$, 
\begin{equation}
\label{shortineq1}
c-2\geq d \geq 2 \text{ and } a\geq b+2,
\end{equation} 
\begin{equation}
\label{shortineq2}
a\geq 3 \text{ and } c\geq b+3.
\end{equation} 

Define 
\begin{equation}
\label{shortineq3}
x = d+a-c.
\end{equation} 
If follows, from \eqref{shortineq1}, that
\begin{equation}
\label{shortineq4}
x\leq a-2.
\end{equation} 

We divide into four subcases: Case 3.1.1 - Case 3.1.3 (together with Case 2.1) correspond to Theorem~\ref{thm:suffshort3} (\RomanNumeralCaps{4}) and Case 3.2.1 and 3.2.2 (together with Case 2.2) correspond to Theorem~\ref{thm:suffshort3} (\RomanNumeralCaps{3}).

\begin{figure}
    \begin{center}
    \begin{tikzpicture}[scale =2]
        \draw[black, thick] (1.4,0) -- (4,0) -- (4,4) -- (1.4,4) -- (1.4,0);
        \draw[black, thick] (1.4,2.2) -- (2.5,2.2) -- (2.5,3) -- (3.3,3) -- (3.3,4);
        \draw[black, thick] (3,0) -- (3,2.3) -- (4,2.3);
        \draw (2,3.5) node{\large$\mu$};
        \draw (3.6, 0.9) node{\large $\lambda^c$};
        \draw (3.5,3.8) node{$1$};
        \draw (3.5,2.9) node{$\vdots$};
        \draw (3.5,2.4) node{$a$};
        \draw (3.8,3.8) node{ $1$};
        \draw (3.8,2.9) node{ $\vdots$};
        \draw (3.8,2.4) node{ $a$};
        \draw (3.2,2.85) node{$1$};
        \draw (3.2,2.7) node{$\vdots$};
        \draw (3.2,2.4) node{$b$};
		\draw (2.7,2.85) node{ $1$};
        \draw (2.7,2) node{ $\vdots$};
		\draw (2.7,1.09) node{ $a-1$}; 
		\draw (2.7,0.91) node{\underline{$a+1$}};
		\draw (2.7,0.6) node{ $\vdots$}; 
        \draw (2.7,0.11) node{ $c+1$};        
        \draw (2.2,2.1) node{ $1$};
        \draw (2.2,1.95) node{ $\vdots$}; 
        \draw (2.2,1.75) node{ $b-1$}; 
        \draw (2.2,1.55) node{ \underline{$b+1$}}; 
        \draw (2.2,1.35) node{ $\vdots$}; 
		\draw (2.2,1.08) node{ $x$}; 
		\draw (2.2,0.9) node[red]{ \underline{$a$}};  
		\draw (2.2,0.7) node{ \underline{$a+1$}};
		\draw (2.2,0.45) node{ $\vdots$}; 
        \draw (2.2,0.1) node{ $c$};       
        \draw (1.7,2.1) node{ $1$};
        \draw (1.7,1.5) node{ $\vdots$};
		\draw (1.7,0.9) node{ $x$}; 
		\draw (1.7,0.7) node[blue]{ \underline{$x+1$}}; 
		\draw (1.7,0.5) node{ \underline{$a+1$}};
		\draw (1.7,0.35) node{ $\vdots$}; 
        \draw (1.7,0.11) node{ $c-1$}; 
        
    \end{tikzpicture}
\end{center}
    \caption{Ballot filling in Case 3.1.1}
    \label{Case3a'}
\end{figure}
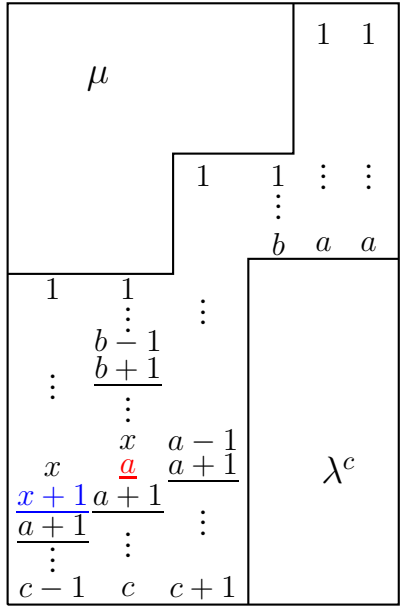

\noindent \emph{Case 3.1.1 ($a< c$, $b< x$):} In this case $\rho = c$. Let $T_1$ be the filling shown in Figure~\ref{Case3a'}. Fill both $A$-columns with $[a]$ and the $B$-column with $[b]$. Fill the $C$-column with $[c+1]\setminus \{a\}$. Fill the left $D$-column with $[x+1]\cup ([c-1]\setminus [a])$, and the right $D$-column with $([x]\setminus \{b\})\cup ([c]\setminus [a-1])$. By \eqref{shortineq2}, the $D$-columns have $d$ entries. Define $T_2$ to be the filling that is obtained by swapping the blue $x+1$ and red $a$ in Figure~\ref{Case3a'}.

\noindent ($T_1,T_2$ are semistandard): $T_1$ and $T_2$ are semistandard by \eqref{shortineq4}.

\noindent ($T_1,T_2$ are ballot:) Let $u$ and $v$ be the column reading words of $T_1$ and $T_2$, respectively. Applying Lemma~\ref{Lemma:ballot}, we only need to check that the initial factors of $u$ and $v$ that terminate at the underlined entries in Figure~\ref{Case3a'} are ballot. The following remarks hold for both $u$ and $v$. There are three entries equal to $a+1$. There are two entries equal to $a$ before the first underlined $a+1$. There are at least two entries equal to $a$ before the second underlined $a+1$. There are three entries equal to $a$ before the third underlined $a+1$. There are four entries equal to $b$, and three equal to $b+1$, before the underlined $b+1$. There are at least three entries equal to $a-1$, and two equal to $a$, before the underlined $a$. If the $B$-column contains $x+1$, then it must contain $x$. This, combined with the fact that there is no $x$ in the $C$-column, implies there is at least one more entry equal to $x$ than $x+1$, before the underlined $x+1$. Thus $T_1$ and $T_2$ are ballot.

We conclude $s_{\lambda/\mu}(x_1,\ldots,x_{\rho+1})$ is not multiplicity-free.

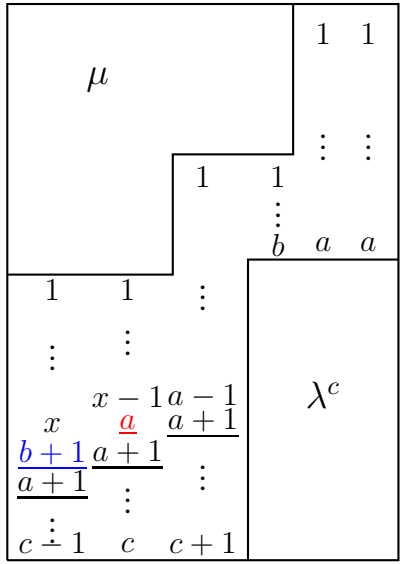
\begin{figure}
    \begin{center}
    \begin{tikzpicture}[scale =2]
        \draw[black, thick] (1.4,0) -- (4,0) -- (4,3.7) -- (1.4,3.7) -- (1.4,0);
        \draw[black, thick] (1.4,1.9) -- (2.5,1.9) -- (2.5,2.7) -- (3.3,2.7) -- (3.3,3.7);
        \draw[black, thick] (3,0) -- (3,2) -- (4,2);
        \draw (2,3.2) node{\large$\mu$};
        \draw (3.5, 1.1) node{\large $\lambda^c$};
        \draw (3.5,3.5) node{$1$};
        \draw (3.5,2.8) node{$\vdots$};
        \draw (3.5,2.1) node{$a$};
        \draw (3.8,3.5) node{ $1$};
        \draw (3.8,2.8) node{ $\vdots$};
        \draw (3.8,2.1) node{ $a$};
        \draw (3.2,2.55) node{$1$};
        \draw (3.2,2.35) node{$\vdots$};
        \draw (3.2,2.1) node{$b$};
		\draw (2.7,2.55) node{ $1$};
        \draw (2.7,1.8) node{ $\vdots$};
		\draw (2.7,1.09) node{ $a-1$}; 
		\draw (2.7,0.91) node{ \underline{$a+1$}};
		\draw (2.7,0.6) node{ $\vdots$}; 
        \draw (2.7,0.11) node{ $c+1$};        
        \draw (2.2,1.8) node{ $1$};
        \draw (2.2,1.5) node{ $\vdots$}; 
		\draw (2.2,1.08) node{ $x-1$}; 
		\draw (2.2,0.9) node[red]{ \underline{$a$}};  
		\draw (2.2,0.7) node{ \underline{$a+1$}};
		\draw (2.2,0.45) node{ $\vdots$}; 
        \draw (2.2,0.1) node{ $c$};       
        \draw (1.7,1.8) node{ $1$};
        \draw (1.7,1.4) node{ $\vdots$};
		\draw (1.7,0.9) node{ $x$}; 
		\draw (1.7,0.7) node[blue]{ \underline{$b+1$}}; 
		\draw (1.7,0.5) node{ \underline{$a+1$}};
		\draw (1.7,0.25) node{ $\vdots$}; 
        \draw (1.7,0.11) node{ $c-1$}; 
        
    \end{tikzpicture}
\end{center}
    \caption{Ballot filling in Case 3.1.2}
    \label{Case3a''}
\end{figure}

\noindent \emph{Case 3.1.2 ($a< c$, $b\geq x\geq 1$):} In this case $\rho = c$. Let $T_1$ be the filling shown in Figure~\ref{Case3a''}. Fill both $A$-columns with $[a]$ and the $B$-column with $[b]$. Fill the $C$-column with $[c+1]\setminus \{a\}$. Fill the left $D$-column with $[x]\cup \{b+1\} \cup ([c-1]\setminus [a])$ and the right $D$-column with $[x-1]\cup ([c]\setminus [a-1])$. If $x=1$, we will simply start the right $D$-column with $a$. It follows, by \eqref{shortineq3}, that there are $d$ entries in each $D$-column. Define $T_2$ to be the filling that arises by swapping the blue $b+1$ with the red $a$ in Figure~\ref{Case3a''}.

\noindent ($T_1,T_2$ are semistandard): By \eqref{shortineq1}, \eqref{shortineq4}, and the hypothesis of this case, $T_1$ and $T_2$ are semistandard.

\noindent ($T_1,T_2$ are ballot:) Let $u$ and $v$ be the column reading words of $T_1$ and $T_2$, respectively. By Lemma~\ref{Lemma:ballot}, it suffices to check that the initial factors of $u$ and $v$ terminating at the underlined entries in Figure~\ref{Case3a''} are ballot. This follows by a similar argument to Case 3.1.1. 

Thus $s_{\lambda/\mu}(x_1,\ldots,x_{\rho+1})$ is not multiplicity-free.

\noindent \emph{Case 3.1.3 ($a< c$, $x<1$):} In this case $\rho = c$. Let $T_1$ be the filling shown in Figure~\ref{Case3a'''}. Fill both $A$-columns with $[a]$ and the $B$-column with $[b]$. Fill the $C$-column with $[c+1]\setminus \{a\}$. Fill the left $D$-column with $\{1,b+1\} \cup ([a+d-2]\setminus [a])$ and the right $D$-column with $([a+d-1]\setminus [a-1])$. Define $T_2$ to be the filling that arises by swapping the blue $b+1$ with the red $a$ in Figure~\ref{Case3a'''}.

\noindent ($T_1,T_2$ are semistandard): By \eqref{shortineq3} we have $a+d-c = x \leq 0$ and thus $a\leq c-d$. Therefore the $a+1$ entry in $C$-column is in a row above the red $a$ entry in the right $D$-column. It also follows that $a+d-1<c+1$. As a result, $T_1$ and $T_2$ are semistandard.

\noindent ($T_1,T_2$ are ballot:) Let $u$ and $v$ be the column reading words of $T_1$ and $T_2$, respectively. By Lemma~\ref{Lemma:ballot}, it suffices to check that the initial factors of $u$ and $v$ ending at the underlined entries in Figure~\ref{Case3a''} are ballot. This follows by a similar argument to Case 3.1.1. 

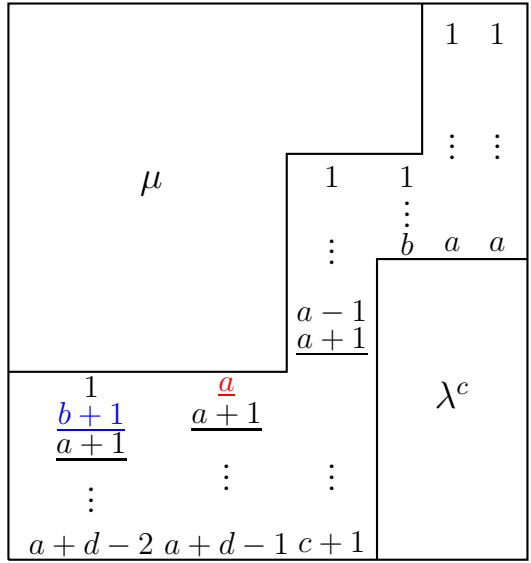
\begin{figure}
    \begin{center}
    \begin{tikzpicture}[scale =2]
        \draw[black, thick] (0.55,0) -- (4,0) -- (4,3.7) -- (0.55,3.7) -- (0.55,0);
        \draw[black, thick] (0.55,1.25) -- (2.4,1.25) -- (2.4,2.7) -- (3.3,2.7) -- (3.3,3.7);
        \draw[black, thick] (3,0) -- (3,2) -- (4,2);
        \draw (1.5,2.5) node{\large$\mu$};
        \draw (3.5, 1.1) node{\large $\lambda^c$};
        \draw (3.5,3.5) node{$1$};
        \draw (3.5,2.8) node{$\vdots$};
        \draw (3.5,2.1) node{$a$};
        \draw (3.8,3.5) node{ $1$};
        \draw (3.8,2.8) node{ $\vdots$};
        \draw (3.8,2.1) node{ $a$};
        
        \draw (3.2,2.55) node{$1$};
        \draw (3.2,2.35) node{$\vdots$};
        \draw (3.2,2.1) node{$b$};
        
		\draw (2.7,2.55) node{ $1$};
        \draw (2.7,2.1) node{ $\vdots$};
		\draw (2.7,1.65) node{ $a-1$}; 
		\draw (2.7,1.45) node{ \underline{$a+1$}};
		\draw (2.7,0.6) node{ $\vdots$}; 
        \draw (2.7,0.11) node{ $c+1$};  
        
		\draw (2,1.15) node[red]{ \underline{$a$}};  
		\draw (2,0.95) node{ \underline{$a+1$}};
		\draw (2,0.6) node{ $\vdots$}; 
        \draw (2,0.1) node{ $a+d-1$}; 
        
        \draw (1.1,1.15) node{$1$};
		\draw (1.1,0.95) node[blue]{ \underline{$b+1$}};
		\draw (1.1,0.75) node{ \underline{$a+1$}};
		\draw (1.1,0.45) node{ $\vdots$}; 
        \draw (1.1,0.1) node{$a+d-2$}; 
        
    \end{tikzpicture}
\end{center}
    \caption{Ballot filling in Case 3.1.3}
    \label{Case3a'''}
\end{figure}

Thus $s_{\lambda/\mu}(x_1,\ldots,x_{\rho+1})$ is not multiplicity-free.

\begin{figure}
    \begin{center}
    \begin{tikzpicture}[scale =2]
        \draw[black, thick] (1.4,0) -- (4,0) -- (4,3.2) -- (1.4,3.2) -- (1.4,0);
        \draw[black, thick] (1.4,1.1) -- (2.5,1.1) -- (2.5,1.9) -- (3.3,1.9) -- (3.3,3.2);
        \draw[black, thick] (3,0) -- (3,1.2) -- (4,1.2);
        \draw (2,2.7) node{\large$\mu$};
        \draw (3.6, 0.9) node{\large $\lambda^c$};
        \draw (3.5,3) node{$1$};
        \draw (3.5,2.1) node{$\vdots$};
        \draw (3.5,1.3) node{$a$};
        \draw (3.8,3) node{ $1$};
        \draw (3.8,2.1) node{ $\vdots$};
        \draw (3.8,1.3) node{ $a$};
        \draw (3.2,1.75) node{$1$};
        \draw (3.2,1.6) node{$\vdots$};
        \draw (3.2,1.3) node{$b$};
		\draw (2.7,1.75) node{ $1$};
        \draw (2.7,1.1) node{ $\vdots$}; 
		\draw (2.7,0.51) node{ $c-2$}; 
		\draw (2.7,0.31) node{ \underline{$a+1$}}; 
        \draw (2.7,0.11) node{ $a+2$};        
        \draw (2.2,1) node{ $1$};
        \draw (2.2,0.8) node{ $\vdots$}; 
		\draw (2.2,0.51) node{ $d-2$};  
		\draw (2.2,0.31) node[red]{ \underline{$c-1$}};
        \draw (2.2,0.11) node{ \underline{$a+1$}};       
        \draw (1.7,1) node{ $1$};
        \draw (1.7,0.6) node{ $\vdots$};
		\draw (1.7,0.31) node{ $d-1$};
        \draw (1.7,0.11) node[blue]{ \underline{$b+1$}}; 
    \end{tikzpicture}
\end{center}
    \caption{Ballot filling in Case 3.2.1}
    \label{Case3b'}
\end{figure}
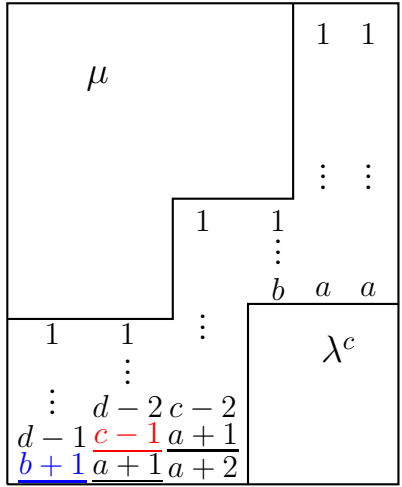

\noindent \emph{Case 3.2.1 ($a\geq c$, $d< b+2$):} In this case $\rho = c$. Let $T_1$ be the filling illustrated in Figure~\ref{Case3b'}. Fill the $A$-columns with $[a]$ and the $B$-column with $[b]$. Fill the $C$-column with $[ c-2]\cup \{a+1, a+2\}$. Fill the left $D$-column with $[d-1]\cup \{ b+1\}$ and the right $D$-column with $[d-2]\cup\{ c-1, a+1\}$. Let $T$ be the filling obtained by swapping the blue $b+1$ and red $c-1$ in Figure~\ref{Case3b'}.

\noindent ($T_1,T_2$ are semistandard): $T_1$ and $T_2$ are semistandard by the hypothesis of this case.

\noindent ($T_1,T_2$ are ballot:) Let $u$ and $v$ be the column reading words of $T_1$ and $T_2$, respectively. Once again, we only need to check that the initial factors of $u$ and $v$ that end at the underlined entries in Figure~\ref{Case3b'} are ballot. The following statements hold for both $u$ and $v$. There are exactly two entries equal to $a+1$, with at least two entries equal to $a$ prior to them. If there is a $c-1$ in each $A$-column, then there must be a $c-2$. Since $\lambda^{\vee}$ has shortness at least 3, there is no $c-2$ or $c-1$ in the $B$-column. Thus, since there is no $c-1$ in the $C$-column, there is at least one more entry equal to $c-2$ than $c-1$, before the underlined $c-1$. There is at least 4 entries equal to $b$, and exactly three equal to $b+1$, before the underlined $b+1$. Thus $T_1$ and $T_2$ are ballot.

Hence $s_{\lambda/\mu}(x_1,\ldots,x_{\rho+2})$ is not multiplicity-free.

\begin{figure}
    \begin{center}
    \begin{tikzpicture}[scale=2]
        \draw[black, thick] (1.4,0) -- (4,0) -- (4,4) -- (1.4,4) -- (1.4,0);
        \draw[black, thick] (1.4,1.7) -- (2.5,1.7) -- (2.5,2.5) -- (3.3,2.5) -- (3.3,4);
        \draw[black, thick] (3,0) -- (3,1.8) -- (4,1.8);
        \draw (2,3.5) node{\large$\mu$};
        \draw (3.6, 0.9) node{\large $\lambda^c$};
        \draw (3.5,3.8) node{$1$};
        \draw (3.5,2.9) node{$\vdots$};
        \draw (3.5,1.9) node{$a$};
        \draw (3.8,3.8) node{ $1$};
        \draw (3.8,2.9) node{ $\vdots$};
        \draw (3.8,1.9) node{ $a$};
        \draw (3.2,2.35) node{$1$};
        \draw (3.2,2.2) node{$\vdots$};
        \draw (3.2,1.9) node{$b$};
		\draw (2.7,2.35) node{ $1$};
        \draw (2.7,1.2) node{ $\vdots$}; 
		\draw (2.7,0.51) node{ $c-2$}; 
		\draw (2.7,0.31) node{ \underline{$a+1$}}; 
        \draw (2.7,0.11) node{ $a+2$};        
        \draw (2.2,1.6) node{ $1$};
        \draw (2.2,1.42) node{ $\vdots$}; 
        \draw (2.2,1.2) node{ $b-1$}; 
        \draw (2.2,1) node{ \underline{$b+1$}}; 
        \draw (2.2,0.77) node{ $\vdots$}; 
		\draw (2.2,0.51) node{ $d-1$};  
		\draw (2.2,0.31) node[red]{ \underline{$c-1$}};
        \draw (2.2,0.11) node{ \underline{$a+1$}};       
        \draw (1.7,1.6) node{ $1$};
        \draw (1.7,1) node{ $\vdots$};
        \draw (1.7,0.31) node{ $d-1$};
        \draw (1.7,0.11) node[blue]{ $d$}; 
    \end{tikzpicture}
\end{center}
    \caption{Ballot filling in Case 3.2.2}
    \label{Case3b''}
\end{figure}
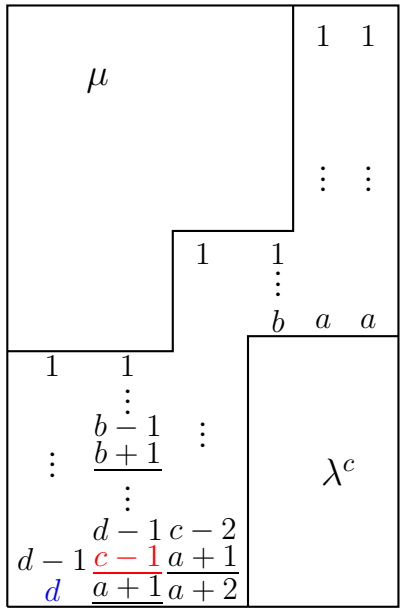

\noindent \emph{Case 3.2.2 ($a\geq c$, $d\geq b+2$):} In this case $\rho = c$. Let $T_1$ be the filling found in Figure~\ref{Case3b''}. Fill the $A$-columns with $[a]$ and the $B$-column with $[b]$. Fill the $C$-column with $[c-2]\cup \{ a+1, a+2\}$. Fill the left $D$-column with $[d]$, and right $D$-column with $([d-1]\setminus\{b\})\cup\{ c-1, a+1\}$. The filling $T_2$ is obtained by swapping the blue $d$ with the red $c-1$ in Figure~\ref{Case3b''}.

\noindent ($T_1,T_2$ are semistandard): $T_1$ and $T_2$ are semistandard by the hypothesis of this case.

\noindent ($T_1,T_2$ are ballot:) Let $u$ and $v$ be the column reading words of $T_1$ and $T_2$, respectively. Once again, we only need to check that the initial factors of $u$ and $v$ that end at the underlined entries in Figure~\ref{Case3b'} are ballot. This follows by a similar argument to Case 3.2.1.

We conclude that $s_{\lambda/\mu}(x_1,\ldots,x_{\rho+2})$ is not multiplicity-free.

\begin{figure}
    \begin{center}
    \begin{tikzpicture}
        \draw[black, thick] (0,0) -- (4,0) -- (4,4) -- (0,4) -- (0,0);
        \draw[black, thick] (0,1) -- (1.5,1) -- (1.5,2.5) -- (3,2.5) -- (3,4);
        \draw[black, thick] (1.5,0) -- (1.5,1) -- (4,1);
        \draw[gray, dashed] (3,1) -- (3,2.5);
        \draw (0.75,3) node{$\mu$};
        \draw (2.75, 0.5) node{$\lambda^c$};
        \draw (1.3,1.2) node{\footnotesize $A$};
        \draw (2.8,2.7) node{\footnotesize $B$};
        \draw (1.7,0.8) node{\footnotesize $C$};
        \draw (3.5,2) node{\footnotesize $a$};
        \draw (2.25,1.75) node{\footnotesize $b$};
        \draw (0.6,0.5) node{\footnotesize $c$};
    \end{tikzpicture}
\end{center}
    \caption{Case 4}
    \label{Case4}
\end{figure}
\noindent \emph{Case 4 ($\lambda_2 = \mu_2$ with $k_1 + k_2 = l_1$):} By Corollary~\ref{Cor:GUT}, it is enough to show multiplicity when there are two $A$-columns, two $B$-columns, and three $C$-columns. Equivalently, \[ \lambda = (7^{l_1},3^{l_2})\text{ and } \mu=(5^{k_1},3^{k_2}). \]

Since $\mu$ has shortness at least 2, and $\lambda^\vee$ has shortness at least 3,
\begin{align}\label{ineq1:case4}
    b\geq 2, c\geq 3, a\geq 4, \text{ and } a-b \geq 2.
\end{align}

Let $\tilde\lambda = (6^{l_1},3^{l_2})$ and $\tilde\mu = (4^{k_1},2^{k_2})$. Then $|\lambda| - |\mu| = |\tilde\lambda| - |\tilde\mu|$ and $\tilde\lambda/\tilde\mu$ satisfies the hypotheses of Case 2 in Theorem~\ref{thm:suffshort3}. Let $T$ be a ballot tableau of shape $\tilde\lambda/\tilde\mu$. Define $T^{\sf shift}$ to be the filling of shape $\lambda / \mu = (7^{l_1},3^{l_2}) / (5^{k_1},3^{k_2})$ where for all $(r,k) \in \lambda / \mu$,  
\begin{equation*}
T^{\sf shift}(r,k) = \begin{cases} T(r,k-1) & \text{for }r \leq l_1 \\ T(r,k) & \text{for }r > l_1  \end{cases}
\end{equation*}

\begin{claim}\label{Claim:Case4} If $T$ is a ballot tableau of shape $\tilde\lambda/\tilde\mu$ and content $\nu$, then $T^{\sf shift}$ is a ballot
tableau of shape $\lambda / \mu$ and content $\nu$.
\end{claim}
\noindent \emph{Proof of Claim~\ref{Claim:Case4}:}
It is trivial to check that $T$ and $T^{\sf shift}$ have the same content. Since $T$ is semistandard, $T^{\sf shift}$ will be semistandard. By construction the column reading words of $T$ and $T^{\sf shift}$ are identical, and thus both are ballot. \qed

By Lemma~\ref{Lemma:krowreduction} and \eqref{ineq1:case4}, we reduce to the case where either $c=3$ or $b=2$. We divide into $4$ subcases: Case 4.2.1 corresponds to Theorem~\ref{thm:suffshort3} (\RomanNumeralCaps{5}), and the other cases correspond to Theorem~\ref{thm:suffshort3} (\RomanNumeralCaps{6}).

The column lengths of the $A$,$B$,$C$, and $D$-columns of $\tilde\lambda/\tilde\mu$ are $a, b, b+c,$ and $c$, respectively.

\noindent \emph{Case 4.1.1 ($c=3,a-b<3$):}  By \eqref{ineq1:case4}, we have $a = b+2 = b+c-1$. This implies $\rho = a$. Since $a<b+c$, this implies that $\tilde\lambda/\tilde\mu$ satisfies the hypotheses of Case 2.1. Let $T_1$ and $T_2$ be the two ballot fillings of $\tilde\lambda/\tilde\mu$ constructed in Case 2.1. $T_1$ and $T_2$ had content with length $b+c+1$. Claim~\ref{Claim:Case4} implies that $T_1^{\sf shift}$ and $T_2^{\sf shift}$ are two ballot fillings of $\lambda / \mu$ with the same content as $T_1$ and $T_2$. Thus, we have two ballot tableaux of shape $\lambda / \mu$, with content of length $b+c+1 = a+2 = \rho + 2$. Thus $s_{\lambda/\mu}(x_1,\ldots,x_{\rho+2})$ is not multiplicity-free.

\noindent \emph{Case 4.1.2 ($c=3,a-b\geq 3$):} In this case, $a\geq b+c$. This implies $\rho = a$, and that $\tilde\lambda/\tilde\mu$ satisfies the hypotheses of Case 2.2. By the same argument as in Case 4.1.1, we can find two ballot tableaux of shape $\lambda / \mu$ with content of length $a + 2 = \rho + 2$. Thus $s_{\lambda/\mu}(x_1,\ldots,x_{\rho+2})$ is not multiplicity-free.

\noindent \emph{Case 4.2.1 ($b=2, a-c < 2$):} Since $a<b+c$, $\tilde\lambda/\tilde\mu$ satisfies the hypotheses of Case 2.2. By the same argument as in Case 4.1.1, we can find two ballot tableaux of shape $\lambda / \mu$ with content $\nu$ of length $b + c + 1 =  c + 3$. If $a-c=1$, we have $\rho = a$ and $\ell(\nu) = \rho + 2$. If $a = c$, we have $\rho = c$ and $\ell(\nu) = \rho + 3$. We conclude that for both $a-c=1$ and $a=c$, $s_{\lambda/\mu}(x_1,\ldots,x_{\rho+3})$ is not multiplicity-free.

\noindent \emph{Case 4.2.2 ($b=2,a-c\geq 2$):} In this case, $a\geq b+c$, which implies $\rho = a$. It also implies that $\tilde\lambda/\tilde\mu$ satisfies the hypotheses of Case 2.2. Thus we can find two ballot tableaux of shape $\lambda / \mu$ with content of length $a + 2 = \rho + 2$. Thus $s_{\lambda/\mu}(x_1,\ldots,x_{\rho+2})$ is not multiplicity-free.

\section{Tightness of upper bounds}

In this section, we will show that the upper bounds on $\sf{m}(\lambda/\mu)$ obtained in the previous section are all tight. We will use the symmetry (\ref{eq:LRsymmetrybasic}) and interpret $c^{\lambda}_{\mu,\nu}$ as the number of Littlewood-Richardson tableaux of shape $\lambda/\nu$ and content $\mu$ throughout this section. Note that the content $\mu$ is guaranteed to be a partition. Let us begin with three lemmas:

\begin{lemma}\label{Lemma:basicballot}
Let $T$ be a ballot tableau of shape $\lambda/\nu$ and content $\mu = (\mu_1,\mu_2,\ldots,\mu_k)$. If $1\leq i<j\leq k$ and $c = min\{ z : \text{column $z$ contains $j$} \}$, then 
\begin{center}
$|\{ z : z < c\text{ and column $z$ contains $i$} \}| \leq \mu_i-\mu_j$.
\end{center}
 \end{lemma}
\begin{proof}
Since $T$ is ballot, the initial factor of the column reading word ending at the $j$ entry in column $c$ must contain at least $\mu_j$ $i$'s. Thus, there are at most $\mu_i-\mu_j$ entries equal to $i$ in the column reading word that appear after this initial factor. This implies the columns $z$, with $z < c$ can contain at most $\mu_i-\mu_j$ $i$'s.
\end{proof}

\begin{lemma}\label{Lemma:rowofT}
For a ballot tableau $T$ with content $\mu$, if $\mu_i= \mu_j$ and $i<j$, then $i$ does not appear in the bottom row of $T$.
\end{lemma}
\begin{proof}
Suppose, for contradiction, that $i$ appears in the bottom row of $T$. Let $x$ be the column index of the left-most $i$ in the bottom row of $T$ and $y$ be the column index of the left-most $j$ in $T$ (not necessarily in the bottom row). Then the semistandardness of $T$ implies that $x<y$. Applying Lemma~\ref{Lemma:basicballot}, we conclude

\begin{center}
$0 < |\{ z : z < y\text{ and column $z$ contains $i$} \}| \leq \mu_i-\mu_j = 0$.
\end{center}

which is a contradiction. Therefore $i$ does not appear in the bottom row of $T$.
\end{proof}

\begin{definition} \label{def:skewbottomrowremoval} Let $T$ be a tableau of shape $\lambda / \nu$ and content $\mu$. Define $(\lambda/\nu)^{{\sf del}(N)}$ to be the skew diagram obtained from $\lambda/\nu$ by removing the last $N$ rows. Define $T^{{\sf del}(N)}$ to be the tableau of shape $(\lambda/\nu)^{{\sf del}(N)}$ obtained from $T$ by removing the last $N$ rows. 
\end{definition} 
\begin{lemma}\label{Lemma:rmvlastrow}
If $T$ is a ballot tableau, then $T^{{\sf del}(N)}$ is a ballot tableau. 
\end{lemma}
\begin{proof}
Since $T$ is semistandard, $T^{{\sf del}(N)}$ must be semistandard. The reverse reading word of $T$ is ballot. Since the reverse reading word of $T^{{\sf del}(N)}$ is an initial factor of the reverse reading word of $T$, $T^{{\sf del}(N)}$ is also ballot.
\end{proof}

As a corollary of Lemma~\ref{Lemma:rmvlastrow}, we have the following result:

\begin{corollary}\label{Cor:rmvlastrow}
Suppose that all ballot tableaux of shape $\lambda/\nu$ and content $\mu$ are identical in the last $N$ rows. Let $\mu^{(N)}$ be the content of $T^{{\sf del}(N)}$, for all ballot tableaux $T$ of shape $\lambda/\nu$ and content $\mu$. If we set $\lambda^{(N)}/\nu^{(N)} := (\lambda/\nu)^{{\sf del}(N)}$, then
\[c^{\lambda}_{\mu,\nu}\leq c^{\lambda^{(N)}}_{\mu^{(N)},\nu^{(N)}},\]

\end{corollary}

\subsection{\texorpdfstring{$\lambda^\vee$}{} is a rectangle of shortness at least \texorpdfstring{$3$}{} and \texorpdfstring{$\mu$}{} is a fat hook of shortness at least \texorpdfstring{$2$}{}}

\begin{theorem}
\label{thm:necessaryshort3}
Let $s_{\lambda/\mu}(x_1,\ldots,x_n)$ be a basic, n-sharp, tight skew Schur polynomial such that $\lambda^\vee$ is a rectangle of shortness at least $3$ and $\mu$ is a fat hook of shortness at least $2$. Then the following hold:
\begin{enumerate}[label=(\Roman*)]
    \item If $\lambda_2 = \mu_1, l_2 > k_1$ then $\sf{m}(\lambda/\mu) = \rho+2$
    \item If $\lambda_2 = \mu_1, l_2 \leq k_1$ then $\sf{m}(\lambda/\mu) = \rho+3$
    \item If $\mu_1 > \lambda_2 > \mu_2,  k_1 \geq l_2$ then $\sf{m}(\lambda/\mu) = \rho+2$
    \item If $\mu_1 > \lambda_2 > \mu_2, l_2 > k_1$ then $\sf{m}(\lambda/\mu) = \rho+1$
    \item If $\mu_2 = \lambda_2, l_2 \geq  l_1$ then $\sf{m}(\lambda/\mu) = \rho+3$
    \item If $\mu_2 = \lambda_2, l_1 > l_2$ then $\sf{m}(\lambda/\mu) = \rho+2$
\end{enumerate}
\end{theorem}
\begin{proof}
By Definition~\ref{def:r1andr2}, the statements in Theorem~\ref{thm:necessaryshort3} (\RomanNumeralCaps{1})-(\RomanNumeralCaps{6}) are equivalent to the claim that $\sf{m}(\lambda/\mu)=\rho + r_1(\lambda / \mu) + r_2 (\lambda / \mu)$. Thus, by Theorem~\ref{thm:suffshort3} and (\ref{eq:LRsymmetrybasic}), to prove Theorem~\ref{thm:necessaryshort3} (\RomanNumeralCaps{1})-(\RomanNumeralCaps{6}), it is enough to show that for any partition $\nu$ such that $\ell(\nu) < \rho + r_1(\lambda / \mu) + r_2 (\lambda / \mu)$, there is at most one ballot tableau of shape $\lambda/\nu$ and content $\mu$. Equivalently, we want to show that
\begin{equation}
\label{eq:equivtoshow-ns3}
c^{\lambda}_{\mu,\nu} = c^{\lambda}_{\nu,\mu}\leq 1\text{ for all }\nu\text{ such that } \ell(\nu) < \rho + r_1(\lambda / \mu) + r_2 (\lambda / \mu).
\end{equation}
Note that \eqref{eq:equivtoshow-ns3} holds trivially for Theorem~\ref{thm:necessaryshort3} (\RomanNumeralCaps{4}) since $\lambda/ \mu$ is $n$-sharp.

We consider the same four cases as in the proof of Theorem~\ref{thm:suffshort3}, though the order in which we consider the cases is different. In all four cases $\nu$ is a partition such that $\ell(\nu) < \rho + r_1(\lambda / \mu) + r_2 (\lambda / \mu)$.

\noindent \emph{Case 1 $(\lambda_2 = \mu_1)$:} This case will corresponds to Theorem~\ref{thm:necessaryshort3} (\RomanNumeralCaps{1}) and (\RomanNumeralCaps{2}). If the first $l$ rows of $\nu$ and $\lambda$ have the same length, for some $l\in \mathbb{Z}_{>0}$, then removing the top $l$ rows of both $\lambda$ and $\nu$ does not change the value of $c^\lambda_{\mu,\nu}$. Denote $\widehat \lambda$ to be the partition obtained by removing the top $l$ rows of $\lambda$. If $\widehat \lambda$ has shortness at most 2 or is a rectangle, then $s_{\widehat \lambda/\mu}$ is multiplicity-free and thus $c^{\widehat \lambda}_{\mu,\widehat \nu} = c^\lambda_{\mu,\nu} \leq 1$ for all $\nu$. Otherwise, $(\widehat \lambda,\mu)$ satisfies the hypotheses of Case 1 since $(\lambda,\mu)$ satisfied the hypotheses of Case 1.
Therefore, we assume, without loss of generality, that  
\begin{equation}
\label{eq:nu1<la1}
    \nu_1<\lambda_1.
\end{equation}

\begin{claim}
\label{Claim:l(la)>l(nu)}
In the case of Theorem~\ref{thm:necessaryshort3} (\RomanNumeralCaps{1}) and (\RomanNumeralCaps{2}), 
\[\ell(\nu)<\rho+r_1(\lambda/\mu)+r_2(\lambda/\mu)\leq \ell(\lambda).\]
\end{claim}

\noindent \emph{Proof of Claim~\ref{Claim:l(la)>l(nu)}:} In the case of Theorem~\ref{thm:necessaryshort3} (\RomanNumeralCaps{1}), we have $\rho = l_1+l_2-k_1$. Since $\mu$ has shortness at least $2$ and $r_1(\lambda/\mu)+r_2(\lambda/\mu) = 2$, we obtain 
\[\rho+r_1(\lambda/\mu)+r_2(\lambda/\mu)\leq \ell(\lambda)-k_1+2 \leq \ell(\lambda).\] 
In the case of Theorem~\ref{thm:necessaryshort3} (\RomanNumeralCaps{2}), we have $\rho = l_1$. Since $\lambda$ has shortness at least $3$ and $r_1(\lambda/\mu)+r_2(\lambda/\mu) = 2$, we get
\[\rho+r_1(\lambda/\mu)+r_2(\lambda/\mu)\leq l_1+3 \leq \ell(\lambda).\]
\qed

\begin{claim}
\label{Claim:lastrowofT}
Let $T$ be any ballot tableau of shape $\lambda/\nu$ and content $\mu$ where both $\lambda$ and $\mu$ are fat hooks. If $\ell(\nu)<\ell(\lambda)$ and $\lambda_2 = \mu_1$, then the bottom row of $T$ contains $\mu_1-\mu_2$ boxes filled by $k_1$ and $\mu_2$ boxes filled by $k_1+k_2$.
\end{claim}
\noindent \emph{Proof of Claim~\ref{Claim:lastrowofT}:} Since $\mu$ is a fat hook, by Lemma~\ref{Lemma:rowofT}, $T(\ell(\lambda),c) \in \{k_1,k_1+k_2\}$ for all $c\in [\mu_1]$. Let $z$ be the column index of the left-most $k_1+k_2$ in $T$. By Lemma~\ref{Lemma:basicballot}, 
\begin{center}
$ |\{ x : x < z\text{ and column $x$ contains $k_1$} \}| \leq \mu_1-\mu_2 $.
\end{center}
Therefore $z\leq \mu_1-\mu_2+1$ and $T(\ell(\lambda),c) = k_1+k_2$ for all $c\in[z,\mu_1]$. Since 
\[|\{(r,c):T(r,c) = k_1+k_2\}| = \mu_2,\]
we get $z\geq \mu_1-\mu_2+1$ hence $z = \mu_1-\mu_2+1$.
We can then conclude that $T(\ell(\lambda),c) = k_1$ for $c\in[\mu_1-\mu_2]$ and $T(\ell(\lambda),c) = k_1+k_2$ for $c\in[\mu_1-\mu_2+1,\mu_1]$.
\qed

We divide into three subcases.

\noindent \emph{Case 1.1 $(k_1+k_2 = l_1)$:}  By (\ref{eq:nu1<la1}), we know that ${\sf CS}_{\lambda_1}(\lambda / \nu) = k_1+k_2$. As a result, for any ballot tableaux $T$ of shape $\lambda/\nu$ and content $\mu$, we have $T(l_1,\lambda_1) = k_1+k_2$. 

By Claim~\ref{Claim:l(la)>l(nu)}, we know that any triple of partitions $(\lambda,\mu,\nu)$ in Theorem~\ref{thm:necessaryshort3} (\RomanNumeralCaps{1}) and (\RomanNumeralCaps{2}) satisfies the conditions in Claim~\ref{Claim:lastrowofT}. Therefore by Claim~\ref{Claim:lastrowofT}, $T(r,c) = k_1+k_2$ only if $r = \ell(\lambda)$ for any ballot tableau $T$ of shape $\lambda/\nu$ and content $\mu$. This contradicts $T(l_1,\lambda_1) = k_1+k_2$. Therefore no such ballot tableau exists and $c^{\lambda}_{\mu,\nu} = 0$.

\begin{figure}
    \begin{center}
    \begin{subfigure}{0.45\textwidth}
        \begin{tikzpicture}[scale = 1.3]
        \draw[black, thick] (0,0) -- (4,0) -- (4,4) -- (0,4) -- (0,0);
        \draw[black,thick] (4,2) -- (4.3,2);
        \draw[black,thick] (4,4) -- (4.3,4);
        \draw[black,thick] (4,0) -- (4.3,0);
        \draw[arrows=->, line width = .7pt] (4.2,4) -- (4.2,3.3);
        \draw[arrows=->, line width = .7pt] (4.2,2) -- (4.2,2.7);
        \draw[arrows=->, line width = .7pt] (4.2,0) -- (4.2,0.7);
        \draw[arrows=->, line width = .7pt] (4.2,2) -- (4.2,1.3);
        \draw[gray,dashed] (0,1.7) -- (3,1.7);
        \draw[black,thick] (0,0)--(-0.3,0);
        \draw[black,thick] (0,1.7)--(-0.3,1.7);
        \draw[arrows=->, line width = .7pt] (-0.2,0) -- (-0.2,0.6);
        \draw[arrows=->, line width = .7pt] (-0.2,1.7) -- (-0.2,1.2);
        \draw (-0.3,0.9) node{$k_1$-$1$};
        \draw (4.2,3) node{$l_1$};
        \draw (4.2,1) node{$l_2$};
        \draw[black, thick] (3,0) -- (3,2) -- (4,2);
        \draw (1.5,-0.3) node{$\lambda_2 = \mu_1$};
        \draw[arrows=->, line width = .7pt] (0,-0.3) -- (1,-0.3);
        \draw[arrows=->, line width = .7pt] (3,-0.3) -- (2,-0.3);
        \draw[black,thick] (0,-0.2)--(0,-0.4);
        \draw[black,thick] (3,-0.2)--(3,-0.4);
        
        \draw (3.5, 1) node{$\lambda^c$};
        
        \draw[black, thick] (0,1.7) -- (0.5,1.7) -- (0.5,2.3) -- (1.4,2.3) -- (1.4,3.2) -- (2.8,3.2) -- (2.8,3.6) -- (3.6,3.6) -- (3.6,4);
        
        \draw (0.7,3.2) node{$\nu$};
        \end{tikzpicture}
        \caption{Case 1.1}
    \label{Case1a}
    \end{subfigure}
    \hspace{1cm}
    \begin{subfigure}{0.45\textwidth}
        \begin{tikzpicture}[scale=1.3]
        \draw[black, thick] (0,0) -- (4,0) -- (4,4) -- (0,4) -- (0,0);
        \draw[black,thick] (4,2) -- (4.3,2);
        \draw[black,thick] (4,4) -- (4.3,4);
        \draw[black,thick] (4,0) -- (4.3,0);
        \draw[arrows=->, line width = .7pt] (4.2,4) -- (4.2,3.3);
        \draw[arrows=->, line width = .7pt] (4.2,2) -- (4.2,2.7);
        \draw[arrows=->, line width = .7pt] (4.2,0) -- (4.2,0.7);
        \draw[arrows=->, line width = .7pt] (4.2,2) -- (4.2,1.3);
        \draw[gray,dashed] (0,1.7) -- (3,1.7);
        \draw[black,thick] (0,0)--(-0.3,0);
        \draw[black,thick] (0,1.7)--(-0.3,1.7);
        \draw[arrows=->, line width = .7pt] (-0.2,0) -- (-0.2,0.6);
        \draw[arrows=->, line width = .7pt] (-0.2,1.7) -- (-0.2,1.2);
        \draw (-0.3,0.9) node{$l_2$-$2$};
        \draw (4.2,3) node{$l_1$};
        \draw (4.2,1) node{$l_2$};
        \draw[black, thick] (3,0) -- (3,2) -- (4,2);
        \draw (1.5,-0.3) node{$\lambda_2 = \mu_1$};
        \draw[arrows=->, line width = .7pt] (0,-0.3) -- (1,-0.3);
        \draw[arrows=->, line width = .7pt] (3,-0.3) -- (2,-0.3);
        \draw[black,thick] (0,-0.2)--(0,-0.4);
        \draw[black,thick] (3,-0.2)--(3,-0.4);
        
        \draw (3.5, 1) node{$\lambda^c$};
        
        \draw[black, thick] (0,1.7) -- (0.5,1.7) -- (0.5,2.3) -- (1.4,2.3) -- (1.4,3.2) -- (2.8,3.2) -- (2.8,3.6) -- (3.6,3.6) -- (3.6,4);
        
        \draw (0.7,3.2) node{$\nu$};
        \end{tikzpicture}
        \caption{Case 1.2}
        \label{Case1b}
    \end{subfigure}
    \caption{}
    \label{Case1necessary}    
    \end{center}
\end{figure}

\noindent \emph{Case 1.2 $(k_1+k_2 > l_1, l_2>k_1)$:} In this case we had $l_1<l_1+l_2-k_1=\rho$ and $s_{\lambda/\mu}(x_1,\ldots,x_{\rho+2})$ is not multiplicity-free. Let $\nu\subset \lambda$ be any partition such that $\ell(\nu) \leq \rho+1$. In this case
\begin{equation}
\label{eq:case1.2rho}
\rho+1 = l_1+l_2-k_1 + 1.
\end{equation}
Therefore $\nu$ lies in the region above the gray dashed line in Figure~\ref{Case1a}. By Claim~\ref{Claim:l(la)>l(nu)}, there is at least one row below the gray dashed line. Moreover, since $l_2 < k_1$ all rows in $\lambda/\nu$ below the gray line in Figure~\ref{Case1a} have $\lambda_2$ boxes. 

Let $T$ be a ballot tableau of shape $\lambda/\nu$ and content $\mu$. Then $T(l_1,\lambda_1)\neq k_1+k_2$ since $k_1+k_2 > l_1$. Define $\lambda^{(1)}/ \nu^{(1)} = (\lambda / \nu)^{{\sf del}(1)}$. By Claim~\ref{Claim:l(la)>l(nu)}, $\lambda^{(1)} = (\lambda_1^{l_1},\lambda_2^{l_2-1})$ and $\nu^{(1)} = \nu$. Thus, by Claim~\ref{Claim:lastrowofT}, $T^{{\sf del}(1)}$ is a ballot tableau of shape $\lambda^{(1)}/\nu$ and content $\mu^{(1)} = (\mu_1^{k_1-1},\mu_2^{k_2})$.

By Corollary~\ref{Cor:rmvlastrow}, we have
\begin{equation}
\label{eq:firstiterationbound}
c^{\lambda}_{\mu,\nu} \leq c^{\lambda^{(1)}}_{\mu^{(1)},\nu}
\end{equation}
We may continue to iterate the above process, constructing $\lambda^{(i+1)}$ and $\mu^{(i+1)}$ from $\lambda^{(i)}$ and $\mu^{(i)}$, so long as the hypotheses of Claim~\ref{Claim:lastrowofT} and Case 1.2 remain satisfied for $\lambda^{(i)}$ and $\mu^{(i)}$. The Case 1.2 hypotheses are satisfied if $\ell(\mu^{(i)}) > l_1$. If $\ell(\lambda^{(i)}) > \rho + 1$, then $\ell(\lambda^{(i)}) \geq \ell(\nu)$ since, in this case, $\rho + 1 \geq \ell(\nu)$. Further, by \eqref{eq:case1.2rho}, if $\ell(\lambda^{(i)}) > \rho + 1$, then $\ell(\lambda^{(i)}) > l_1$ and $\lambda^{(i)}$ is a fat hook. Thus, the hypotheses of Claim~\ref{Claim:lastrowofT} and Case 1.2 are satisfied if 
\begin{equation}\label{eqn:condition1}
    \mu^{(i)}\text{ is a fat hook, }\ell(\mu^{(i)}) > l_1
\end{equation}
and 
\begin{equation}\label{eqn:condition1a}
    \ell(\lambda^{(i)}) > \rho+1 
\end{equation}
If \eqref{eqn:condition1} and \eqref{eqn:condition1a} both hold, then Corollary~\ref{Cor:rmvlastrow} implies
\begin{equation}\label{eqn:increase}
     c^{\lambda^{(i)}}_{\mu^{(i)},\nu} \leq c^{\lambda^{(i+1)}}_{\mu^{(i+1)},\nu}
\end{equation}
Let $m$ be the minimal index where (\ref{eqn:condition1}) or (\ref{eqn:condition1a}) is violated. At least one of the three following statements holds: 
\[\ell(\lambda^{(m)}) = \rho+1, \mu^{(m)} = (\mu_2^{k_2}),\ \text{and } \ell(\mu^{(m)}) = l_1,\]

\noindent \emph{Case 1.2.1 ($\ell(\lambda^{(m)}) = \rho+1$):} Here $m = k_1-1$ since $\ell(\lambda^{(i+1)}) = \ell(\lambda^{(i)})-1$ and \eqref{eq:case1.2rho}. Therefore we have $\mu^{(m)} = (\mu_1,\mu_2^{k_2})$. Since $\lambda^{(m)\vee} = ((\lambda_1-\lambda_2)^{l_2-k_1+1})$ is a rectangle and $\mu^{(m)}$ is a fat hook of shortness $1$, by Theorem~\ref{thm:TYGUT} the skew Schur function $s_{\lambda^{(m)}/\mu^{(m)}}$ is multiplicity-free. Thus $c^{\lambda^{(m)}}_{\mu^{(m)},\nu} \leq 1$. Now, by \eqref{eq:firstiterationbound} and \eqref{eqn:increase}, we get
\[c^{\lambda}_{\mu,\nu} \leq c^{\lambda^{(1)}}_{\mu^{(1)},\nu}\leq c^{\lambda^{(m)}}_{\mu^{(m)},\nu}\leq 1\text{ for all }\nu\text{ such that }\ell(\nu)\leq \rho+1.\]

\noindent \emph{Case 1.2.2 ($\mu^{(m)} = (\mu_2^{k_2})$):} Here $m = k_1 > k_1-1$. This violates the minimality of $m$, since if $m = k_1-1$, then \eqref{eq:case1.2rho} implies $\ell(\lambda^{(m)}) = \rho+1$. Hence this case is not possible.

\noindent \emph{Case 1.2.3 ($\ell(\mu^{(m)}) = l_1$):} If $T$ is a ballot tableau $T$ of shape $\lambda^{(m)}/\nu$ and content $\mu^{(m)}$, then $T(l_1,\lambda_1) = k_1+k_2-m$. We assume the two previous cases do not hold, since if they did we could apply their arguments to get our desired result. Thus we assume $\mu^{(m)}$ is a fat hook, and $\ell(\lambda)^{(m)} < \rho+1$. Combining this with  $(\lambda^{(m)})_2 = (\mu^{(m)})_1$ we conclude via Claim~\ref{Claim:lastrowofT} that $T$ can not exist. Thus \eqref{eq:firstiterationbound} and \eqref{eqn:increase} imply
\[c^{\lambda}_{\mu,\nu} \leq c^{\lambda^{(1)}}_{\mu^{(1)},\nu}\leq c^{\lambda^{(m)}}_{\mu^{(m)},\nu}=0\text{ for all }\nu\text{ such that }\ell(\nu)\leq \rho+1.\]

Combining all subcases, we get ${\sf{m}}(\lambda/\mu) = \rho+2$ in Theorem~\ref{thm:necessaryshort3} (\RomanNumeralCaps{1}) as desired.

\noindent \emph{Case 1.3 $(k_1+k_2 > l_1, l_2 \leq k_1)$:} Here we have $l_1+l_2-k_1 \leq l_1 = \rho$ and $s_{\lambda/\mu}(x_1,\ldots,x_{\rho+3})$ is not multiplicity-free. Let $\nu\subset \lambda$ be any partition such that $\ell(\nu)\leq \rho+2$. In this case, we have
\begin{equation}
\label{eq:case1.3rho}
\rho+2 = l_1+2.
\end{equation}
Once again, $\nu$ lies in the region above the gray dashed line in Figure~\ref{Case1b}. By Claim~\ref{Claim:l(la)>l(nu)}, there is at least one row below the gray dashed line. Since $l_2-2 < l_2$, all rows in $\lambda/\nu$ below the gray line have $\lambda_2$ boxes.

Let $T$ be a ballot tableau of shape $\lambda/\nu$ and content $\mu$. Then $T(l_1,\lambda_1)\neq k_1+k_2$ since $k_1+k_2 > l_1$. As in Case 1.2, define $\lambda^{(1)}/ \nu^{(1)} = (\lambda / \nu)^{{\sf del}(1)}$. By Claim~\ref{Claim:l(la)>l(nu)}, $\lambda^{(1)} = (\lambda_1^{l_1},\lambda_2^{l_2-1})$ and $\nu^{(1)} = \nu$. Thus, Claim~\ref{Claim:lastrowofT} implies $T^{{\sf del}(1)}$ is a ballot tableau of shape $\lambda^{(1)}/\nu$ and content $\mu^{(1)} = (\mu_1^{k_1-1},\mu_2^{k_2})$.

By Corollary~\ref{Cor:rmvlastrow}, we have
\begin{equation}
\label{eq:firstiterationbound1.3}
c^{\lambda}_{\mu,\nu} \leq c^{\lambda^{(1)}}_{\mu^{(1)},\nu}
\end{equation}

We may continue to iterate the above process, constructing $\lambda^{(i+1)}$ and $\mu^{(i+1)}$ from $\lambda^{(i)}$ and $\mu^{(i)}$, so long as the hypotheses of Claim~\ref{Claim:lastrowofT} and Case 1.3 remain satisfied for $\lambda^{(i)}$ and $\mu^{(i)}$. By a similar argument to Case 1.2, these hypotheses are satisfied if
\begin{equation}\label{eqn:condition11}
    \mu^{(i)}\text{ is a fat hook, }\ell(\mu^{(i)}) > l_1
\end{equation}
and 
\begin{equation}\label{eqn:condition1b}
    \ell(\lambda^{(i)}) > \rho+2.
\end{equation}

If \eqref{eqn:condition11} and \eqref{eqn:condition1b} are satisfied, then Corollary~\ref{Cor:rmvlastrow} implies
\begin{equation}\label{eqn:increase1.3}
     c^{\lambda^{(i)}}_{\mu^{(i)},\nu} \leq c^{\lambda^{(i+1)}}_{\mu^{(i+1)},\nu}
\end{equation}
Let $m$ be the minimal index where (\ref{eqn:condition11}) or (\ref{eqn:condition1b}) is violated. At least one of the three following statements holds: 
\[\ell(\lambda^{(m)}) = \rho+2, \mu^{(m)} = (\mu_2^{k_2})\ \text{and } \ell(\mu^{(m)}) = l_1.\]

\noindent \emph{Case 1.3.1 ($\ell(\lambda^{(m)}) = \rho+2$):} Since $\ell(\lambda^{(i+1)}) = \ell(\lambda^{(i)})-1$ and \eqref{eq:case1.3rho}, we have $m = l_2-2$. Therefore $\lambda^{(m)} = (\lambda_1^{l_1},\lambda_2^{2})$ and thus has shortness $2$. Since $\mu^{(m)}$ is either a fat hook or a rectangle by construction, the skew Schur function $s_{\lambda^{(m)}/\mu^{(m)}}$ is multiplicity-free by Theorem~\ref{thm:TYGUT}. Therefore $c^{\lambda^{(m)}}_{\mu^{(m)},\nu} \leq 1$, and so \eqref{eq:firstiterationbound1.3} and \eqref{eqn:increase1.3} imply
\[c^{\lambda}_{\mu,\nu} \leq c^{\lambda^{(1)}}_{\mu^{(1)},\nu}\leq c^{\lambda^{(m)}}_{\mu^{(m)},\nu} \leq 1\text{ for all }\nu\text{ such that }\ell(\nu)\leq \rho+2.\]

\noindent \emph{Case 1.3.2 ($\mu^{(m)} = (\mu_2^{k_2})$):} In this case $m = k_1$. By the Case 1.3 hypotheses,
\[k_1 \geq l_2 > l_2-2.\]
This violates the minimality of $m$, since if $m=l_2 - 2$, then \eqref{eq:case1.3rho} implies $\ell(\lambda^{(m)}) = \rho+2$.

\noindent \emph{Case 1.3.3 ($\ell(\mu^{(m)}) = l_1$):} By the same argument as in Case 1.2.3, we have $c^{\lambda}_{\mu,\nu} = 0$. 

Thus ${\sf{m}}(\lambda/\mu) = \rho+3$ in Theorem~\ref{thm:necessaryshort3} (\RomanNumeralCaps{2}) as desired.

\noindent \emph{Case 2 $(\lambda_2 = \mu_2$, $k_1+k_2=l_1)$:} We follow the notation of Theorem~\ref{thm:suffshort3} Case 4. By Lemma~\ref{Lemma:switchingcol}, we can obtain $\tilde\lambda/\tilde\mu$ with $\tilde \lambda = (\tilde\lambda_1^{k_2},\tilde\lambda_2^{k_1},\tilde\lambda_3^{l_2})$ and $\tilde \mu = (\tilde\mu_1^{l_1})$ where $\tilde \mu_1 = \tilde \lambda_3 = \mu_2$ as shown in Figure~\ref{Figure:case2tilde}. In addition we have
\[c^{\lambda}_{\mu,\nu} = c^{\tilde \lambda}_{\tilde \mu,\nu}\text{ for all partitions }\nu,\]
and
\begin{equation}\label{eqn:Case4tilde}
    {\sf{m}}(\lambda/\mu) = {\sf{m}}(\tilde\lambda/\tilde\mu).
\end{equation}

\begin{figure}
    \begin{center}
    \begin{subfigure}{0.35\textwidth}
    \begin{tikzpicture}[scale = 1.3]
        \draw[black, thick] (0,0) -- (4,0) -- (4,4) -- (0,4) -- (0,0);
        \draw[black, thick] (0,1) -- (1.5,1) -- (1.5,2.5) -- (3,2.5) -- (3,4);
        \draw[black, thick] (1.5,0) -- (1.5,1) -- (4,1);
        \draw[gray, dashed] (3,1) -- (3,2.5);
        \draw (0.75,3) node{$\mu$};
        \draw (2.75, 0.5) node{$\lambda^c$};
        
        \draw (3.5,2) node{\footnotesize $l_1$};
        \draw (2.25,1.75) node{\footnotesize $k_2$};
        \draw (0.6,0.5) node{\footnotesize $l_2$};

    \end{tikzpicture}
    \caption{Original Case 2}
    \label{Case4original}
    \end{subfigure}
    \hspace{2cm}
    \begin{subfigure}{0.35\textwidth}
        \begin{tikzpicture}[scale = 1.3]
        \draw[black, thick] (0,0) -- (4,0) -- (4,4) -- (0,4) -- (0,0);
        \draw[black, thick] (0,1) -- (1.5,1) -- (1.5,4) ;
        \draw[black, thick] (1.5,0) -- (1.5,1) -- (2.5,1) -- (2.5,2.5) -- (4,2.5);
        \draw[gray, dashed] (2.5,2.5) -- (2.5,4);
        \draw (0.75,3) node{$\tilde \mu$};
        \draw (2.75, 0.5) node{$\tilde \lambda^c$};

        \draw (2,2) node{\footnotesize $l_1$};
        \draw (3.25,3.25) node{\footnotesize $k_2$};
        \draw (0.6,0.5) node{\footnotesize $l_2$};

    \end{tikzpicture}
    \caption{New Case 2}
    \label{Case4new}
    \end{subfigure}
    
\end{center}
    \caption{Case 2}
    \label{Figure:case2tilde}
\end{figure}

\begin{figure}
    \centering
    \begin{tikzpicture}[scale = 1.3]
        \draw[black, thick] (0,0) -- (4,0) -- (4,4) -- (0,4) -- (0,0);
        \draw[black, thick] (0,2.2) -- (1.5,2.2) -- (1.5,4) ;
        \draw[black, thick] (1.5,0) -- (1.5,2.2) -- (2.5,2.2) -- (2.5,2.7) -- (4,2.7);
        \draw (0.75,3) node{$\tilde \mu$};
        \draw (2.75, 1) node{$\tilde \lambda^c$};
        
        \draw[black,thick] (4,2.7) -- (4.3,2.7);
        \draw[black,thick] (4,4) -- (4.3,4);
        \draw[arrows=->, line width = .7pt] (4.2,4) -- (4.2,3.6);
        \draw[arrows=->, line width = .7pt] (4.2,2.7) -- (4.2,3.1);
        
        \draw[black,thick] (0,2.2) -- (-0.3,2.2);
        \draw[black,thick] (0,4) -- (-0.3,4);
        \draw[arrows=->, line width = .7pt] (-0.2,4) -- (-0.2,3.35);
        \draw[arrows=->, line width = .7pt] (-0.2,2.2) -- (-0.2,2.85);
        
        \draw[black,thick] (0,0) -- (-0.3,0);
        \draw[arrows=->, line width = .7pt] (-0.2,0) -- (-0.2,0.85);
        \draw[arrows=->, line width = .7pt] (-0.2,2.2) -- (-0.2,1.35);
        
        \draw (-0.2,3.1) node{\footnotesize $l_1$};
        \draw (4.2,3.35) node{\footnotesize $k_2$};
        \draw (-0.2,1.1) node{\footnotesize $l_2$};
        
        \draw[gray,dashed] (0,1.6) -- (1.5,1.6);
        \draw[black,thick] (1.5,1.6) -- (1.9,1.6);
        \draw[arrows=->, line width = .7pt] (1.75,1.6) -- (1.75,1.05);
        \draw[arrows=->, line width = .7pt] (1.75,0) -- (1.75,0.55);
        \draw (1.75,0.8) node{\footnotesize $l_1$-$2$};
        
    \end{tikzpicture}
    \caption{Case 2.1}
    \label{Case4.1}
\end{figure}
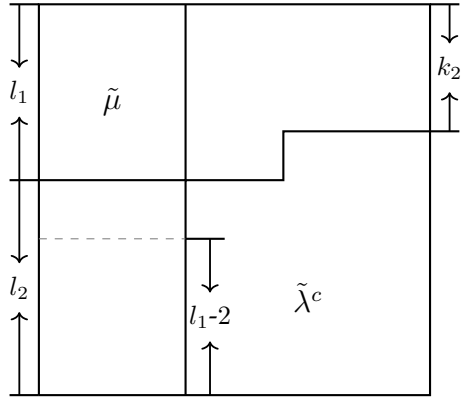

\noindent \emph{Case 2.1 ($\rho = l_2\geq l_1$):} By Theorem~\ref{thm:suffshort3}, we know that ${\sf{m}}(\lambda/\mu)\leq l_2+3$. Let $\nu \subset \tilde\lambda$ with $\ell(\nu) < l_2+3$. Then $\nu$ lies above the gray line in Figure~\ref{Case4.1}. Since $\lambda$ has shortness at least $3$, we know $l_1-2\geq 2$ and thus the bottom row of $\tilde\lambda/\nu$ has size $\tilde\lambda_3$. Since $\tilde\mu$ is a rectangle, by Lemma~\ref{Lemma:rowofT}, \[T(\ell(\tilde\lambda),c_1) = T(\ell(\tilde\lambda),c_2)\text{ for all } 1\leq c_1,c_2\leq \tilde\lambda_3.\]
for any ballot tableau $T$ of shape $\tilde\lambda/\nu$ and content $\tilde\mu$. By Lemma~\ref{Lemma:basicballot}, 
\[min\{ z : \text{column $z$ contains $l_1$}\}=1.\]
Therefore
\begin{equation}
\label{eq:bottomrowallthesame}
T(\ell(\tilde\lambda),c) = l_1 \text{ for all }c\in[\tilde\lambda_3].
\end{equation}

Define $\tilde \lambda^{(1)}/ \nu^{(1)} = (\tilde \lambda / \nu)^{{\sf del}(1)}$. Then $\tilde\lambda^{(1)} = (\tilde\lambda_1^{k_2},\tilde\lambda_2^{k_1},\tilde\lambda_3^{l_2-1})$ and $\nu^{(1)} = \nu$. Thus, \eqref{eq:bottomrowallthesame} implies $T^{{\sf del}(1)}$ is a ballot tableau of shape $\lambda^{(1)}/\nu$ and content $\tilde \mu^{(1)} = (\tilde\mu_1^{l_1-1})$. By Corollary~\ref{Cor:rmvlastrow}
\[c^{\tilde\lambda}_{\tilde\mu,\nu} \leq c^{\tilde\lambda^{(1)}}_{\tilde\mu^{(1)},\nu}.\]
By repeatedly applying the above argument, we can conclude that the last $l_1-2$ rows of $\tilde\lambda/\nu$ are filled by $\{3,\ldots,l_1\}$ where every row is filled by exactly one number. Set $\tilde \lambda^{(l_1-2)} / \tilde \nu^{(l_1-2)} = (\tilde \lambda / \nu)^{{\sf del}(l_1-2)}$ and set $\tilde\mu^{(l_1-2)} = (\tilde \mu_1^2)$. By Corollary~\ref{Cor:rmvlastrow}
\[c^{\tilde\lambda}_{\tilde\mu,\nu}\leq c^{\tilde \lambda^{(l_1-2)}}_{\tilde \mu^{(l_1-2)},\nu}.\]
Since $\tilde\mu^{(l_1-2)}$ is a rectangle of shortness two and $\tilde \lambda^{(l_1-2)}$ is a fat hook, we know by Theorem~\ref{thm:TYGUT} that the skew Schur function $s_{\tilde\lambda^{(l_1-2)}/\tilde\mu^{(l_1-2)}}$ is multiplicity-free and thus
\[c^{\tilde \lambda^{(l_1-2)}}_{\tilde \mu^{(l_1-2)},\nu} \leq 1.\]
Therefore 
\[c^{\tilde\lambda}_{\tilde\mu,\nu} \leq c^{\tilde \lambda^{(l_1-2)}}_{\tilde \mu^{(l_1-2)},\nu} \leq 1 \text{ for all }\nu\text{ such that }\ell(\nu)\leq \rho+2.\]
Thus $m(\lambda/\mu) = m(\tilde\lambda/\tilde\mu) = \rho+3$ as required in Theorem~\ref{thm:necessaryshort3} (\RomanNumeralCaps{5}).

\begin{figure}
    \centering
    \begin{tikzpicture}[scale = 1.3]
        \draw[black, thick] (0,0) -- (4,0) -- (4,4) -- (0,4) -- (0,0);
        \draw[black, thick] (0,1.5) -- (1.5,1.5) -- (1.5,4) ; 
        \draw[black, thick] (1.5,0) -- (1.5,1.5) -- (2.5,1.5) -- (2.5,2.7) -- (4,2.7);
        \draw (0.75,2.5) node{$\tilde \mu$};
        \draw (2.85, 1) node{$\tilde \lambda^c$};
        
        \draw[black,thick] (4,2.7) -- (4.3,2.7);
        \draw[black,thick] (4,4) -- (4.3,4);
        \draw[arrows=->, line width = .7pt] (4.2,4) -- (4.2,3.6);
        \draw[arrows=->, line width = .7pt] (4.2,2.7) -- (4.2,3.1);
        
        \draw[black,thick] (0,1.5) -- (-0.3,1.5);
        \draw[black,thick] (0,4) -- (-0.3,4);
        \draw[arrows=->, line width = .7pt] (-0.2,4) -- (-0.2,3);
        \draw[arrows=->, line width = .7pt] (-0.2,1.5) -- (-0.2,2.5);
        
        \draw[black,thick] (0,0) -- (-0.3,0);
        \draw[arrows=->, line width = .7pt] (-0.2,0) -- (-0.2,0.5);
        \draw[arrows=->, line width = .7pt] (-0.2,1.5) -- (-0.2,1);
        
        \draw (-0.2,2.75) node{\footnotesize $l_1$};
        \draw (4.2,3.35) node{\footnotesize $k_2$};
        \draw (-0.2,0.75) node{\footnotesize $l_2$};
        
        \draw[gray,dashed] (0,1.3) -- (1.5,1.3);
        \draw[black,thick] (1.5,1.3) -- (1.9,1.3);
        \draw[arrows=->, line width = .7pt] (1.75,1.3) -- (1.75,0.8);
        \draw[arrows=->, line width = .7pt] (1.75,0) -- (1.75,0.3);
        \draw (1.75,0.55) node{\footnotesize $l_2$-$1$};
        
    \end{tikzpicture}
    \caption{Case 2.2}
    \label{Case4.2}
\end{figure}
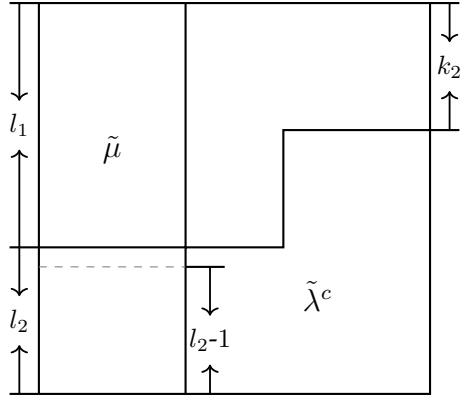

\noindent \emph{Case 2.2 ($\rho=l_1>l_2$):} Let $\nu \subset \tilde\lambda$ with $\ell(\nu) < l_1 + 1$. Then $\nu$ lies in the region above the gray line in Figure~\ref{Case4.2}. Since $\lambda$ has shortness at least $3$, we have $l_2-1\geq 2$ and thus the bottom row of $\tilde\lambda/\nu$ has size $\tilde\lambda_3$. By a similar argument as in Case 2.1, we know that the last $l_2-1$ rows of any ballot tableau $T$ of shape $\tilde \lambda/\nu$ and content $\tilde\mu$ must be filled by $\{l_1-l_2+2,\ldots,l_1\}$. In addition, every row is filled by the same value. By Corollary~\ref{Cor:rmvlastrow}, we conclude 
\[c^{\tilde\lambda}_{\tilde\mu,\nu}\leq c^{\tilde \lambda^{(l_2-1)}}_{\tilde \mu^{(l_2-1)},\nu}.\]
Since $\tilde \lambda^{(l_2-1)}$ is a fat hook of shortness $1$ and $\tilde \mu^{(l_2-1)}$ is a rectangle, Theorem~\ref{thm:TYGUT} implies that $s_{\tilde \lambda^{(l_2-1)/\tilde \mu^{(l_2-1)}}}$ is multiplicity-free. Therefore,
\[c^{\tilde \lambda^{(l_2-1)}}_{\tilde \mu^{(l_2-1)},\nu} \leq 1,\]
and thus
\[c^{\tilde\lambda}_{\tilde\mu,\nu} = c^{\tilde \lambda^{(0)}}_{\tilde \mu^{(0)},\nu}\leq c^{\tilde \lambda^{(l_2-1)}}_{\tilde \mu^{(l_2-1)},\nu} \leq 1 \text{ for all }\nu\text{ such that }\ell(\nu)\leq \rho+1.\]
By (\ref{eqn:Case4tilde}) we conclude that ${\sf{m}}(\lambda/\mu) = \rho+2$ in Theorem~\ref{thm:necessaryshort3} (\RomanNumeralCaps{6}).

\begin{figure}
    \begin{center}
    \begin{subfigure}{0.35\textwidth}
    \begin{tikzpicture}[scale = 1.3]
        \draw[black, thick] (0,0) -- (4,0) -- (4,4) -- (0,4) -- (0,0);
        \draw[black, thick] (0,1) -- (1.5,1) -- (1.5,2.5) -- (3,2.5) -- (3,4);
        \draw[black, thick] (2.5,0) -- (2.5,1) -- (4,1);
        \draw[gray, dashed] (3,2.5) -- (3,1);
        \draw[gray, dashed] (1.5,1) -- (2.5,1);
        \draw (0.75,3) node{$\mu$};
        \draw (3.3, 0.5) node{$\lambda^c$};
        \path[draw,fill = gray!20] (0,0) -- (2.5,0) -- (2.5,1) -- (0,1) -- cycle;
		\draw[black,thick] (4,4) -- (4.4,4);
		\draw[black,thick] (4,1) -- (4.4,1);
        \draw[arrows=->, line width = .7pt] (4.2,4) -- (4.2,2.8);
        \draw[arrows=->, line width = .7pt] (4.2,1) -- (4.2,2.3);
        \draw (4.2,2.55) node{\footnotesize $l_1$};
		\draw[arrows=->, line width = .7pt] (2.8,4) -- (2.8,3.5);
        \draw[arrows=->, line width = .7pt] (2.8,2.5) -- (2.8,3);
        \draw (2.8,3.25) node{\footnotesize $k_1$};
		\draw[black,thick] (0,0) -- (-0.4,0);
		\draw[black,thick] (0,1) -- (-0.4,1);
        \draw[arrows=->, line width = .7pt] (-0.2,0) -- (-0.2,0.25);
        \draw[arrows=->, line width = .7pt] (-0.2,1) -- (-0.2,0.75);
		\draw (-0.2,0.5) node{\footnotesize $l_2$};

    \end{tikzpicture}
    \caption{Original Case 3}
    \label{Case2original}
    \end{subfigure}
    \hspace{1.5cm}
    \begin{subfigure}{0.45\textwidth}
       \begin{tikzpicture}[scale = 1.3]
        \draw[black, thick] (0,0) -- (4,0) -- (4,4) -- (0,4);
        \draw[gray,dashed] (0,4)-- (0,0);
        \draw[black, thick] (-1,1) -- (1.5,1) -- (1.5,2.5) -- (3,2.5) -- (3,4);
        \draw[black, thick] (2.5,1) -- (4,1);
        \draw[gray,dashed] (2.5,0) -- (2.5,1);
        \draw[gray, dashed] (3,2.5) -- (3,1);
        \draw[black, thick] (1.5,0) -- (1.5,1) -- (2.5,1);
        \draw (0.5,2.7) node{$\tilde\mu$};
        \draw (3, 0.5) node{$\tilde\lambda^c$};
        \draw[black, thick] (0,0) -- (-1,0) -- (-1,4) -- (0,4);
        \path[draw, fill = gray!20] (-1,0) -- (1.5,0) -- (1.5,1) -- (-1,1) -- cycle;
		\draw[black,thick] (4,4) -- (4.4,4);
		\draw[black,thick] (4,1) -- (4.4,1);
        \draw[arrows=->, line width = .7pt] (4.2,4) -- (4.2,2.8);
        \draw[arrows=->, line width = .7pt] (4.2,1) -- (4.2,2.3);
        \draw (4.2,2.55) node{\footnotesize $l_1$};
		\draw[arrows=->, line width = .7pt] (2.8,4) -- (2.8,3.5);
        \draw[arrows=->, line width = .7pt] (2.8,2.5) -- (2.8,3);
        \draw (2.8,3.25) node{\footnotesize $k_1$};
		\draw[black,thick] (-1,0) -- (-1.4,0);
		\draw[black,thick] (-1,1) -- (-1.4,1);
        \draw[arrows=->, line width = .7pt] (-1.2,0) -- (-1.2,0.25);
        \draw[arrows=->, line width = .7pt] (-1.2,1) -- (-1.2,0.75);
		\draw (-1.2,0.5) node{\footnotesize $l_2$};

    \end{tikzpicture}
    \caption{New Case 3}
    \label{Case2new}
    \end{subfigure}
    
\end{center}
    \caption{Case 3}
    
\end{figure}
The remaining two cases correspond to Theorem~\ref{thm:necessaryshort3} (\RomanNumeralCaps{3}).

\noindent \emph{Case 3 ($\mu_2<\lambda_2<\mu_1$, $k_1+k_2=l_1$, $k_1\geq l_2$):} In this case $\rho = l_1$ and we interpret  $c^{\lambda}_{\mu,\nu}$ as number of ballot tableaux of shape $\lambda/\mu$ and content $\nu$.

Let $T_0,T_1$ be two distinct ballot tableaux of shape $\lambda/\mu$ with content $\nu$ such that $\ell(\nu) = {\sf{m}}(\lambda/\mu)$. We can shift all the rows in $T_0,T_1$ that are below the bottom row of $\mu$ to the left so that the new shape satisfies the hypotheses of Theorem~\ref{thm:necessaryshort3} Case 2. This shift is visualized in Figure~\ref{Case2original} and Figure~\ref{Case2new}. Let $\overline T_0,\overline T_1$ be the resulting tableaux of shape $\overline\lambda/\overline\mu$, as in Figure~\ref{Case2new}. For $i\in \{0,1\}$, $\overline T_i$ remains semistandard and the row reading word of $T_i$ is the same as $\overline T_i$. As a result, $\overline T_0,\overline T_1$ are two ballot tableaux whose contents are both $\nu$ as well. Therefore, by the definition of ${\sf{m}}(\lambda/\mu)$, 
\begin{equation}
\label{eq:firstcase3}
{\sf{m}}(\lambda/\mu)\geq {\sf{m}}(\overline\lambda/\overline\mu).
\end{equation}
Since $\overline\lambda/\overline\mu$ satisfies the hypotheses of Theorem~\ref{thm:necessaryshort3} Case 2, we know that 
\begin{equation}
\label{eq:secondcase3}
{\sf{m}}(\overline\lambda/\overline\mu) = \rho(\overline\lambda/\overline\mu) +2.
\end{equation}
Since we have $k_1\geq l_2$, we have
\begin{equation}
\label{eq:thirdcase3}
\rho(\overline\lambda/\overline\mu) = \rho(\lambda/\mu).
\end{equation}
Combining \eqref{eq:firstcase3},\eqref{eq:secondcase3}, and \eqref{eq:thirdcase3} we get 
\[{\sf{m}}(\lambda/\mu)\geq \rho(\lambda/\mu)+2.\]
By Theorem~\ref{thm:suffshort3} ${\sf{m}}(\lambda/\mu)\leq \rho(\lambda/\mu)+2$, and so we conclude that ${\sf{m}}(\lambda/\mu) = \rho+2$.

\begin{figure}
    \begin{center}
    \begin{subfigure}{0.35\textwidth}
    \begin{tikzpicture}[scale = 1.3]
        \draw[black, thick] (0,0) -- (4,0) -- (4,4) -- (0,4) -- (0,0);
        \draw[black, thick] (0,1) -- (1.5,1) -- (1.5,2) -- (3,2) -- (3,4);
        \draw[black, thick] (2,0) -- (2,1.5) -- (4,1.5);
        \draw[gray, dashed] (2,1.5) -- (2,2);
        \draw[gray, dashed] (1.5,0) -- (1.5,1);
        \draw[gray, dashed] (3,2) -- (3,1.5);
        \draw (0.75,3) node{$\mu$};
        \draw (3, 0.75) node{$\lambda^c$};

    \end{tikzpicture}
    \caption{Original Case 4}
    \label{Case3original}
    \end{subfigure}
    \hspace{2cm}
    \begin{subfigure}{0.4\textwidth}
        \begin{tikzpicture}[scale = 1.3]
        \draw[black, thick] (0,0) -- (4,0) -- (4,4) -- (0,4) -- (0,0);
        \draw[black, thick] (0,2.5) -- (2,2.5) -- (2,4) ;
        \draw[black, thick] (4,3) -- (2.5,3) -- (2.5,2) -- (1,2) -- (1,0);
        \draw (1,3.25) node{$\lambda^\vee$};
        \draw (2.75, 1) node{$rotate (\mu)$};
        
        \draw[gray, dashed] (0,1.4) -- (1,1.4);
        \draw[black,thick] (0,1.4) -- (-0.3,1.4);
        \draw[black,thick] (0,4) -- (-0.3,4);
        \draw[arrows=->, line width = .7pt] (-0.15,4) -- (-0.15,2.95);
        \draw[arrows=->, line width = .7pt] (-0.15,1.4) -- (-0.15,2.45);
        \draw (-0.23,2.7) node{\footnotesize$l_1$+$1$};

	   \draw[arrows=->, line width = .7pt] (1.15,2) -- (1.15,1.2);
        \draw[arrows=->, line width = .7pt] (1.15,0) -- (1.15,0.8);
		\draw (1.15,1) node{\footnotesize$k_1$};

        \draw[black,thick] (0,0) -- (0,-0.3);
        \draw[black,thick] (1,0) -- (1,-0.3);
        \draw[arrows=->, line width = .7pt] (0,-0.15) -- (0.15,-0.15);
        \draw[arrows=->, line width = .7pt] (1,-0.15) -- (0.85,-0.15);
        \draw (0.5,-0.15) node{\footnotesize$\lambda_1$-$\mu_1$};

    \end{tikzpicture}
    \caption{New Case 4}
    \label{Case3new}
    \end{subfigure}
    
\end{center}
    \caption{Case 4}
    
\end{figure}

\noindent \emph{Case 4 ($\mu_2<\lambda_2<\mu_1$, $k_1+k_2 > l_1$, $k_1\geq l_2$):} In this case $\rho = l_1$. By \eqref{eq:LRsymmetry-transpose}, it suffices to show
\[c^{\mu^\vee}_{\lambda^\vee,\nu} \leq 1 \text{ for all }\nu\text{ such that }\ell(\nu)\leq \rho+1.\]

Let $\nu\subset \mu^\vee$ with $\ell(\nu) < l_1$. We show there is at most one ballot tableau of shape $\mu^\vee / \nu$ and content $\lambda^\vee$. Any such $\nu$ lies above the gray dashed line in Figure~\ref{Case3new}. We have
\[\lambda^\vee = ((\lambda_1-\lambda_2)^{l_2})\ \text{and }\mu^\vee = (\lambda_1^{l_1+l_2-k_1-k_2}, (\lambda_1-\mu_2)^{k_2},(\lambda_1-\mu_1)^{k_1}).\]
We write $\nu$ as
\[\nu = (\nu_1,\nu_2,\ldots,\nu_{\ell(\nu)}).\]
Let $T$ be a ballot tableau of shape $\mu^\vee / \nu$ and content $\lambda^\vee$. Since $\rho=l_1$, there are $l_2-1$ rows below the gray dashed line in Figure~\ref{Case3new}. Since $\lambda^\vee$ is a rectangle, Lemma~\ref{Lemma:rowofT} implies
\[T(\ell(\lambda),c) = l_2 \text{ for all } c\in[\lambda_1-\mu_1].\]
Thus the second row from the bottom of $T$ cannot contain any entry equal to $l_2$. Again by Lemma~\ref{Lemma:rowofT}, the row must be entirely filled by $l_2-1$. Since $k_1\geq l_2>l_2-1$, we know that all rows below the gray line have the same size. We can then apply Lemma~\ref{Lemma:rowofT} $l_2-1$ times and conclude that the last $l_2-1$ rows of $T$ are filled by $\{2,\ldots,l_2\}$. Set
\[\lambda^{(l_2-1)} = ((\lambda_1-\lambda_2),(\mu_1-\lambda_2)^{l_2-1})\ \text{and }\mu^{(l_2-1)} = (\lambda_1^{l_1+l_2-k_1-k_2}, (\lambda_1-\mu_2)^{k_2},(\lambda_1-\mu_1)^{k_1-l_2+1}).\]
Since $k_1\geq l_2$, we have 
\begin{equation}\label{eqn:Case3}
    k_1-l_2+1> 0
\end{equation}
and $\mu^{(l_2-1)}$ is a partition.
By Corollary~\ref{Cor:rmvlastrow}, we have
\begin{equation}
\label{eq:Case4bound}
c^{\mu^\vee}_{\lambda^\vee,\nu}\leq c^{\mu^{(l_2-1)}}_{\lambda^{(l_2-1)},\nu}\text{ for all }\nu\text{ such that }\ell(\nu)\leq \rho+1.
\end{equation}

By (\ref{eqn:Case3}), there is least one row of size $\lambda_1-\mu_1$ in $\mu^{(l_2-1)}$. 

\noindent \emph{Case 4.1 ($\nu_{l_1+1} = \lambda_1-\mu_1$):}  Since $\ell(\mu^{(l_2-1)}) = l_1+1$, the first $\lambda_1-\mu_1$ columns in $\mu^{(l_2-1)}/\nu$ are empty. Let $\widetilde \mu^{(l_2-1)}/\widetilde \nu$ be the skew diagram that arises from deleting the first $\lambda_1-\mu_1$ columns in $\mu^{(l_2-1)}/\nu$. Now $\widetilde \mu^{(l_2-1)}$, $\widetilde \nu$, and $\lambda^{(l_2-1)}$ are all contained in $(\mu_1^{l_1+l_2-k_1})$. Fix $(\mu_1^{l_1+l_2-k_1})$ as our ambient rectangle. Then $\widetilde \mu^{(l_2-1)\vee}$ is a rectangle and $\lambda^{(l_2-1)}$ is a fat hook of shortness $1$. Notice by construction that $\lambda^{(l_2-1)}_{i}< \tilde\mu^{(l_2-1)}_i$ for all $i\in [l_1+l_2-k_1]$. Thus the skew partition $\tilde\mu^{(l_2-1)}/\lambda^{(l_2-1)}$ is basic unless $\lambda^{(l_2-1)}_1> \tilde\mu^{(l_2-1)}_2$. If the skew partition is not basic, the basic demolition will remove column $\tilde\mu^{(l_2-1)}_2+1$ through $\lambda^{(l_2-1)}_1$. Thus $(\lambda^{(l_2-1)})^{\sf{ba}}$ would be a fat hook of shortness $1$ and $((\tilde\mu^{(l_2-1)})^{\sf{ba}})^\vee$ would be a rectangle. Therefore by Theorem~\ref{thm:TYGUT}, the skew Schur function $s_{\widetilde \mu^{(l_2-1)}/\lambda^{(l_2-1)}}$ is multiplicity-free. As a result, by \eqref{eq:Case4bound},
\[c^{\mu^\vee}_{\lambda^\vee,\nu}\leq c^{\mu^{(l_2-1)}}_{\lambda^{(l_2-1)},\nu} = c^{\widetilde \mu^{(l_2-1)}}_{\lambda^{(l_2-1)},\tilde \nu} \leq 1,\]

\noindent \emph{Case 4.2 ($\nu_{l_1+1} < \lambda_1-\mu_1$):} Since for any ballot tableau $T$ of shape $\mu^\vee/\nu$ and content $\lambda^\vee$, $T(l_1+2,c) = 2$ for all $c\in [\lambda_1-\mu_1]$, we know that $T(l_1+1,c) = 1$ for all $c\in [\nu_{l_1+1}+1,\lambda_1-\mu_1]$. Therefore if there exists a box in the bottom row of $\mu^{(l_2-1)}$ such that the box above it is not in $\nu$, then any filling will violate semistandardness. As a result
\begin{equation}
\label{eq:Case4.2}
    \nu_{l_1}\geq \lambda_1-\mu_1.
\end{equation}
Let $x = (\lambda_1-\mu_1)-\nu_{l_1+1}$ be the number of entries in the bottom row of $\mu^{(l_2-1)}/\nu$. Remove the bottom row of $\mu^{(l_2-1)}$ and its content as well as any box in the $(l_1+1)$\textsuperscript{th} row of $\nu$ to obtain
\[\lambda^{(l_2)} = ((\lambda_1-\lambda_2-x),(\mu_1-\lambda_2)^{l_2-1})\]
\[\mu^{(l_2)} = (\lambda_1^{l_1+l_2-k_1-k_2}, (\lambda_1-\mu_2)^{k_2},(\lambda_1-\mu_1)^{k_1-l_2}),\]
\[\nu^{(1)} = (\nu_1,\nu_2,\ldots,\nu_{l_1}).\]
By Corollary~\ref{Cor:rmvlastrow}, we obtain
\[c^{\mu^\vee}_{\lambda^\vee,\nu}\leq c^{\mu^{(l_2)}}_{\lambda^{(l_2)},\nu^{(1)}}.\]
By (\ref{eq:Case4.2}), the first $\lambda_1-\mu_1$ columns in $\mu^{(l_2)}/\nu^{(1)}$ are empty. Let $\tilde\mu^{(l_2)}/\tilde\nu^{(1)}$ be the skew partition that arises from deleting the first $\lambda_1-\mu_1$ columns in $\mu^{(l_2)}/\nu^{(1)}$.
Then $\tilde\mu^{(l_2)}$ lies in the rectangle $(\mu_1^{l_1+l_2-k_1})$. Fix $(\mu_1^{l_1+l_2-k_1})$ as the ambient rectangle. Then $\tilde\mu^{(l_2)\vee}$ is a rectangle and $\lambda^{(l_2)}$ is either a fat hook of shortness $1$ or a rectangle. Notice by construction that $\lambda^{(l_2)}_{i}< \tilde\mu^{(l_2)}_i$ for all $i\in [l_1+l_2-k_1]$. The skew partition $\tilde\mu^{(l_2)}/\lambda^{(l_2)}$ is basic unless $\lambda^{(l_2)}_1> \tilde\mu^{(l_2)}_2$. If the skew partition is not basic, the basic demolition will remove column $\tilde\mu^{(l_2)}_2+1$ through $\lambda^{(l_2)}_1$. Thus $(\lambda^{(l_2)})^{\sf{ba}}$ would be a fat hook of shortness $1$ and $((\tilde\mu^{(l_2)})^{\sf{ba}})^\vee$ would be a rectangle. 
Therefore, by Theorem~\ref{thm:TYGUT}, the skew Schur function $s_{\mu^{(l_2)}/\lambda^{(l_2)}}$ is multiplicity-free. As a result,
\[c^{\mu^\vee}_{\lambda^\vee,\nu} \leq c^{ \mu^{(l_2)}}_{\lambda^{(l_2)},\nu^{(1)}} \leq 1.\]
We conclude
\[c^{\mu^\vee}_{\lambda^\vee,\nu} \leq 1 \text{ for all }\nu\text{ such that }\ell(\nu)\leq \rho+1.\]
Combining Case 3 with Case 4, we have ${\sf{m}}(\lambda/\mu) = \rho+2$ in Theorem~\ref{thm:necessaryshort3} (\RomanNumeralCaps{3}).
\end{proof}

\subsection{\texorpdfstring{$\lambda^\vee$}{} is a rectangle of shortness at least \texorpdfstring{$2$}{} and \texorpdfstring{$\sf{np}(\mu) > 2$}{}}

\begin{theorem}
\label{thm:necessarynotfathook}
Let $s_{\lambda/\mu}(x_1,\ldots,x_n)$ be a basic, n-sharp, tight skew Schur polynomial such that $\lambda^\vee$ is a rectangle of shortness at least $2$ and $\mu$ is not a fat hook, then the following are true:
\begin{enumerate}[label=(\Roman*)]
    \item If $\lambda_2 = \mu_q, l_2\geq l_1$ then $\sf{m}(\lambda/\mu) = \rho+2$.
    \item If $\lambda_2 = \mu_q, l_2< l_1$ then $\sf{m}(\lambda/\mu) = \rho+1$
    \item If $\lambda_2 = \mu_1, k_1\geq l_2$ then $\sf{m}(\lambda/\mu) = \rho+2$.
    \item If $\lambda_2 = \mu_1, k_1< l_2$ then $\sf{m}(\lambda/\mu) = \rho+1$
    \item If $\mu_q<\lambda_2<\mu_1$ then $\sf{m}(\lambda/\mu) = \rho+1$
\end{enumerate}
\end{theorem}
\begin{proof}
Similar to Theorem~\ref{thm:necessaryshort3}, it is enough to show (\ref{eq:equivtoshow-ns3}) holds in each of the five cases. Note that (\ref{eq:equivtoshow-ns3}) holds trivially for Theorem~\ref{thm:necessarynotfathook} (\RomanNumeralCaps{2}), (\RomanNumeralCaps{4}) and (\RomanNumeralCaps{5}) since $\lambda/\mu$ is $n$-sharp. 

\noindent \emph{Case 1 ($\lambda_2 = \mu_q, l_2\geq l_1$):} By Lemma~\ref{Lemma:switchingcol}, we can obtain $\tilde\lambda/\tilde\mu$ by setting
\[\tilde\lambda = (\tilde\lambda_1^{k_q}, \ldots,\tilde\lambda_q^{k_1},\tilde\lambda_{q+1}^{l_2})\text{ and }\tilde\mu = (\tilde\mu_1^{l_1}),\]
where $\tilde\lambda_i = \lambda_1+\mu_q-\mu_{q-i+1}$ for $i\in [q]$ and $\tilde\lambda_{q+1} = \tilde\mu_1 = \lambda_2$. This is visualized in Figure~\ref{case1aoriginal} and Figure~\ref{case1anew}. We then have
\[c^{\lambda}_{\mu,\nu} = c^{\tilde\lambda}_{\tilde\mu,\nu}\text{ for all partitions }\nu\]
Therefore we may equivalently to show that $s_{\tilde\lambda/\tilde\mu}(x_1,\ldots,x_{\rho+1})$ is multiplicity-free.

    \begin{figure}
    \begin{center}
    \begin{subfigure}{0.35\textwidth}
    \begin{tikzpicture}[scale=1.3]
        \draw[black, thick] (0,0) -- (4,0) -- (4,4) -- (0,4) -- (0,0);
        \draw[black, thick] (0,2.2) -- (1,2.2) -- (1,2.6) -- (2,2.6) -- (2,3) -- (2.7,3) -- (2.7,3.4) -- (3.4,3.4) -- (3.4,4);
        \draw[black, thick] (1,0) -- (1,2.2) -- (4,2.2);
        \draw (1,3.2) node{$\mu$};
        \draw (2.4, 1.1) node{$\lambda^c$};
        
        \draw[gray, dashed] (2,2.2) -- (2,2.6);
        \draw[gray, dashed] (2.7,2.2) -- (2.7,3);
        \draw[gray, dashed] (3.4,2.2) -- (3.4,3.4);
        \end{tikzpicture}
    \caption{Original diagram}
    \label{case1aoriginal}
    \end{subfigure}
    \hspace{2cm}
    \begin{subfigure}{0.35\textwidth}
        \begin{tikzpicture}[scale = 1.3]
        \draw[black, thick] (0,0) -- (4,0) -- (4,4) -- (0,4) -- (0,0);
        \draw[black, thick] (0,2.2) -- (1,2.2) -- (1,4);
        \draw[black, thick] (1,0) -- (1,2.2) -- (1.6,2.2) -- (1.6, 2.8) -- (2.3, 2.8) -- (2.3,3.2) -- (3,3.2) -- (3,3.6) -- (4,3.6);
        \draw (0.5,3.2) node{$\tilde\mu$};
        \draw (2.4, 1.1) node{$\tilde\lambda^c$};

        \draw[gray, dashed] (1.6,2.8) -- (1.6,4);
        \draw[gray, dashed] (2.3,3.2) -- (2.3,4);
        \draw[gray, dashed] (3,3.6) -- (3,4);
        \draw[gray, dashed] (0,2) -- (1,2);
        
        \draw[black, thick] (0,2) -- (-0.4,2);
        \draw[black, thick] (0,4) -- (-0.4,4);

        \draw[arrows=->, line width = .7pt] (-0.2,4) -- (-0.2,3.25);
        \draw[arrows=->, line width = .7pt] (-0.2,2) -- (-0.2,2.75);
        \draw (-0.27,3) node{$l_2$+$1$};
        
    \end{tikzpicture}
    \caption{New diagram}
    \label{case1anew}
    \end{subfigure}
    
\end{center}
    \caption{Theorem~\ref{thm:necessarynotfathook} (\RomanNumeralCaps{1})}
    
\end{figure}

Since $\rho = l_2$, by (\ref{eq:LRsymmetrybasic}), it is equivalent to show that there is at most one ballot tableau of shape $\tilde\lambda/\nu$ and content $\tilde\mu$ for any partition $\nu$ such that $\ell(\nu)\leq l_2+1$. Any such $\nu$ lies entirely in the region above the gray dashed line in Figure~\ref{case1anew}. Since $l_2+1 > l_1$, all rows in $\tilde\lambda/\nu$ below the gray line have the same size. By Lemma~\ref{Lemma:rowofT}, 
\[T(\ell(\tilde\lambda),c) = l_1 \text{ for all }c\in[\tilde\lambda_{q+1}].\]
Define $\tilde\lambda^{(1)}/\nu^{(1)} = (\tilde\lambda/\nu)^{{\sf{del}}(1)}$. Since $\lambda$ has shortness at least $2$, $l_1-1>0$ and thus $\ell(\nu)<\ell(\tilde\lambda)$,
$\tilde\lambda^{(1)} = (\tilde\lambda_1^{k_q}, \ldots,\tilde\lambda_q^{k_1}, \tilde\lambda_{q+1}^{l_2-1}), \nu^{(1)} = \nu$ and $\tilde\mu^{(1)} = (\tilde\mu_1^{l_1-1})$.
By Corollary~\ref{Cor:rmvlastrow}, we know that
\[c^{\tilde\lambda}_{\tilde\mu,\nu}\leq c^{\tilde\lambda^{(1)}}_{\tilde\mu^{(1)},\nu}.\]
We may continue to apply Lemma~\ref{Lemma:rowofT} and Corollary~\ref{Cor:rmvlastrow}, removing the bottom row of the skew diagram, as long as the bottom row of $\tilde\lambda^{(i)}$ and $\tilde\mu^{(i)}$ are of the same size and $\ell(\tilde\lambda^{(i)})>\ell(\nu)$. We have $\ell(\tilde\lambda) - \ell(\nu)\geq l_2-1 \geq l_1-1$ and $\ell(\tilde\mu) = l_1$. Thus we may iterate this process $l_1-1$ times and get
\[\tilde\lambda^{(l_1-1)} = (\tilde\lambda_1^{k_q}, \ldots\tilde\lambda_q^{k_1}, \tilde\lambda_{q+1}^{1})\text{ and }\tilde\mu^{(l_1-1)} = (\tilde\mu_1),\]
and
\[c^{\tilde\lambda}_{\tilde\mu, \nu}\leq c^{\tilde\lambda^{(a-1)}}_{\tilde\mu^{(a-1)}, \nu}.\]
Since $\tilde\mu^{(a-1)}$ is a one row rectangle, it has shortness $1$. As a result, by Theorem~\ref{thm:TYGUT}, the skew Schur function $s_{\tilde\lambda^{(a-1)}/\tilde\mu^{(a-1)}}$ is multiplicity-free and thus 
\[c^{\tilde\lambda^{(l_1-1)}}_{\tilde\mu^{(l_1-1)}, \nu}\leq 1.\]
Thus $s_{\tilde\lambda/\tilde\mu}(x_1,\ldots,x_{\rho+1})$ is multiplicity-free and ${\sf{m}}(\lambda/\mu) = \rho+1$ in Theorem~\ref{thm:necessarynotfathook} (\RomanNumeralCaps{1}).

\noindent \emph{Case 2 ($\lambda_2 = \mu_1, k_1\geq l_2$)} Here we have $\rho = l_1$. By \eqref{eq:LRsymmetry-transpose}, it is enough to show
\[c^{\mu^\vee}_{\lambda^\vee,\nu}\leq 1\text{ for all partition }\nu\subseteq \mu^\vee \text{ such that }\ell(\nu)\leq l_1+1.\]
Since $k_1\geq l_2$, we get $l_1+1 > l_1+l_2-k_1$. Therefore all $l_2-1$ rows in $\tilde\lambda/\nu$ below the horizontal gray dashed line in Figure~\ref{case2anecessary} have the same size for any $\nu$ such that $\ell(\nu)\leq l_1+1$. 
Write
\[\lambda^\vee =  ((\lambda_1-\mu_1)^{l_2})\text{ and }\mu^\vee =  (\lambda_1^{l_1+l_2-\ell(\mu)},(\lambda_1-\mu_q)^{k_q},\ldots,(\lambda_1-\mu_1)^{k_1}).\]
By applying Lemma~\ref{Lemma:rowofT} and Corollary~\ref{Cor:rmvlastrow} $l_2-1$ times, we get
\[c^{\mu}_{\lambda,\nu} \leq c^{\mu^{\vee(l_2-1)}}_{\lambda^{\vee(l_2-1)},\nu}\]
where
\[\lambda^{\vee(l_2-1)} = (\lambda_1-\mu_1)\text{ and }\mu^{\vee(l_2-1)} = (\lambda_1^{l_1+l_2-\ell(\mu)},(\lambda_1-\mu_q)^{k_q},\ldots,(\lambda_1-\mu_1)^{k_1-l_2+1}).\]
Since $\lambda^{\vee(l_1-1)}$ has shortness $1$, the skew Schur function $s_{\mu^{\vee(l_1-1)}/\lambda^{\vee(l_1-1)}}$ is multiplicity-free and thus 
\[c^{\lambda}_{\mu,\nu} = c^{\mu^\vee}_{\lambda^\vee,\nu} = c^{\mu^{\vee(l_1-1)}}_{\lambda^{\vee(l_1-1)},\nu} \leq 1\text{ for all }\nu\text{ such that }\ell(\nu)\leq \rho+1.\]
Thus $s_{\lambda/\mu}(x_1,\ldots,x_{\rho+1})$ is multiplicity-free and ${\sf{m}}(\lambda/\mu) = \rho+2$ in Theorem~\ref{thm:necessarynotfathook} (\RomanNumeralCaps{3}).
\begin{figure}
    \begin{center}
        \begin{tikzpicture}[scale = 1.3]
        \draw[black, thick] (0,0) -- (4,0) -- (4,4) -- (0,4) -- (0,0);
        \draw[black, thick] (4,3.2) -- (3,3.2) -- (3,2.6) -- (2,2.6) -- (2,1.9) -- (1.3,1.9) -- (1.3,1.2) -- (0.6,1.2) -- (0.6,0);
        \draw[black, thick] (0.6,4) -- (0.6,3) -- (0,3);
        \draw (2.5,1) node{$\mu^c$};
        \draw (0.3, 3.5) node{$\lambda$};
        \draw[gray, dashed] (0.6,3) -- (0.6,1.2);
        \draw[gray, dashed] (0,1) -- (0.6,1);
        
        \draw (0.3,1.5) node{$l_1$};
        
        \draw[black, thick] (0,1) -- (-0.4,1);
        \draw[black, thick] (0,4) -- (-0.4,4);

        \draw[arrows=->, line width = .7pt] (-0.2,4) -- (-0.2,2.75);
        \draw[arrows=->, line width = .7pt] (-0.2,1) -- (-0.2,2.25);
        \draw (-0.27,2.5) node{$l_1$+$1$};
        
        \draw[arrows=->, line width = .7pt] (0.8,1.2) -- (0.8,0.85);
        \draw[arrows=->, line width = .7pt] (0.8,0) -- (0.8,0.35);
        \draw (0.8,0.6) node{$k_1$};
        
        \end{tikzpicture}
      \caption{Theorem~\ref{thm:necessarynotfathook} (\RomanNumeralCaps{3}) with $\mu^\vee$}
    \label{case2anecessary}  
    \end{center}

\end{figure}
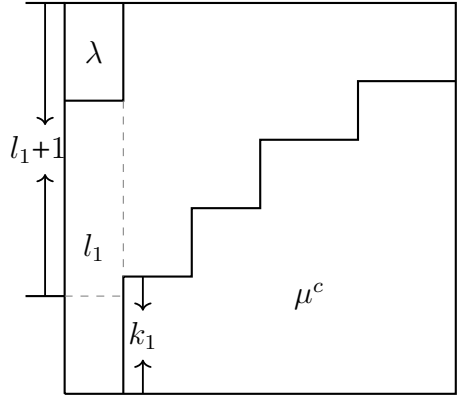
\end{proof}

\section*{Acknowledgments}
We thank Alexander Yong for helpful discussions. RH was partially supported by an AMS-Simons Travel Grant.

\end{document}